\documentclass[12pt, reqno]{amsart}

\usepackage[dvipsnames]{xcolor}
\usepackage{
    afterpage, 
    amsmath, 
    amsthm, 
    amscd, 
    amsfonts, 
    amssymb, 
    booktabs,
    enumerate,
    enumitem,
    geometry,
    graphicx, 
    lineno,
    longtable,
    makecell,
    mathtools,
    multirow,
    newfloat,
    physics, 
    placeins,
    rotating,
    tikz,
    tikz-cd,
    comment,
    }
\usepackage[bookmarksnumbered, colorlinks, plainpages]{hyperref}
\usepackage[export]{adjustbox} 
\usepackage[font=small]{caption}
\usepackage{orcidlink}
\usepackage{amsaddr}
\usepackage[normalem]{ulem}

\theoremstyle{definition}
\newtheorem{theorem}{Theorem}[section]

\newtheorem{lemma}[theorem]{Lemma}
\newtheorem{proposition}[theorem]{Proposition}

\newtheorem{corollary}[theorem]{Corollary}
\theoremstyle{definition}
\newtheorem{definition}[theorem]{Definition}
\newtheorem{example}[theorem]{Example}

\theoremstyle{remark}
\newtheorem{remark}[theorem]{Remark}
\numberwithin{equation}{section}

\newcommand{\qedex}{$\hfill\lozenge$}

\DeclareMathOperator{\coker}{coker}
\DeclareMathOperator{\coim}{coim}
\DeclareMathOperator{\dom}{dom}

\DeclareMathOperator{\im}{im}

\DeclareMathOperator{\cl}{cl}
\DeclareMathOperator{\opn}{opn}
\DeclareMathOperator{\mo}{mo}
\DeclareMathOperator{\inv}{Inv}
\DeclareMathOperator{\inter}{int}
\newcommand{\zint}[1]{[0,#1]_{\ZZ}}
\newcommand{\zintp}[1]{[1,#1]_{\ZZ}}
\newcommand{\zintab}[2]{[#1,#2]_{\ZZ}}
\newcommand{\inscr}{\sqsubseteq}
\newcommand{\ovscr}{\sqsupseteq}

\newcommand{\dgV}{G_\cV}
\newcommand{\dgVz}{G_{\cV_0}}
\newcommand{\dgVo}{G_{\cV_1}}

\newcommand{\scc}{scc}

\DeclareMathOperator{\paths}{Paths}
\newcommand{\pathsV}{\paths_\cV}
\DeclareMathOperator{\sol}{Sol}
\DeclareMathOperator{\esol}{eSol}
\newcommand{\esolV}{\esol_\cV}

\DeclareMathOperator{\isol}{iSol}
\newcommand{\solV}{\sol_\cV}
\newcommand{\isolV}{\isol_\cV}
\newcommand{\pbeg}[1]{#1^\sqsubset}
\newcommand{\pend}[1]{#1^\sqsupset}

\newcommand{\mvm}{F_\cV}

\DeclareMathOperator{\smvf}{\mathbb{MVF}}

\DeclareMathOperator{\pf}{pf}

\newcommand\mvx[1]{[#1]_\cV}

\newcommand{\tInt}{\Lambda}
\newcommand{\tdyn}{\lambda}
\newcommand{\tdynmax}{T}

\newcommand{\ipair}[1]{\left(P_{#1}, E_{#1}\right)}

\newcommand{\ipairpp}[3]{\left(N_{#1}^{#3}, N_{#2}^{#3}\right)}

\newcommand{\ipairleqp}[2]{\left(N_{#1}^{#2}, N_{0}^{#2}\right)}

\newcommand{\ipairdb}[2]{P_{#1}^{#2}\setminus E_{#1}^{#2}}
\newcommand{\ipairl}[1]{\left(P_{#1}^{\,\vdash}, E_{#1}^{\,\vdash}\right)}
\newcommand{\ipairr}[1]{\left(P_{#1}^{\,\dashv}, E_{#1}^{\,\dashv}\right)}

\DeclareMathOperator{\con}{Con}
\newcommand{\conv}[3]{
    \ifnum\numexpr#1\relax=0 {0} \else {\ZZ_2^{#1}}\fi
    \times
    \ifnum\numexpr#2\relax=0 {0} \else {\ZZ_2^{#2}}\fi
    \times
    \ifnum\numexpr#3\relax=0 {0} \else {\ZZ_2^{#3}}\fi
    }

\DeclareMathOperator{\uim}{uim}
\newcommand{\uimp}{\uim^+}
\newcommand{\uimm}{\uim^{-}}
\newcommand{\uimpm}{\uim^{\pm}}

\newcommand{\MD}{\cM}
\newcommand{\BD}{\cB}
\newcommand{\BDmd}{\BD_{\!\bullet}}
\newcommand{\BDmdt}[1]{\BD_{#1\bullet}}
\newcommand{\zzBD}{\mathfrak{B}}
\newcommand{\zzTD}{\mathfrak{TD}}
\newcommand{\zzMD}{\mathfrak{M}}
\newcommand{\zzV}{\mathfrak{V}}
\newcommand{\bl}{B}
\newcommand{\blz}[1]{B_{#1, 0}}
\newcommand{\blo}[1]{B_{#1, 1}}
\newcommand{\blt}[1]{B_{#1, 2}}

\newcommand{\cset}[1]{C_{#1}}

\newcommand{\indMD}[1]{#1_{\bullet}}
\newcommand{\indMDV}[2]{#1_{\bullet,#2}}

\newcommand{\mvfseq}[1]{\cV_0, \cV_1, \ldots, \cV_{#1}}

\newcommand{\seqof}[2]{#1_0, #1_1, \ldots, #1_{#2}}
\newcommand{\idxmap}[1]{\iota_{#1}}
\newcommand{\idxfwd}[1]{\overrightarrow{\iota}_{\!\!#1}}
\newcommand{\idxbck}[1]{\overleftarrow{\iota}_{\!\!#1}}

\newcommand{\quiver}{Q}

\newcommand{\stringset}{S}
\newcommand{\stringmodule}[1]{\mathbb{S}_{#1}}

\newcommand{\fk}{k}
\newcommand{\dg}{d}

\newcommand{\perm}[1]{\mathbb{#1}}

\DeclareMathOperator{\id}{Id}

\newcommand{\sbm}[1]{{\let\amp=&\left[\begin{smallmatrix}#1\end{smallmatrix}\right]}}
\newcommand{\poset}{\mathbb{P}}

\DeclareMathOperator{\vect}{Vect_\fk}

\def\cA{\text{$\mathcal A$}}
\def\cB{\text{$\mathcal B$}}

\def\cM{\text{$\mathcal M$}}

\def\cT{\text{$\mathcal T$}}

\def\cV{\text{$\mathcal V$}}

\def\NN{\mathbb{N}}
\def\PP{\mathbb{P}}
\def\QQ{\mathbb{Q}}
\def\RR{\mathbb{R}}
\def\SS{\mathbb{S}}

\def\ZZ{\mathbb{Z}}

\makeindex

\title[Conley-Morse persistence barcode]
    {Conley-Morse persistence barcode:\\ 
        a homological signature of combinatorial bifurcations}

\author[Tamal K. Dey, Micha\l{} Lipi\'nski, Manuel Soriano-Trigueros]
    {Tamal K. Dey$^1$\orcidlink{0000-0001-5160-9738}, 
    Micha\l{} Lipi\'nski$^{2*}$\orcidlink{0000-0001-9789-9750} and 
    Manuel Soriano-Trigueros$^{2,3}$\orcidlink{0000-0003-2449-1433}}

\address{$^{1}$ Purdue University, IN, US}
\address{$^{2}$ Institute of Science and Technology Austria (ISTA)}
\address{$^{3}$ Universidad de Sevilla, Spain}
\email{tamaldey@purdue.edu}
\email{michal.lipinski@ist.ac.at}
\email{msoriano4@us.es}

\definecolor{darkred}{rgb}{1, 0.1, 0.3}
\definecolor{darkblue}{rgb}{0.1, 0.1, 1}



\thanks{$^{*}$ Corresponding author}

\begin{document}

\setcounter{page}{1}

\maketitle
\begin{center}
    (Communicated by Peter Bubenik)
\end{center}

\begin{abstract}
Bifurcation characterizes the qualitative changes in parameterized dynamical systems and is one of the major topics in the field.
In this work, we study combinatorial bifurcations within the framework of combinatorial dynamical systems---a young but already well-established theory.
We introduce the Conley-Morse persistence barcode, a compact algebraic descriptor of combinatorial bifurcations.
This barcode captures structural changes in a dynamical system at the level of Morse decompositions and provides a characterization of the nature of observed transitions in terms of the Conley index.
The construction of the Conley-Morse persistence barcode builds upon ideas from topological persistence.
Specifically, we 
consider a persistence module obtained from the Conley index of invariant sets indexed over a poset. Using gentle algebras, we prove that this module decomposes into simple intervals (bars) and
compute them by adapting the zigzag persistence algorithm to our purpose.
\end{abstract}

\tableofcontents

\section{Introduction}
A common approach to understanding the structure of a dynamical system involves analyzing its invariant sets and connections among them,
    as these sets constitute the long-term behavior of the system.
An invariant set is called \emph{isolated} if there exists a neighborhood---known as an isolating block---that separates it from other invariant sets.
Charles Conley~\cite{Conley1978} introduced a Morse decomposition as a means of organizing isolated invariant sets into a unified structure. 
It is defined as a collection of isolated invariant sets, known as \emph{Morse sets}, such that the network of connections among them forms a partial order, thereby reflecting the system's global gradient structure. 
Notably, all recurrent behavior is encapsulated within these Morse sets.
Furthermore, each isolated invariant set can also be characterized by a homological signature, called the \emph{Conley index}.

An important property of isolated invariant sets and the Conley index is their robustness with respect to sufficiently small perturbations. 
This means that for every isolated invariant set $S$, we can find a corresponding $S'$ in the perturbed system with an isomorphic Conley index and similar dynamical behavior.
The stability extends further to the notion of \emph{continuation}, which allows for the identification of isolated invariant sets across parameter spaces.
Franzosa further generalized this idea by introducing the continuation of the Morse decompositions~\cite{Franzosa1988}.
In this paper, we build upon these concepts to study the continuation of Morse sets in a combinatorial dynamical system.

Recently Mrozek~\cite{Mrozek2017} introduced the theory of \emph{combinatorial multivector fields}, 
which was later generalized in~\cite{LKMW2022}.
Throughout this paper we work within the generalized framework.
Combinatorial multivector fields can be viewed as a combinatorial counterpart to classical continuous vector fields and admit
    combinatorial analogues of fundamental dynamical notions, including  Morse decomposition and the Conley index. 
The theory of multivector fields is based on Forman's combinatorial vector fields~\cite{Forman1998b, Forman1998a}.
These combinatorial frameworks have already proven effective in analyzing continuous dynamical systems~\cite{MrozekSrzednickiThorpeWanner2022, MrozekWanner2021,  Woukeng2024}.
Recently, the concept of continuation has also been adapted to the combinatorial setting~\cite{DeLiMrSl2022},
    introducing, in particular, the notion of \emph{combinatorial perturbation}, 
    which enables the definition of a parameterized combinatorial dynamical system.
It has also been shown there that combinatorial continuation fits naturally into the language of \emph{persistent homology}. 

We leverage these observations to capture \emph{combinatorial bifurcations} of Morse sets in terms of a persistence module, referred to as the \emph{Conley--Morse persistence module}.
The module is induced by a \emph{transition diagram}, which fuses local changes to Morse sets into a single unified structure.
The transition diagram is composed of smaller sub-diagrams called \emph{AR-split diagrams}, which capture local attractor-repeller breakdowns of combinatorial isolated invariant sets.
The Conley-Morse persistence module is a module over a poset, which is generally difficult to analyze \cite{poset_decomposition}.
However, we show that dynamical constraints cause it to fall into the family of \emph{gentle algebras} \cite{Assem1981, Assem1987}, allowing a decomposition into \emph{zigzag intervals} (or \emph{strings}).
We refer to this decomposition as the \emph{Conley-Morse persistence barcode}.
The strings represent timelines of Conley index generators, 
    allowing us to track changes to the corresponding Conley indices, 
    including redistribution, 
        mutual annihilation, or emergence of generators due to changes in the dynamics.
Moreover, building on a recent zigzag persistence algorithm~\cite{DW22},
    we provide an algorithm for computing the Conley-Morse persistence barcode.

The development of computational tools for characterizing bifurcations---or, more broadly, the evolution of a dynamical system---has been an active area of research in recent years. 
We highlight several notable directions.
A standard input in \emph{topological data analysis} is a point cloud; 
    accordingly, in~\cite{GuMuKh2022,TyMuKh2020}, the authors use zigzag persistence to detect bifurcations by analyzing time series generated by the underlying dynamics. 
Various methods have also been proposed to track critical points of a combinatorial gradient vector field 
    (e.g., those induced by a scalar field), 
    using path connectivity~\cite{King2017}, merge trees~\cite{Hotz2023}, topological robustness~\cite{trophy2024}, 
    or direct analysis of crossing values~\cite{DhaChaNat2025}.
The tracking of critical points is often motivated by the Cerf theorem~\cite{Cerf1970}, 
    which in the smooth setting guarantees the existence of a parametrization connecting two Morse functions with only finitely many degenerate critical points along the parametrization.
    A~take on Cerf theory from a multiparameter persistence perspective is presented in~\cite{Bubenik2024}. 
    Another approach to studying multiparameter discrete Morse theory is developed in~\cite{Brouillette:2024aa}, where the authors incorporate multivector fields to handle violations of the classical definition of a discrete Morse function.
    However, all these tracking approaches are rooted in Morse theory and are therefore limited to gradient systems.

More general results can be obtained by switching to Conley index theory.
For instance, 
    tools for the qualitative classification of parameter spaces based on combinatorial dynamics have been developed in \cite{Arai2009, BushGameiroHarkerKokubuEtc:2012}, 
    and a framework for describing bifurcations that combines Conley index theory and sheaf theory was introduced in~\cite{DowKalVan2023}.
In a similar spirit, our results
    can be viewed as a connection between bifurcation theory and persistence theory in the context of combinatorial dynamics.

The paper is organized as follows.
In Section~\ref{sec:main-results}, we present the motivation and general intuition behind the Conley-Morse persistence barcode. 
In Section~\ref{sec:preliminaries}, we recall some basic facts and fix the notation. 
Section~\ref{sec:multivector-fields-theory} introduces multivector fields theory and the combinatorial continuation of an isolated invariant set. 
We also establish the concept of a combinatorial isolating block and a block decomposition, 
     and present their properties.  
Section~\ref{sec:transition-diagram} contains the main construction of the paper: the transition diagram for a zigzag filtration of block decompositions, 
    which gives rise to the Conley-Morse persistence module. 
Moreover, we provide an explicit recipe for its construction.
Section~\ref{sec:gentle-algebras-persistence} recalls the necessary background from persistence theory and gentle algebras, which are fundamental for the decomposition of the Conley-Morse persistence module.
In Section~\ref{sec:cm-barcodes} we define the Conley-Morse persistence barcode as a decomposition of the persistence module induced by the transition diagram and discuss some of its properties.
Section~\ref{sec:algorithm} presents an algorithm for computing the Conley-Morse persistence barcode.
Finally, Section~\ref{sec:discussion} outlines further directions and open questions inspired by this work.
In addition, we supply the paper with two appendices: a list of symbols and  an index of concepts to simplify navigation through the paper.

\section{Motivation and Main Ideas}\label{sec:main-results}
\subsection{The sphere example}\label{subsec:sphere-example}
Even though all results presented in this paper primarily concern combinatorial multivector fields, 
    we begin with an informal reminder of some classical concepts and an example of a continuous flow.
    This is done for two reasons:
first, the developed framework is strongly inspired by the continuation theory for continuous flows, 
    and we plan to adapt the framework to that setting in a follow-up work; 
second, we believe that continuous flows are better for building an intuition, 
    even for readers already familiar with combinatorial multivector fields.

Consider a parameterized vector field $\varphi_\lambda:\RR\times\SS^2\rightarrow\SS^2$, where $\lambda\in[0,5]$,  as presented in Figure~\ref{fig:sphere-example-continuous-vf}. 
For $\lambda=0$, we have a simple dynamics with the red repelling equilibrium $R$ at the north pole and the attracting equilibrium at the south pole.
For certain $\lambda\in(0,1)$, the system undergoes a Hopf bifurcation, transforming the north pole equilibrium into an attracting equilibrium $E$ and a repelling periodic orbit $O$.
The orbit then travels south and eventually breaks into repelling equilibrium $T$ (the red point) and the saddle $S$ (the yellow point) (see Figure~\ref{fig:sphere-example-continuous-vf} for $\lambda=3$).
Finally, for a $\lambda\in(4,5)$, the saddle and the southern attractor collide, annihilating each other.

The diagram in Figure~\ref{fig:CM-barcode-sphere_example} presents the so-called \emph{Conley--Morse persistence barcode} capturing the evolution of the vector field on the sphere in terms of Conley indices.
Prior to explaining the diagram, we recall a few standard definitions.

\begin{figure}[t]
    \centering
    \includegraphics[width=0.32\linewidth]{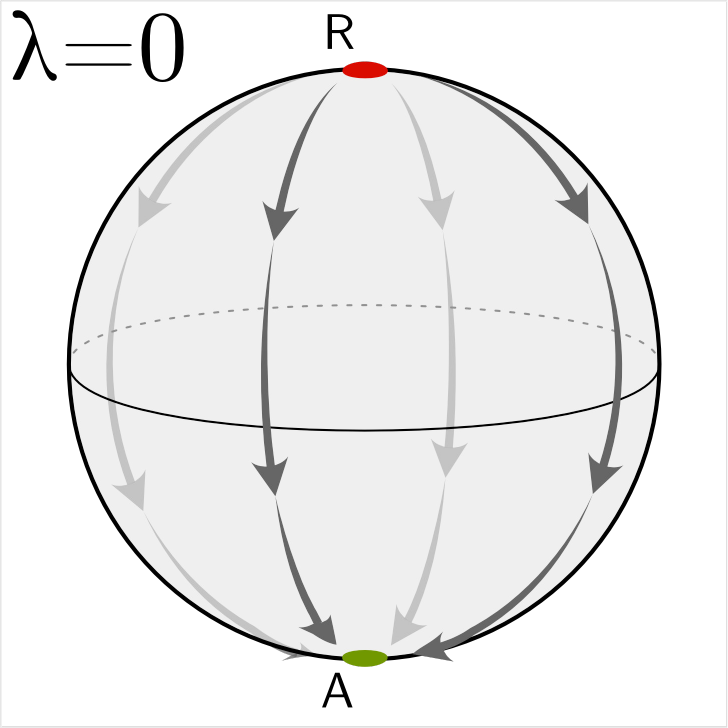}
    \includegraphics[width=0.32\linewidth]{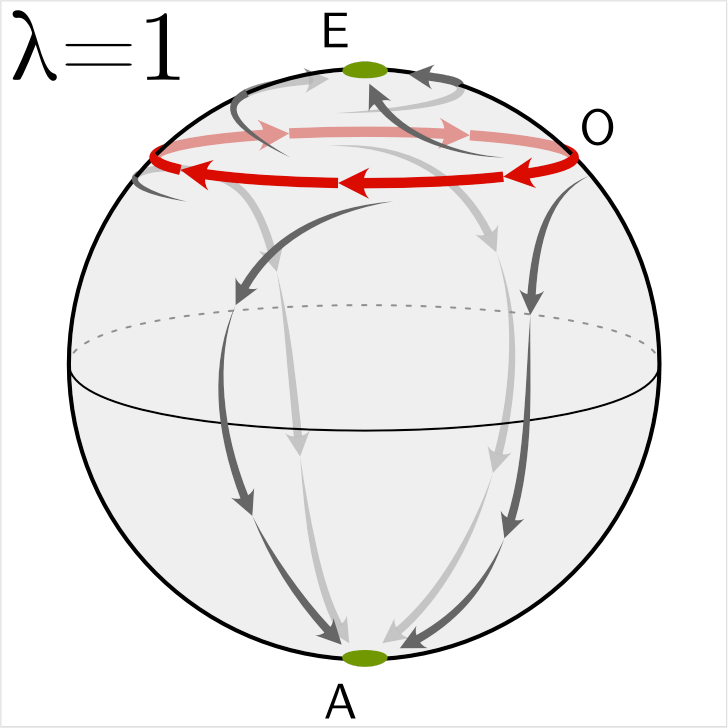}
    \includegraphics[width=0.32\linewidth]{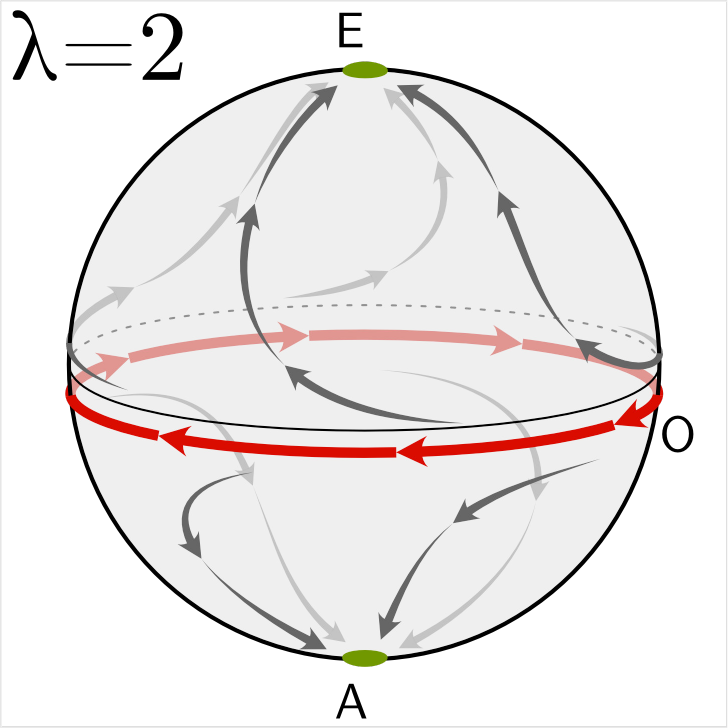}\\
    
    \vspace{1.2mm}
    \includegraphics[width=0.32\linewidth]{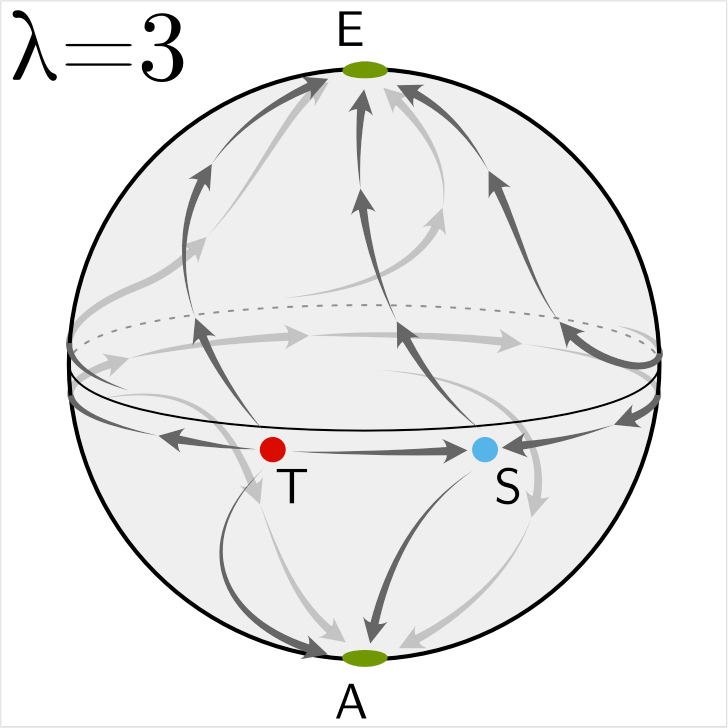}
    \includegraphics[width=0.32\linewidth]{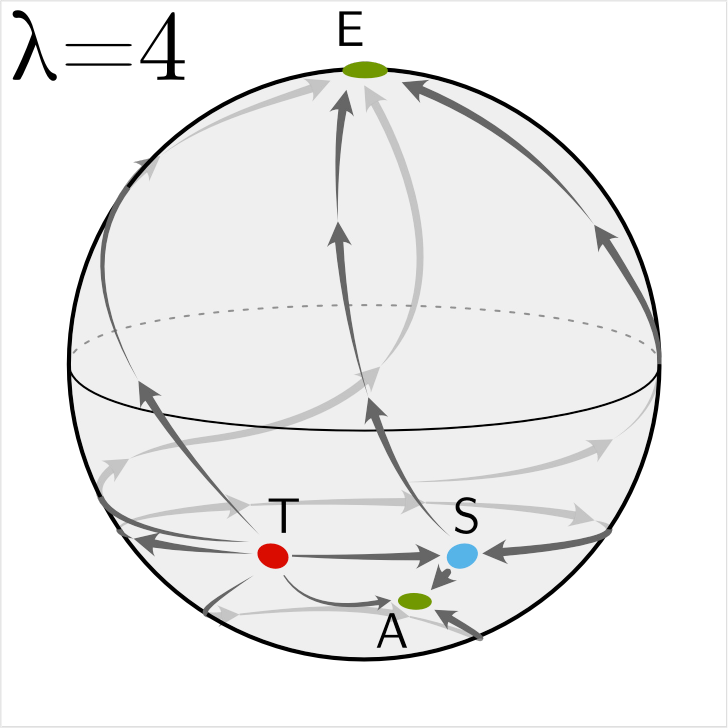}
    \includegraphics[width=0.32\linewidth]{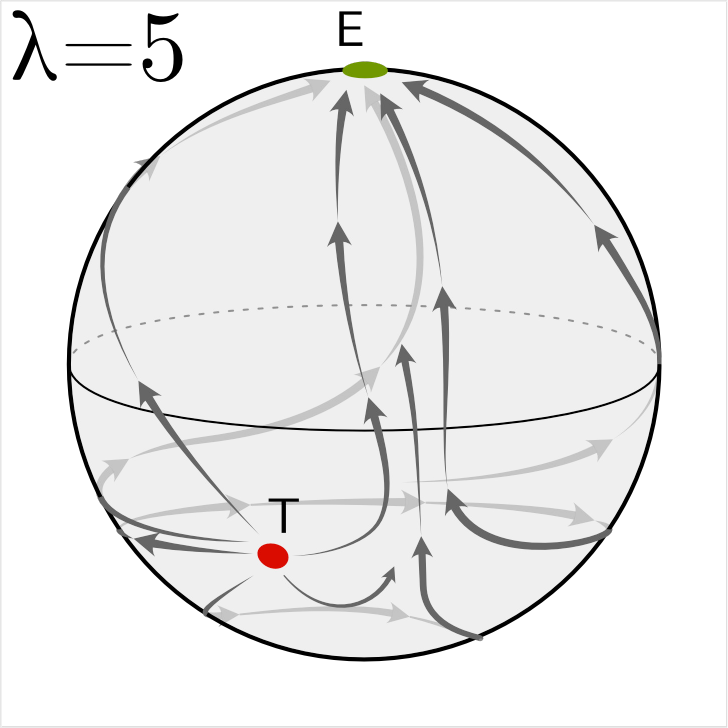}
    \caption{A parameterized flow on a 2-sphere. 
        }
    \label{fig:sphere-example-continuous-vf}
    \centering
    \vspace{0.45cm}
    \includegraphics[width=0.95\textwidth]{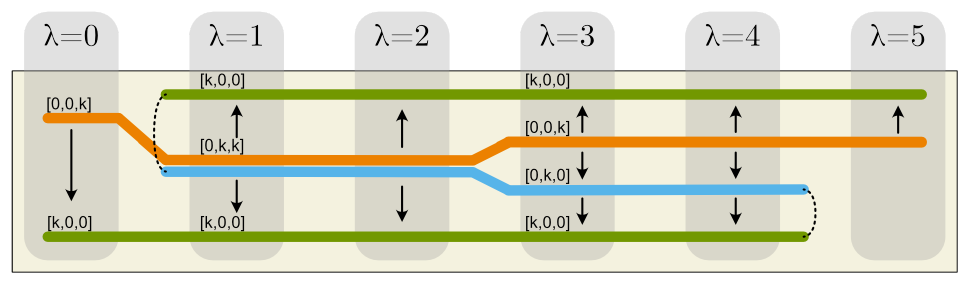}
    \caption{Conley-Morse persistence barcode corresponding to the parameterized vector field in Figure~\ref{fig:sphere-example-continuous-vf}.}
    \label{fig:CM-barcode-sphere_example}
\end{figure}

\begin{figure}[t]
    \centering
    \includegraphics[width=0.25\linewidth]{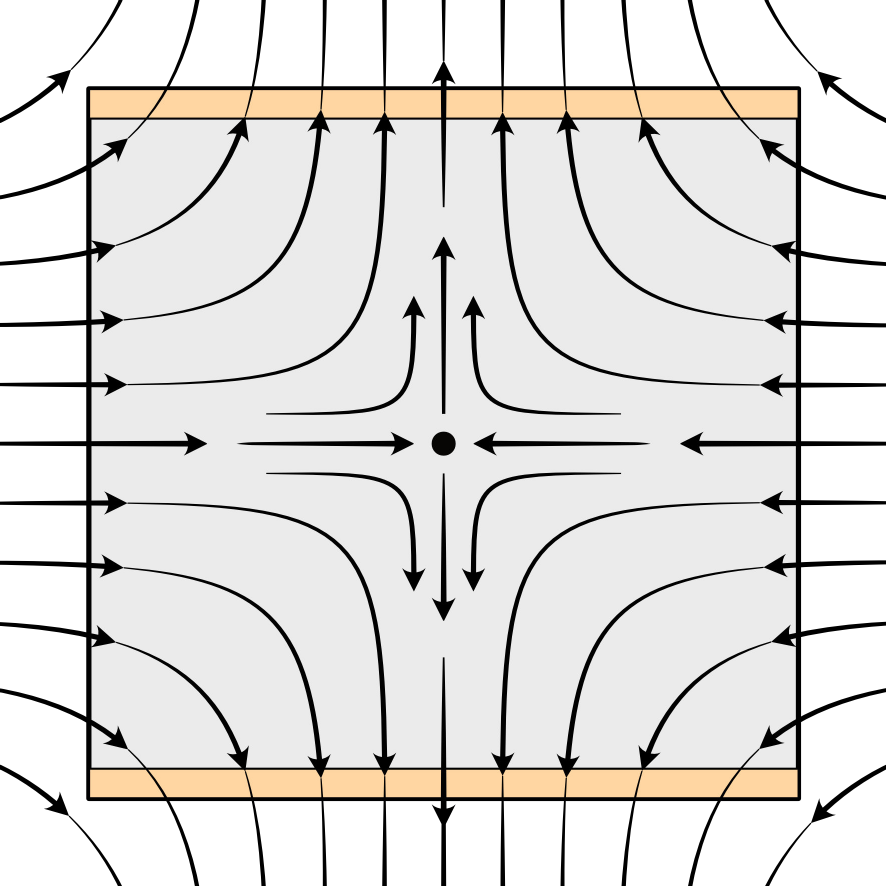}
    \caption{The gray region is an isolating block for the saddle point in the center and 
        the orange segments represent its exit part.}
    \label{fig:index-pair-example}
\end{figure}

Given a vector field $\varphi$, an \emph{invariant part} of a set $S\subset\RR^n$ is $\inv S\coloneqq \{x\in S\mid \varphi(\RR,x)\subset S\}$.
In particular, $S$ is \emph{invariant} if $\inv S=S$.
We say that a set is an \emph{isolated invariant set} if there exists a compact neighborhood $N$ of $S$ such that $S=\inv N\subset \inter N$.
The equilibria and the periodic orbit in our example are instances of isolated invariant sets.
A stronger notion than an isolating neighborhood is an \emph{isolating block}\index{isolating block} defined as a compact set such that its \emph{exit set} (that is the part of the boundary through which the flow escapes) given by
\begin{equation*}
    B^-\coloneqq  \left\{x\in B\mid \varphi\left([0,T),x\right)\not\subset B, \forall T>0\right\},
\end{equation*}
is closed 
and for all $T>0$ we have
\begin{equation*}
    \inv_T(B, \varphi)\subset \inter B,
\end{equation*}
where
\begin{equation*}
    \inv_{T}(B)\coloneqq \{x\in B\mid \varphi([-T,T], x)\subset B\}.
\end{equation*}
In particular, every isolated invariant set $S$ admits an isolating block $B$ such that $\inv B=S$.
An example of an isolating block (the gray region) for a saddle point with the exit set marked in orange is depicted in Figure~\ref{fig:index-pair-example}.
Using an isolating block, we can compute the homological Conley index, 
    given by
    $\con(S) \coloneqq[H_{0}(B,B^-)$, $H_{1}(B,B^-)$, $H_{2}(B,B^-)]$,
    where $H_{d}(B,B^-)$ represents the relative (singular) homology of degree $d$ calculated over the field $k=\ZZ_2$.
In particular, for attracting equilibria, $A$ and $E$, repelling equilibria, $R$ and $T$, saddle $S$ and repelling periodic orbit $O$ in the example, we have 
    $\con(A)=\con(E)=[\fk,0,0]$, 
    $\con(R)=\con(T)=[0,0,\fk]$, 
    $\con(S)=[0,\fk,0]$, and
    $\con(O)=[0,\fk,\fk]$. 
Note that each type of an isolated invariant set in this example admits a different Conley index. 
We refer to \cite{MischMro_Conley_2002} for a brief introduction to Conley index theory.

We say that an isolated invariant set $S$ in $\varphi_\lambda$ \emph{continues} to $S'$ in $\varphi_{\lambda'}$ if there exists a set $B$, 
    which is an isolating block in $\varphi_{{\tau}}$ for all ${\tau}\in[\lambda,\lambda']$,
    $\inv_{\varphi_{\lambda}}B=S$ and $\inv_{\varphi_{\lambda'}}B=S'$.
It follows that the Conley index is preserved through a continuation;
    however, the structure of the invariant set isolated by $B$ may change significantly.

Consider again the first step in our example.
The Hopf bifurcation at the north pole turns the repelling equilibrium $R$ into attracting equilibrium $E$ and periodic orbit $O$.
With a proper isolating block around the north pole one can show that $R$ continues into the union of $E$, $O$, and the trajectories connecting them, which together form an isolated invariant disc, which we denote by $D$ (compare the invariant part of the set $N_2$ for $\lambda=0$ and $\lambda=1$ in Figure~\ref{fig:hopf-example}).
Note that such a disc behaves globally as a repeller; 
    in particular, its Conley index is the same as that of $R$, that is $[0,0,\fk]$.
However, $D$ can be decomposed into $E$ and $O$.
As a result of this split the total rank of the Conley indices increases by two.
In particular, we argue, that the degree 2 generator of $D$, which continues from $R$,
    has been passed to $O$ through the bifurcation. 
Moreover, the split of $D$ into two subcomponents created a new pair of coupled generators, 
    in this case of degree 0 and 1.
We record these events in the \emph{Conley-Morse persistence barcode} shown in Figure \ref{fig:CM-barcode-sphere_example}.
In particular, 
    the degree 2 generator of $R$ at $\lambda=0$ is represented by the orange bar that continues to $\lambda=1$ (and further).
For a $\lambda\in (0,1)$, two new bars are born representing the emergence of generators through the described split;
    the dashed line indicates that they are coupled (in the sense captured by Theorem~\ref{thm:index_triple}).
From the remaining part of the diagram we can read off further qualitative changes in the dynamics,
    for instance, we can see that for certain $\lambda\in(2,3)$, 
    the periodic orbit breaks creating two equilibria.
The diagram tells that their Conley indices inherit generators from the orbit.
Finally, for $\lambda\in(3,4)$, two equilibria and their Conley indices annihilate each other, 
    which is reflected by the ending of the corresponding bars.

While the bars represent the Conley index generators, the black vertical arrows indicate connecting trajectories between corresponding isolated invariant sets.
In particular, they represent \emph{Morse decompositions} and therefore the global gradient structure of the system.

Summarizing, the diagram in Figure~\ref{fig:CM-barcode-sphere_example} illustrates the evolution of the entire dynamical system through timelines of the Conley index generators; 
    for instance, the generator corresponding to $R$ at $\lambda=0$ travels from the north pole via the orbit to the new repeller $T$ in $\lambda=5$ close to the south pole, 
    which is not obvious simply from inspection of the vector field.

\begin{figure}[b]
    \centering
    \includegraphics[width=0.25\linewidth]{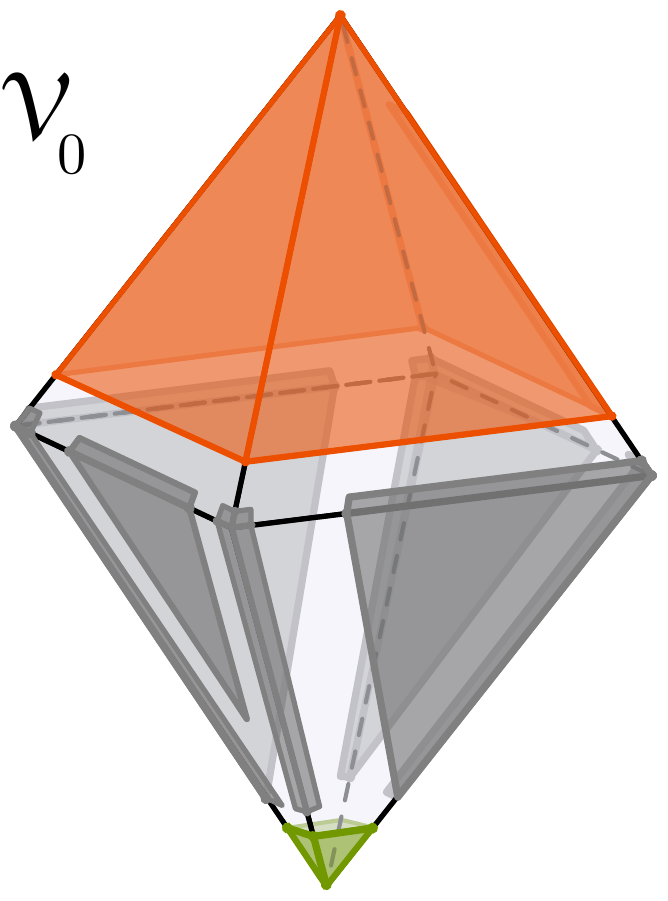}
    \includegraphics[width=0.25\linewidth]{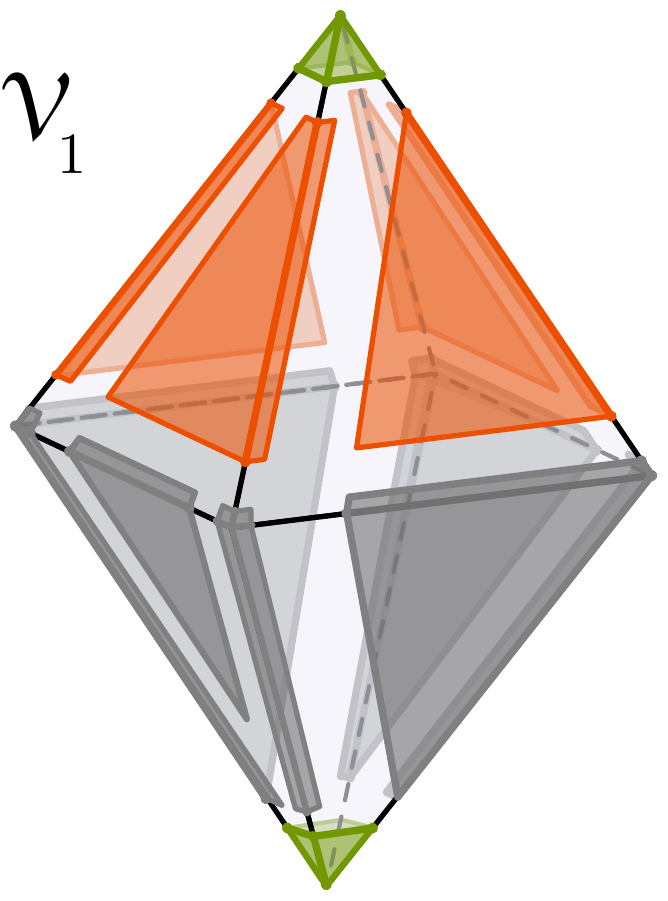}
    \includegraphics[width=0.25\linewidth]{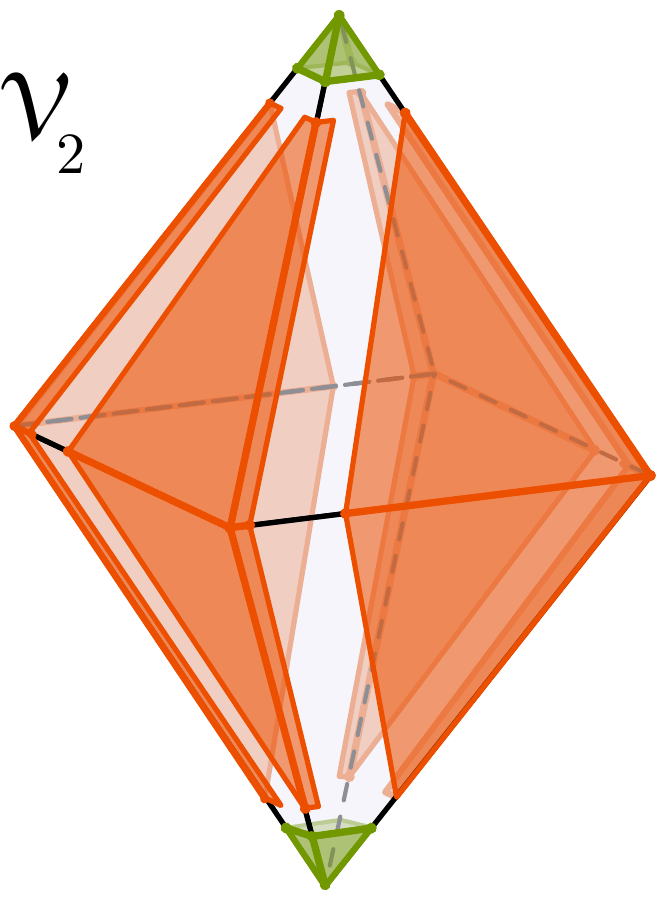}\\
    
    \vspace{1.2mm}
    \includegraphics[width=0.25\linewidth]{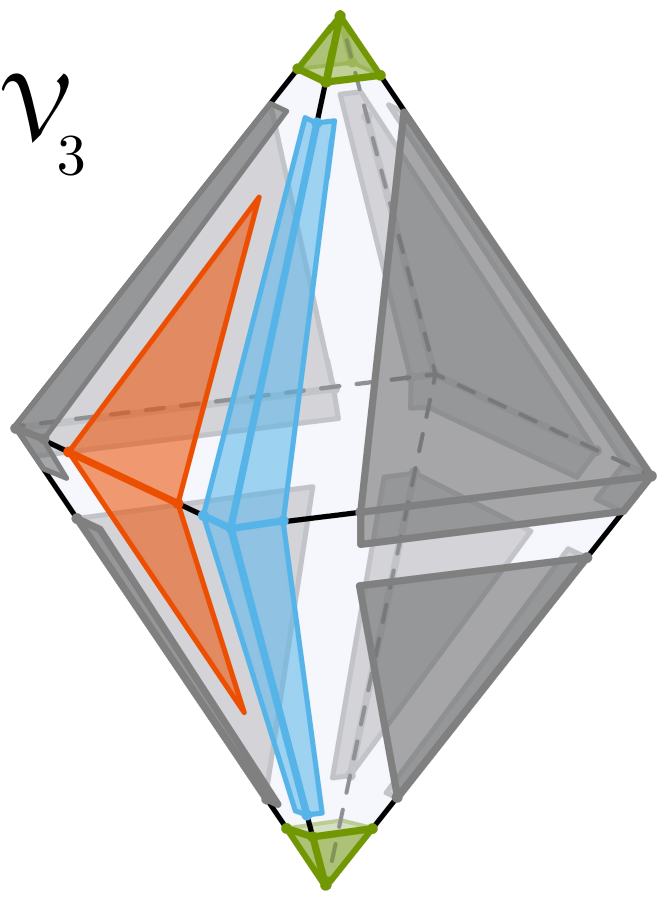}
    \includegraphics[width=0.25\linewidth]{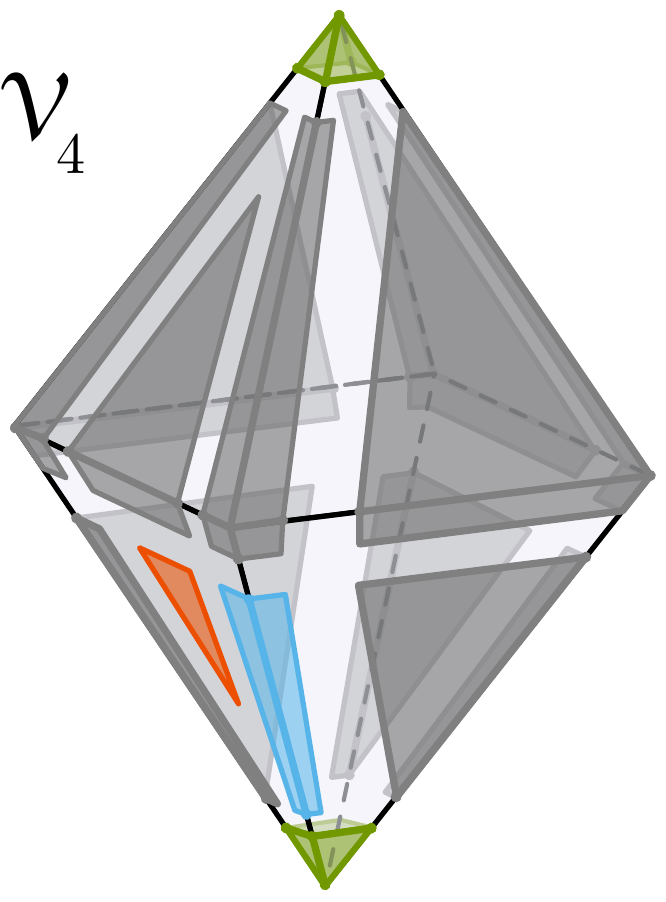}
    \includegraphics[width=0.25\linewidth]{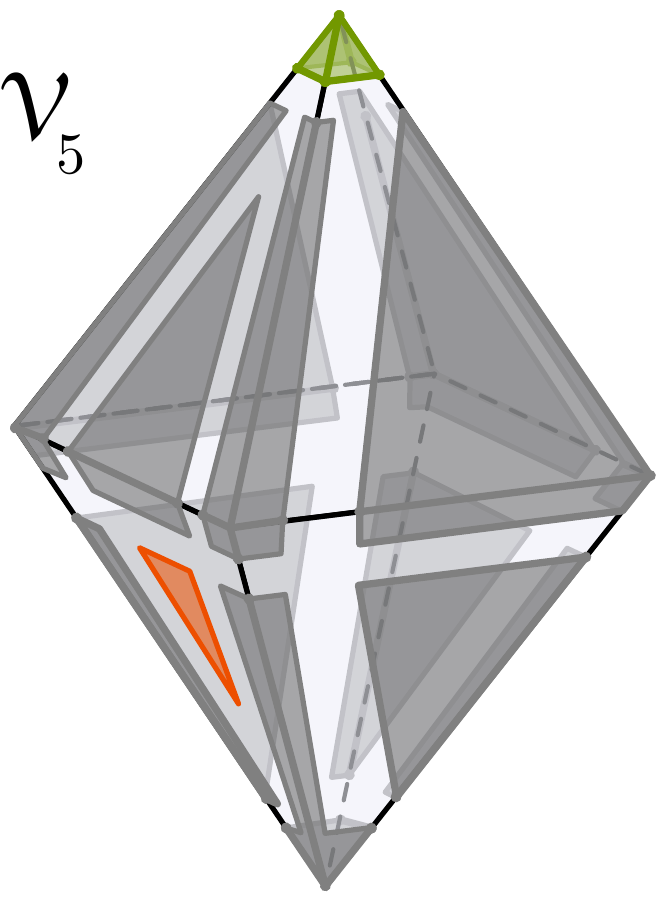}
    \caption{A parameterized combinatorial multivector field on a combinatorial 2-sphere, which models the system in Figure~\ref{fig:sphere-example-continuous-vf}.}
    \label{fig:sphere-example-combinatorial-mvf}
\end{figure}

We emphasize once more that the continuous example serves only as a source of intuition and an outline of the ultimate goal, which is beyond the scope of this paper. 
All results presented in this work are in the framework of combinatorial multivector fields (though, some can be also applied in the continuous case),
    which can serve as a model or an approximation of a continuous vector field.
Nevertheless, we strongly believe that the presented construction of the Conley-Morse persistence barcode can be adapted to the continuous setting; we leave it for future work.

All the dynamical concepts mentioned so far have their counterparts in 
combinatorial dynamics; see Section~\ref{sec:multivector-fields-theory}. 
In this context, a~multivector field $\cV$ models a continuous vector field, and a multivalued map $\mvm$---the flow (Subsection~\ref{subsec:mvf-elementary}).
The theory provides the concept of combinatorial isolated invariant sets and isolating blocks (Definition~\ref{def:isolating_block}), as well as Morse decomposition (Definition~\ref{def:morse_decomposition}) and the Conley index (Definition~\ref{def:conley_index}).
We also define a combinatorial continuation of an isolated invariant set (Definition~\ref{def:continuation}).

Figure~\ref{fig:sphere-example-combinatorial-mvf} shows a combinatorial model of the example in Figure~\ref{fig:sphere-example-continuous-vf}.
The hollow octahedron models the sphere and the sequence of multivector fields---as we argue in Section~\ref{subsec:combinatorial_continuation}---can be seen as a continuously parameterized combinatorial multivector field.
Each stage in Figure~\ref{fig:sphere-example-continuous-vf} has a counterpart in Figure~\ref{fig:sphere-example-combinatorial-mvf} with an isomorphic Morse decomposition and Conley indices.
In particular, $\cV_\lambda$ corresponds to $\varphi_\lambda$ for all $\lambda\in\{0,1,\ldots,5\}$.
We do not unfold this example in detail here, but we encourage the reader to take another look at the combinatorial example after going through Section~\ref{sec:multivector-fields-theory},
    to convince themselves that it indeed models the continuous flow shown in Figure~\ref{fig:sphere-example-continuous-vf}.
Moreover, the Conley-Morse persistence barcode in Figure~\ref{fig:CM-barcode-sphere_example} describes the 
    combinatorial dynamics in Figure~\ref{fig:sphere-example-combinatorial-mvf} as well.

	\begin{figure}[b]
        \centering
        \includegraphics[width=0.30\linewidth]{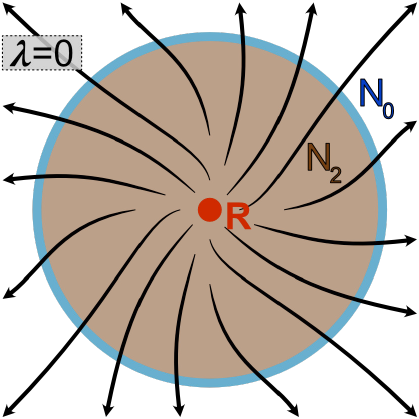}
        \hspace{1cm}
        \includegraphics[width=0.30\linewidth]{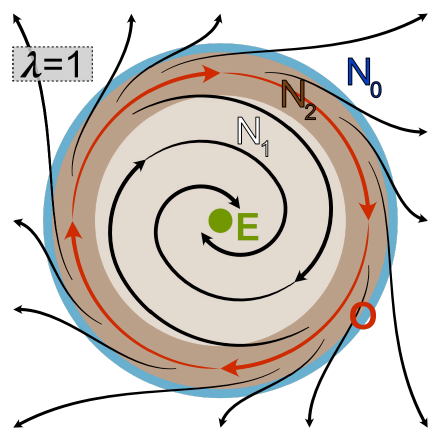}
        \caption{A Hopf bifurcation of repelling equilibrium $R$ (left panel) turning into attracting equilibrium $E$ and repelling periodic orbit $O$ (right panel).}
        \label{fig:hopf-example}
    \end{figure}

\subsection{Attractor-repeller split example}
    In this section we illustrate the key observation that allows us to construct the Conley-Morse persistence barcode.
    Consider again the Hopf bifurcation between $\lambda=0$ and $\lambda=1$ at the northern part of the sphere in Figure~\ref{fig:sphere-example-continuous-vf}, 
        with the repelling equilibrium $R$ turning into the attracting equilibrium $E$ and the repelling periodic orbit $O$. 
    Figure~\ref{fig:hopf-example} provides a top view of the sphere as the bifurcation occurs.

    Define $N_0$ as the blue ring and $N_2$ as the brown disc in the left panel of Figure~\ref{fig:hopf-example}.
    Let $N_1$ be the union of $N_0$ and the light brown disc in the right panel.
    In particular, we have $N_0\subset N_1\subset N_2$.
    
    As discussed earlier, $R$ in $\lambda=0$ continues to the invariant disc at $\lambda=1$---that is, the union of $E$, $O$ and the trajectories connecting them---because both sets are isolated by a common isolating block $N_2$.
    With $N_1$ and $N_0$, we can decompose the Conley index of $R$ as follows.
    First, observe that $B_E\coloneqq\cl(N_1\setminus N_0)$, $B_O\coloneqq\cl(N_2\setminus N_1)$, and $B_R\coloneqq\cl(N_2\setminus N_0)$ are isolating blocks for $E$, $O$, and $R$, respectively.
    However, we can compute the Conley index directly using $N_0$, $N_1$ and $N_2$. 
        By the excision property, we have $H(N_2, N_0)\cong H(B_R, B_R^-)$;
        similarly, $H(N_1, N_0)\cong H(B_E, B_E^-)$ and $H(N_2, N_1)\cong H(B_O, B_O^-)$ for 
            $B_E$ and $B_O$.
    In fact, these pairs form so-called \emph{index pairs}---we will discuss their combinatorial analogues in Section~\ref{subsec:conley-index}.
    They induce the following diagram that we call the \emph{attractor-repeller split diagram} (or \emph{AR-split diagram}):
    \begin{equation}\label{eq:introduction-basic-split}
        \begin{tikzcd}[row sep=normal, column sep=1.5em]
            H(N_2,N_0)\arrow[r]\arrow[rd, leftarrow]
                & H(N_2,N_1)\\
                & H(N_1,N_0)\arrow[u,dashed,swap]
        \end{tikzcd}
    \qquad\cong\qquad
        \begin{tikzcd}[row sep=normal, column sep=2.5em]
         [ 0, 0, \fk]\arrow[r, "{[0, 0, \id]}"] 
            & {[0, \fk, \fk]} \\
        & {[\fk, 0, 0]}\arrow[u, dashed,swap, "0"]\arrow[lu, "{[0, 0, 0]}"]
        \end{tikzcd}        
    \end{equation}       
    where the homomorphisms are induced by inclusions.
    These maps relate Conley index generators of isolated invariant sets before and after bifurcation.
    The diagram on the right shows concrete vector spaces and maps for the example.
    In particular, we see that the degree-2 generator of $R$ is mapped into the Conley index of $O$.

    The basic split theorem introduced in Section~\ref{subsec:basic-transition-diagram} (Theorem~\ref{thm:index_triple}) provides additional insight. 
    For instance, property~\ref{it:index_triple_split} of that theorem states that
        whenever an isolated invariant set splits, 
        no Conley index generator is lost---in other words,
        through a decomposition, every generator of a bifurcating set is ``passed on'' to one of the newly created sets.
    Property~\ref{it:index_triple_isomorphism} states that new generators are always born (or die) in pairs of codimension $1$.
    Both properties are illustrated in diagram~\eqref{eq:introduction-basic-split}: 
        first, because the degree-2 generator is mapped from $R$ to $O$;
        second, two new generators, of degree-$0$ and $1$, are born together during the split.
        
    \begin{figure}
        \centering
        \includegraphics[width=0.32\linewidth]{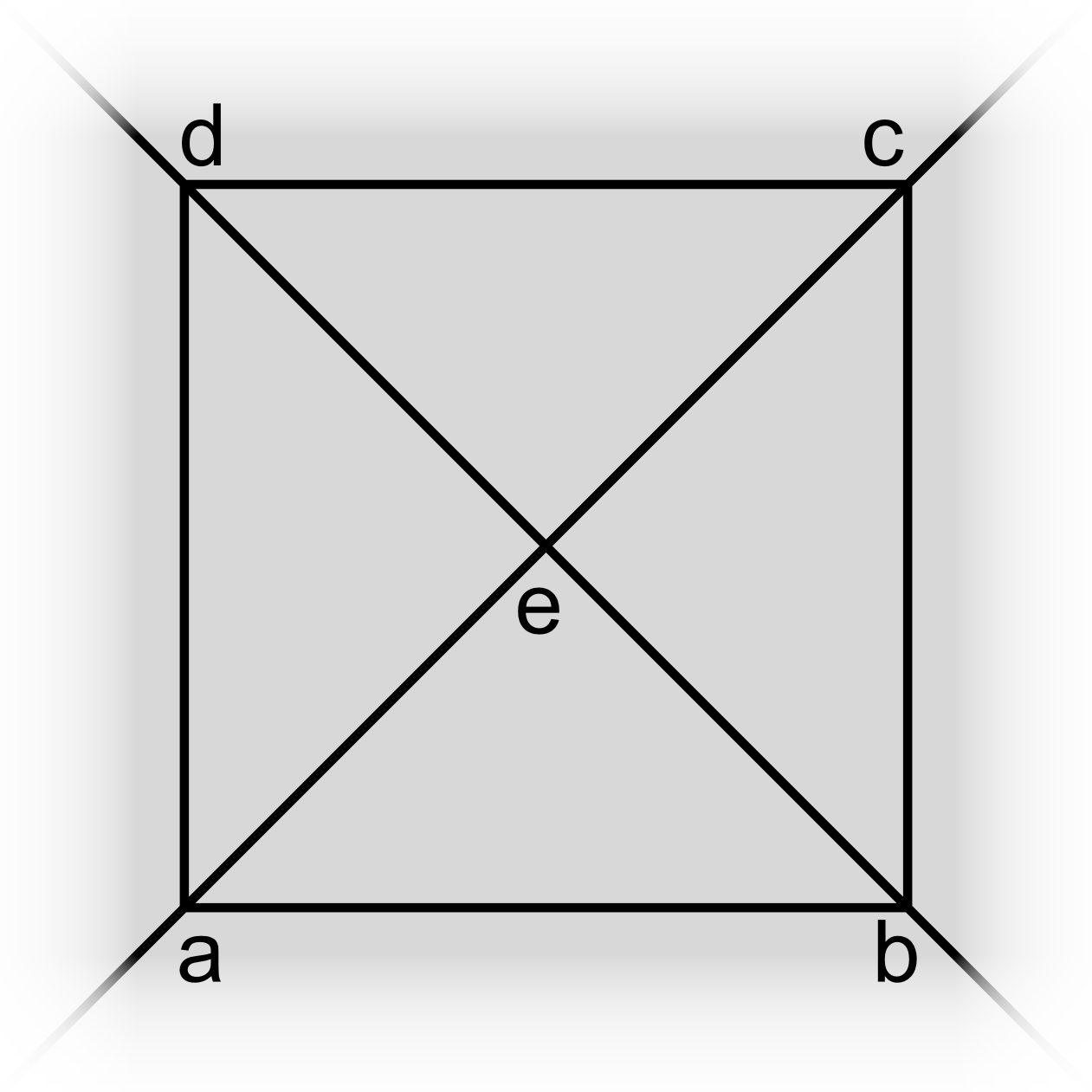}
        \includegraphics[width=0.32\linewidth]{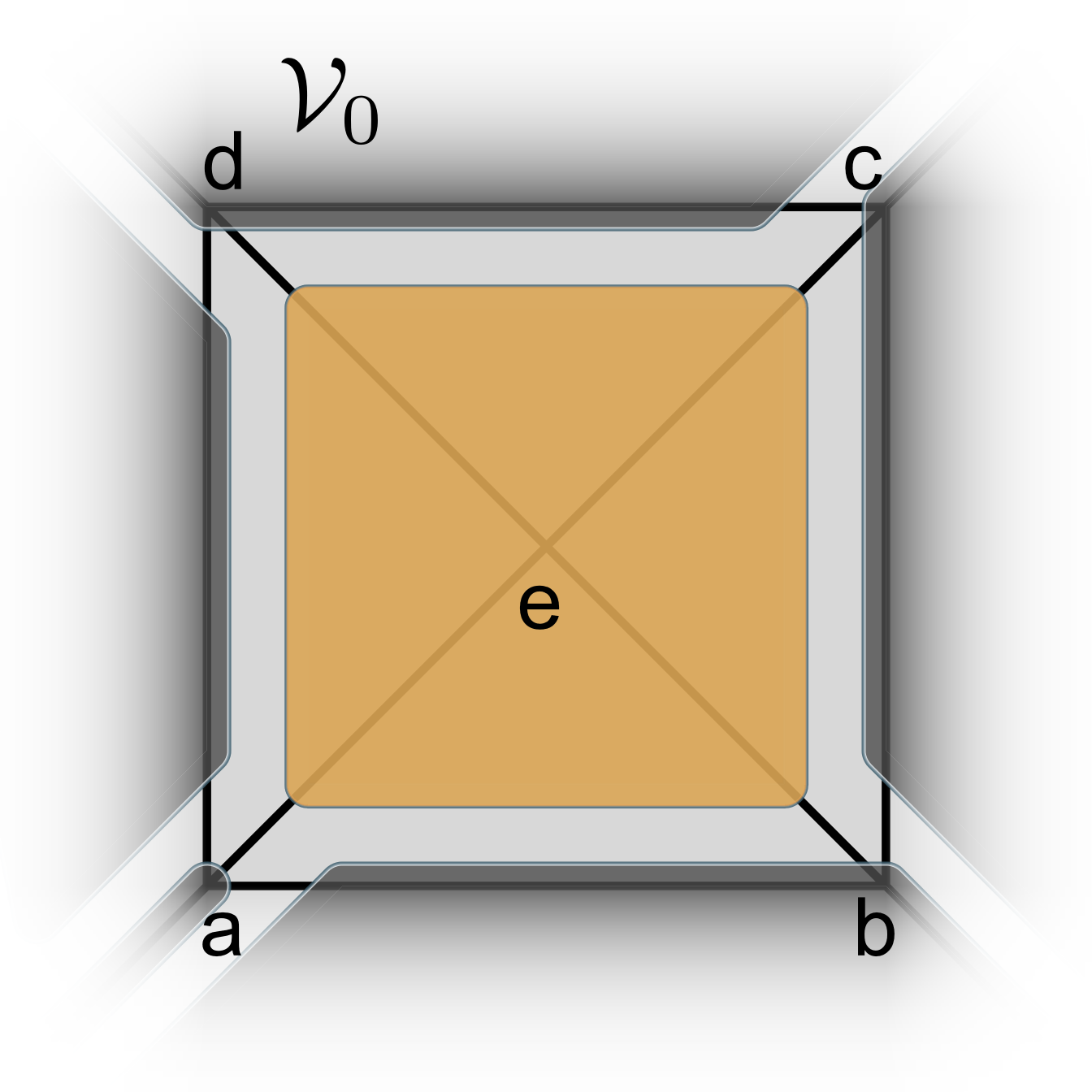}
        \includegraphics[width=0.32\linewidth]{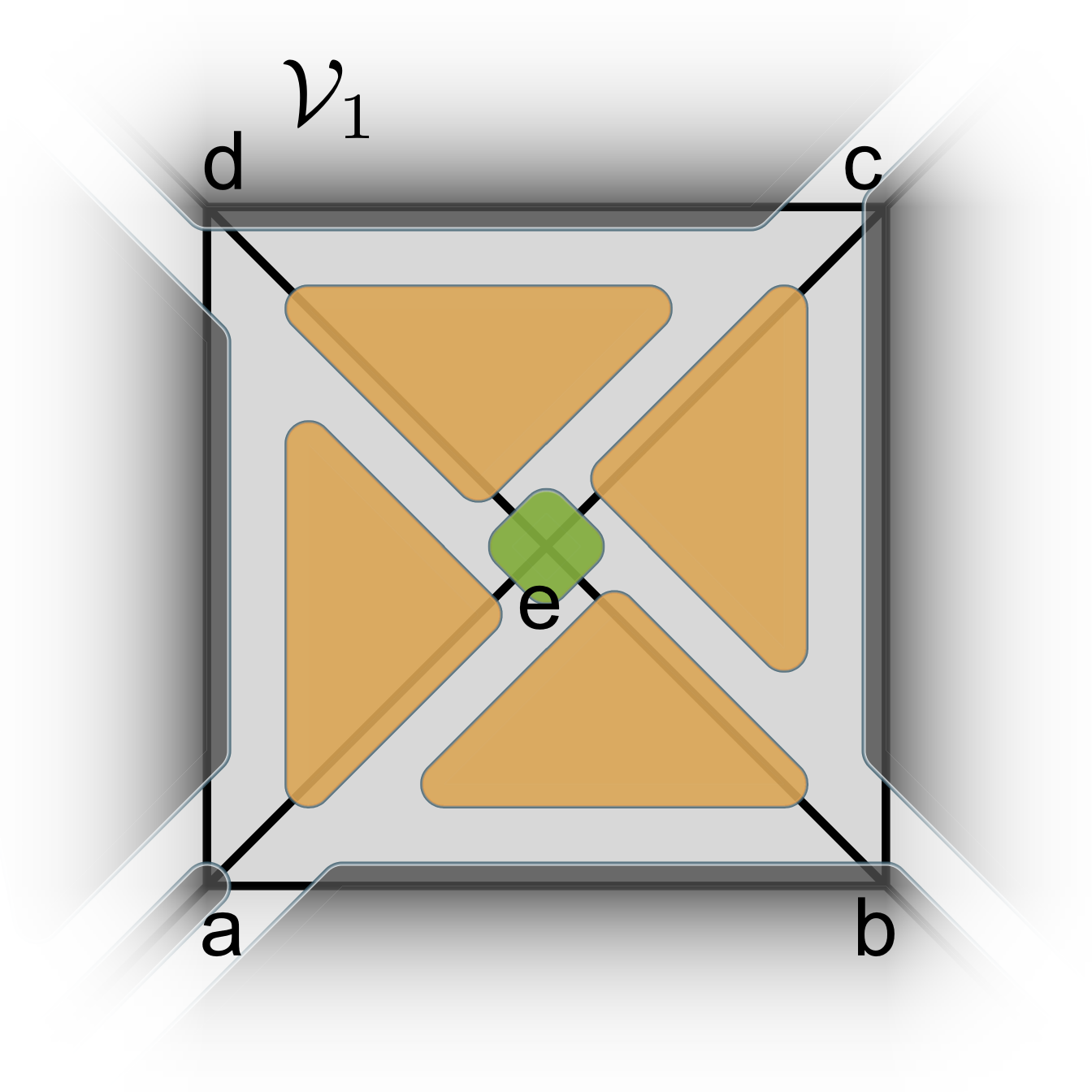}
        \caption{Combinatorial analogue of a Hopf bifurcation.
            The left panel represents the top view of the octahedron in Figure~\ref{fig:sphere-example-combinatorial-mvf},
            similarly, the middle and right panel show projections of multivector fields $\cV_0$ and $\cV_1$.
            }
        \label{fig:sphere-example-hopf-combinatorial}
    \end{figure}

    Figure~\ref{fig:sphere-example-hopf-combinatorial} 
        shows a combinatorial version of that bifurcation.
    In particular, it is the top view of the octahedron in Figure~\ref{fig:sphere-example-combinatorial-mvf}, where point $e$ corresponds to the vertex at the north pole.
    The combinatorial counterparts of $N_0$, $N_1$ and $N_2$ are
    \begin{align*}
        N_0&\coloneqq\{a, b, c, d, ab, bc, cd, ad\},\\
        N_1&\coloneqq N_0 \cup \{e\},\\
        N_2&\coloneqq N_1 \cup \{ae, be, ce, de, abe, bce, cde, ade\}.
    \end{align*}
    One can verify that the combinatorial AR-split diagram is identical, in terms of homology groups, to diagram~\eqref{eq:introduction-basic-split}.
    Moreover, $E=N_1\setminus N_0$,  $R=N_2\setminus N_0$, and $O=N_2\setminus N_1$ are combinatorial isolating blocks (see Definition~\ref{def:isolating_block}) 
   for the corresponding combinatorial isolated invariant sets.

\subsection{Homological signature of a bifurcation}
Another motivation for the development of the Conley-Morse persistence barcodes is the classification of bifurcations.
Consider two 1-dimensional flows, $\varphi_\lambda$ and $\psi_\lambda$ parameterized by $\lambda$, as illustrated in the upper panels of Figure~\ref{fig:bifurcation1d}.
Each vertical line through the plot represents a 1-dimensional flow on the real line.
Red and green segments denote repelling and attracting equilibria, respectively.
Note that, pointwise---for a fixed $\lambda$---both dynamical systems are qualitatively the same: 
\begin{itemize}
    \item for $\lambda<-1$ and $\lambda>1$ they have a single attracting equilibrium; 
    \item for $\lambda\in(-1,1)$---two attracting and one repelling equilibrium in between;
    \item for $\lambda\in\{-1,1\}$---an attracting equilibrium and one degenerate equilibrium.
\end{itemize}
However, the corresponding Conley-Morse persistence barcodes presented in Figure~\ref{fig:bifurcation1d} (bottom) capture the difference in the nature of the two bifurcations.
Thus, the Conley-Morse persistence barcode can be regarded as an algebraic signature of a bifurcation.

Figure~\ref{fig:cusp-isola-mvf} presents minimalist combinatorial models for both bifurcations.
Again, we encourage the reader to revisit the examples in Figures~\ref{fig:sphere-example-hopf-combinatorial} and~\ref{fig:cusp-isola-mvf} after going through Section~\ref{sec:multivector-fields-theory}.

\begin{figure}
    \centering
    \includegraphics[width=0.45\linewidth]{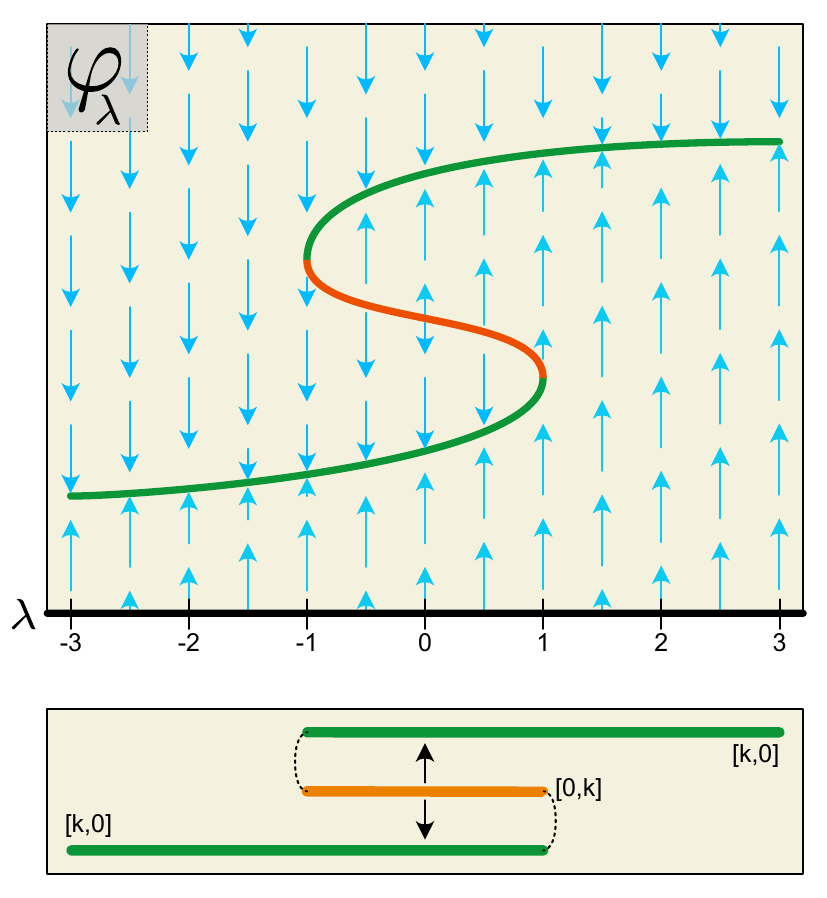}
    \includegraphics[width=0.45\linewidth]{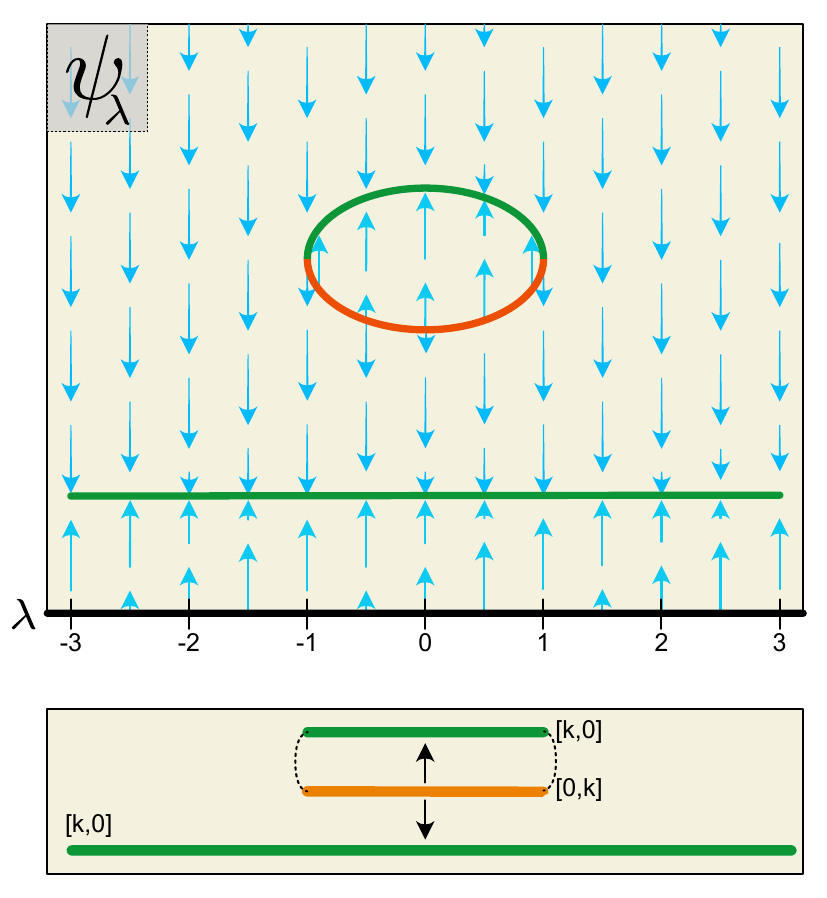}
    \caption{Two 1-dimensional flows parameterized by $\lambda$.
    The bottom row presents the corresponding Conley-Morse persistence barcodes.}
    \label{fig:bifurcation1d}
\end{figure}

\begin{figure}
    \centering
    \includegraphics[width=0.45\linewidth]{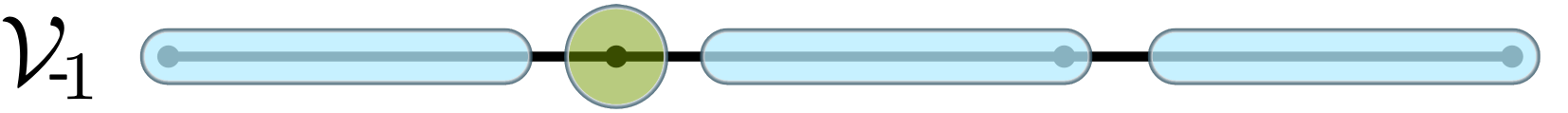}
    \hfill
    \includegraphics[width=0.45\linewidth]{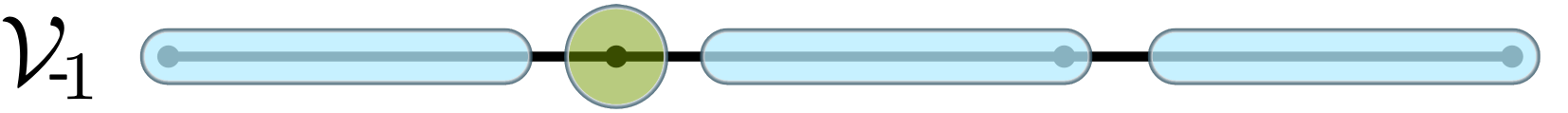}
    
    \includegraphics[width=0.45\linewidth]{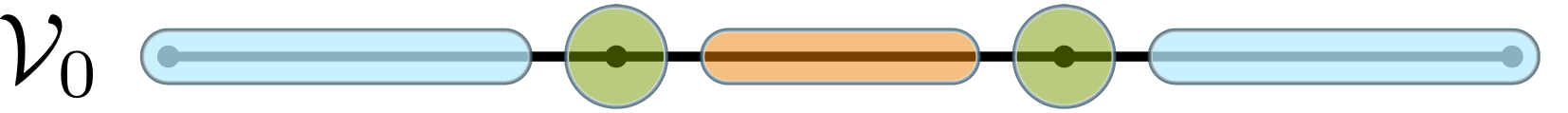}
    \hfill
    \includegraphics[width=0.45\linewidth]{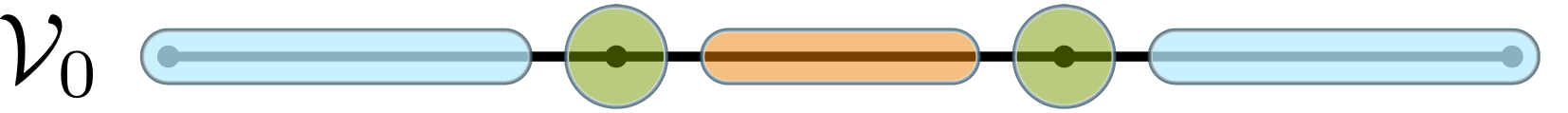}
    
    \includegraphics[width=0.45\linewidth]{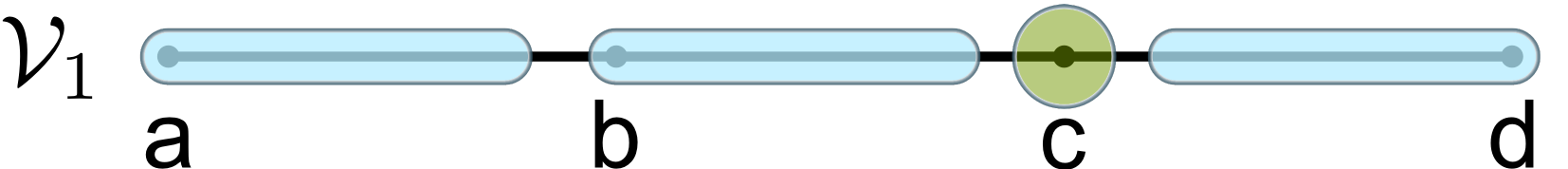}
    \hfill
    \includegraphics[width=0.45\linewidth]{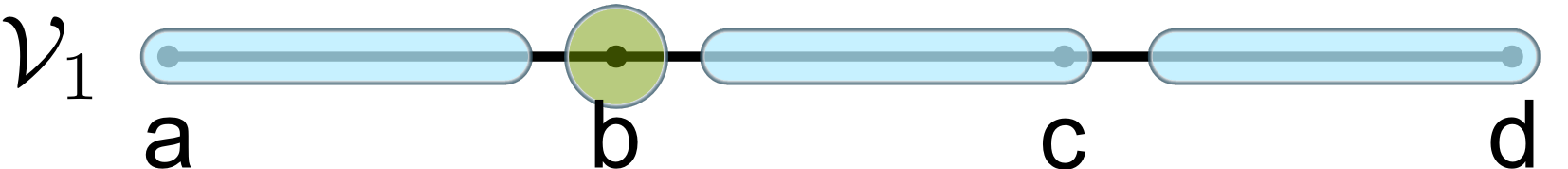}
    \caption{Minimalist combinatorial models for the parameterized 1-dimensional flows in Figure~\ref{fig:bifurcation1d}.}
    \label{fig:cusp-isola-mvf}
\end{figure}

\section{Preliminaries}\label{sec:preliminaries}
\subsection{Sets}\label{subsec:prelim-sets}
A $\ZZ$-interval $I$ is an intersection of an interval in $\RR$ with $\ZZ$.
We say that a $\ZZ$-interval $I$ is \emph{right-bounded}\index{right/left-bounded set} if $I$ admits a maximal element, otherwise $I$ is \emph{right-infinite}\index{right/left-infinite set}.
Similarly, if $I$ has a minimal element, then it is \emph{left-bounded}; otherwise, it is \emph{left-infinite}. 
We denote bounded $\ZZ$-intervals by $\zintab{n}{m} \coloneqq \{n, n+1, \ldots, m\}$.

Let $\cA$ and $\cB$ be families of subsets of $X$.
Then we say that $\cA$ is \emph{inscribed in}\index{inscribed (relation)} $\cB$ if for every $A\in\cA$ there exists $B\in\cB$ such that $A\subset B$.
We denote this relation by writing $\cA\sqsubseteq\cB$.

\subsection{Digraphs}\label{subsec:prelim-graphs}
A pair $G=(V, E)$ is called a \emph{directed graph}\index{directed graph} (or \emph{digraph}\index{digraph}), where $V$ is the set of vertices and relation $E\subset V\times V$ is the collection of edges. 
A sequence $\rho\coloneqq \rho_0,\rho_1,\ldots, \rho_\ell$ is a \emph{path}\index{path} of length $\ell$ in $G$ if $(\rho_{i-1}, \rho_i)\in E$ for all $i=1,2,\ldots, \ell$.
The endpoints of $\rho$ are denoted by $\pbeg\rho\coloneqq \rho_0$ and $\pend\rho\coloneqq \rho_\ell$.
Let $\psi$ be another path in $G$ such that $(\pend\rho, \pbeg\psi)\in E$ then the \emph{concatenation}\index{concatenation of paths} of $\rho$ and $\psi$ is also a path in $G$, 
    we denote it by $\rho\cdot\psi$.
Sometimes we identify a vertex $v\in V$ with a path of length $0$, 
    which allows us to write $x\cdot y\cdot z$ for a path of length 2, under the assumption that $(x,y),(y,z)\in E$. 

A \emph{strongly connected component}\index{strongly connected component} (\scc) of $G$ 
    is a maximal subset $A\subset V$ such that for every $x,y\in A$ there exists a path $\rho$ from $x$ to $y$.

\subsection{Relations}\label{subsec:prelim-rels}
Let $R\subset X\times X$ be a binary relation on a set~$X$.
A pair $(X, R)$ is called a \emph{partially ordered set} (or a \emph{poset})\index{partial order}\index{poset} if $R$ is a reflexive, antisymmetric, and transitive relation.
If additionally any two elements of $X$ are comparable then $(X,R)$ is called a \emph{linear order}\index{linear order}.
A linear order $(X,R')$ is a \emph{linear extension}\index{linear extension} of a partial order $(X, R)$ if $R\subset R'$.

For a partial order $(X,\leq)$ and $x,y\in X$ we write $x\prec y$ if there is no $z\in X\setminus\{x,y\}$ such that $x<z<y$.
The \emph{Hasse diagram}\index{Hasse diagram} of a partial order $(X,\leq)$ is a restriction of $\leq$ to pairs $(x,y)$ such that $x\prec y$.

A subset $A\subset X$ is called an \emph{upper set}\index{upper set} if $x< y $ and $x\in A$ implies $y\in A$.
Symmetrically, $A$ is a \emph{down set}\index{down set} if $x< y $ and $y\in A$ implies $x\in A$.
An intersection of a down set and an upper set is called a \emph{convex set}\index{convex set}.

A \emph{fence}\index{fence} in a poset $(X,\leq)$ is a sequence $\seqof{x}{n}\subset X$ such that for every $i\in\{1,2,\ldots n\}$ either $x_{i-1}\leq x_i$ or $x_{i-1}\geq x_i$.

\subsection{Finite topological spaces}\label{subsec:prelim-ftop}

Let $(X, \cT)$ be a finite topological space satisfying the T$_0$ separation axiom.
For a subset $A\subset X$ we denote its \emph{closure}\index{closure} by $\cl A$.
Since $X$ is finite, there exists a \emph{minimal open set}\index{open set!minimal} in $\cT$ containing  $A$, which we denote by $\opn A$.
We define the \emph{mouth}\index{mouth} of $A$ as $\mo A\coloneqq \cl A\setminus A$. 
A set $A$ is said to be \emph{locally closed}\index{locally closed set} if $\mo A$ is closed.

\begin{proposition}\label{prop:lcl}\cite[Problem~2.7.1]{Engelking1989}
    Let $A\subset X$. Then, the following conditions are equivalent:
    \begin{enumerate}
        \item $A$ is locally closed,
        \item $A$ is a difference of two closed sets,
        \item $A$ is an intersection of an open and closed set.
    \end{enumerate}
\end{proposition}

The following theorem allows us to identify a finite $T_0$ topological space with a finite partial order. 
In particular, open sets translate into upper sets, closed sets into down sets, and locally closed sets into convex sets.
\begin{theorem}[Alexandrov Theorem \cite{Alexandrov1937}]\label{thm:alexandrov_theorem}
    Let $(X, \leq)$ be a finite, partially ordered set.
    The family of all upper sets of $(X, \leq)$ forms a $T_0$ topology $\cT_\leq$ on $X$.
    Conversely, a $T_0$ finite topological space $(X,\cT)$ induces a partial order $(X,\leq_\cT)$, where $x \leq_\cT y$ whenever $x\in \cl y$.
    In particular $\cT\equiv\cT_{\leq_\cT}$ and $\leq\equiv\leq_{\cT_\leq}$.
    Moreover, continuous maps can be identified with order-preserving maps.
\end{theorem}

\section{Combinatorial Multivector Fields Theory}
In this section we cover the  theory of combinatorial multivector fields.
Most of the definitions come from~\cite{DeLiMrSl2022, LKMW2022}.
We reformulate some of the concepts  and introduce new ones to fit our specific needs. 
\label{sec:multivector-fields-theory}
\subsection{Elementary notions}\label{subsec:mvf-elementary}
Let $X$ be a finite topological space; in particular $X$ can be a simplicial or a regular CW-complex. 
A \emph{multivector field}~$\cV$\index{multivector field} on~$X$ is a partition of $X$ into locally closed subsets, called \emph{multivectors}\index{multivector}.
Since~$\cV$ is a partition, every $x\in X$ belongs to a unique multivector $V\in\cV$; we denote it by $[x]_\cV$. 
A~multivector $V$ is called \emph{regular}\index{multivector!regular} if the relative homology group $H(\cl V,\mo V)$ is $0$, otherwise $V$ is \emph{critical}\index{multivector!critical}.
A~set~$A\subset X$ is called \emph{$\cV$-compatible}\index{$\cV$-compatible set} if it is a union of multivectors in $\cV$.

A multivector field $\cV$ induces a multivalued map $F_\cV: X\multimap X$ defined as
\begin{equation}\label{eq:FV}
    \mvm(x)\coloneqq  \cl x\cup \mvx{x}.
\end{equation}
The map $\mvm$ can be viewed as a digraph $\dgV$ with nodes given by $X$ and the set of edges consisting of pairs $(x,y)\in X\times X$ such that $y\in \mvm(x)$.
A \emph{solution}\index{solution} is a path in $\dgV$, which we represent as a map $\rho:I\rightarrow X$, where $I$ is an interval in $\ZZ$.

\begin{example}\label{ex:first-example-MVF}
    Figure \ref{fig:first-example-MVF} shows an example of a multivector field $\cV$ on a simplicial complex, consisting of ten multivectors: 
        $\{abc\}$, $\{a,ab\}$, $\{b,bd\}$,
        $\{bc, bcd\}$,
        $\{c,ac\}$,
        $\{d,cd\}$,
        $\{df, def\}$,
        $\{e,ce,de,cde\}$,
        $\{ef\}$, and 
        $\{f\}$.
    The three singleton multivectors are critical and the rest are regular.
    Singleton $\{f\}$ models an attracting equilibrium;
        $\{ef\}$ a saddle; and
        $\{abc\}$ a repelling equilibrium.
    In particular, $\mvm(f)=\{f\}$, thus there are no outward arrows from $f$ in $\dgV$.
    One can also check that there are no arrows pointing toward $\{abc\}$. 
    \qedex
\end{example}

\begin{figure}
    \centering
    \includegraphics[width=0.65\linewidth]{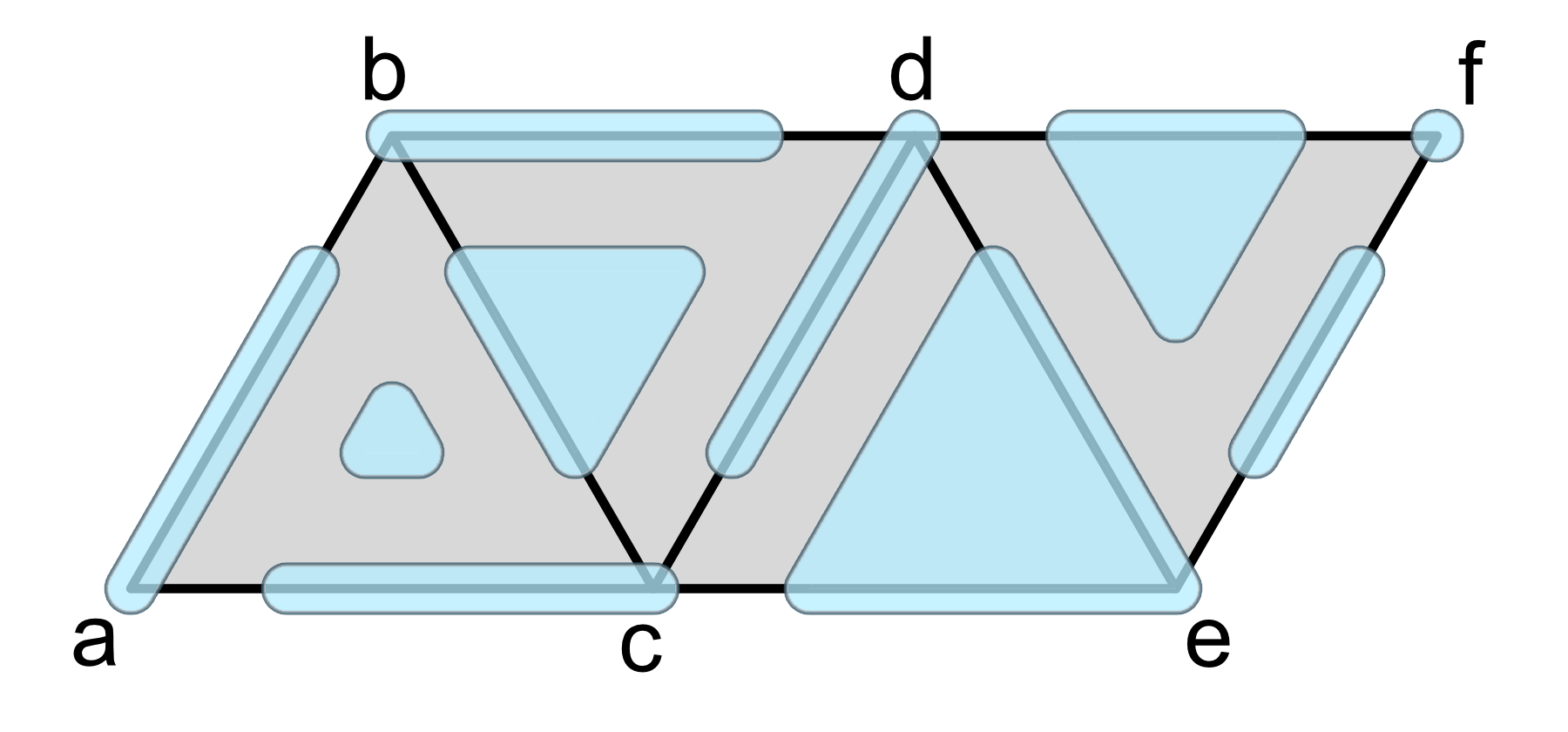}
    \caption{Example of a multivector field on a simplicial complex.}
    \label{fig:first-example-MVF}
\end{figure}

We write $\solV(A)$ for the family of all solutions in the graph $\dgV$ such that $\im\rho\subset A$.
It will be handy to distinguish the following subsets of $\solV(A)$: 
\begin{align*}
    \pathsV(A) &\coloneqq  \{\rho\in\solV(A)\mid \dom\rho \text{ is bounded}\},\\
    \pathsV(B_s,B_e,A) &\coloneqq  \{\rho\in\pathsV(A)\mid \pbeg\rho\in B_s \text{ and } \pend\rho\in B_e\},\\
    \isol_\cV(A) &\coloneqq  \{\rho\in\solV(A)\mid \dom\rho=\ZZ\}.
\end{align*}
In particular, $\pathsV(B_s, B_e,A)$ contains all bounded solutions in $A$ 
    with the starting point in $B_s$ and the end point in $B_e$.
On the other hand $\isol_\cV(A)$ consists of all bi-infinite solutions in~$A$, we call them \emph{full solutions}\index{solution!full}.

While useful, the above notion of a solution is not enough to capture the essential structure of a multivector field.
Therefore, we distinguish another type of solutions, called \emph{essential}.
In particular, a full solution $\varphi$ is called \emph{right-essential}\index{solution!right-essential} (\emph{left-essential}\index{solution!left-essential}) if for every $t\in\ZZ$ there exist $s>t$ ($s<t$) such that $\mvx{\varphi(t)}\neq\mvx{\varphi(s)}$ or $\mvx{\varphi(s)}$ is critical. 
A full solution $\varphi$ is called \emph{essential}\index{solution!essential} if it is both right- and left-essential.
In other words, a full solution is essential if it leaves every regular multivector it enters within a finite amount of steps---both in forward and backward time direction.
Note that a regular multivector may still be visited an infinite number of times by a single essential solution.
We denote the set of all essential solutions in $A\subset X$ by $\esolV(A)$, and the subset of essential solutions passing through $x\in A$ by
\begin{equation*}
    \esolV(x, A)\coloneqq \{\varphi\in\esolV(A)\mid \varphi(0)=x\}.
\end{equation*}
We usually use $\rho$ and $\gamma$ for non-essential solutions and $\varphi$ and $\psi$ for essential solutions.
To summarize, we have the following correspondence between the introduced families of solutions:
\begin{equation*}
    \esolV(x,A)\subset\esolV(A)\subset\isolV(A)\subset\solV(A)\supset\pathsV(A)\supset\pathsV(x,y,A).
\end{equation*}

Similarly to paths in a graph, 
given two solutions $\rho,\gamma\in\solV(A)$, right- and left-bounded, respectively, such that $\pbeg\gamma\in\mvm(\pend\rho)$ we write $\rho\cdot\gamma$ for the new solution constructed as the concatenation\footnote{
    Clearly, the concatenation of two solutions requires adjusting the domain of the new path.
    There are couple ways of doing that (for instance, see \cite[Section 4.2]{LKMW2022});
    nevertheless the choice of reparameterization is irrelevant here.
}.
Sometimes, for simplicity, we identify a point $x\in X$ with the trivial solution $\rho_x:\{0\}\rightarrow X$, defined $\rho_x(0)\coloneqq x$.
For instance, by $x\cdot y\cdot z$ we mean the solution $\rho\coloneqq \rho_x\cdot\rho_y\cdot\rho_z$ (under the assumption that $y\in\mvm(x)$ and $z\in\mvm(y)$), that is
\begin{equation*}
    \rho(t)\coloneqq \begin{cases}
        x; & t=0,\\
        y; & t=1,\\
        z; & t=2.
    \end{cases}
\end{equation*}

With the notion of an essential solution we are ready to define the notion of invariance and isolation in the context of multivector fields.
The \emph{invariant part}\index{invariant part} of a set $A\subset X$ is defined as
    $\inv_\cV(A) \coloneqq  \{x\in A \mid \esolV(x, A)\neq\emptyset\}=\bigcup\{\im\varphi\mid\esolV(A)\}$.
A set $S\subset X$ is called an \emph{invariant}\index{invariant set} if $\inv_\cV S= S$. 
Clearly, an invariant part forms an invariant set.

\begin{definition}\label{def:isolating_block} (Isolating block and isolated invariant set)
A set $N\subset X$ is a (combinatorial) \emph{isolating block}\index{isolating block}
if for every path $x\cdot y\cdot z$ such that $x,z\in N$ implies $y\in N$.
An invariant set $S$ is called an \emph{isolated invariant set}\index{invariant set!isolated} if there exists an isolating block $N$ such that $S=\inv_\cV N$.
\end{definition}

\begin{proposition}\label{prop:iso_block_is_lcl_Vcomp}
    A set $N$ is an isolating block if and only if $N$ is locally closed and $\cV$-compatible.
\end{proposition}
\begin{proof}
    Assume that $N$ is an isolating block.
    To prove that $N$ is $\cV$-compatible, we need to show that for any $x \in N$, $\mvx{x}$ is a subset of $N$.
    To do so, fix  $x\in N$ and $y\in\mvx{x}$ and notice that $y\in\mvm(x)$ and $x\in\mvm(y)$.
    By definition, path $x\cdot y\cdot x$ implies that $y\in N$.
    Thus, $\mvx{x}\subset N$, which proves that $N$ is $\cV$-compatible. 
    Similarly, let $x,z\in N$ and $y\in X$ such that $x>y>z$.
    It follows that $y\in\mvm(x)$ and $z\in\mvm(y)$.
    Therefore, a path $x\cdot y\cdot z$ is a solution and $y\in N$ by definition of isolating block, which proves that $N$ is locally closed.

    Now assume that $N$ is a locally closed and $\cV$-compatible set.
    Suppose that there exists a path $x\cdot y\cdot z$ such that $x,z\in N$ and $y\not\in N$.
    Necessarily, $\mvx{x} \neq \mvx{y} \neq \mvx{z}$, because of the $\cV$-compatibility of $N$.
    Thus, we have $y\in\cl(x)$, $z\in\cl(y)$ and $y \in \mo N$.
    Lastly, since $N$ is locally closed, 
    $\mo N$ is closed and $z \in \mo N$, a contradiction.
\end{proof}

\begin{proposition}\cite[Propositions 4.10, 4.11 \& 4.12]{LKMW2022}
\label{prop:iso_inv_set_is_lcl_Vcomp}
    An invariant set $S$ is an isolated invariant set if and only if it is locally closed and $\cV$-compatible.
\end{proposition}

\begin{corollary}\label{cor:iso_inv_is_iso_block}
    An isolated invariant set $S$ is itself the minimal isolating block for $S$.
\end{corollary}

\begin{example}
    The left panel of Figure \ref{fig:first-example-BD} shows a multivector field with four isolating blocks highlighted in brown. 
    Three of them, $B_1$, $B_2$ and $B_3$, have non-empty invariant part, which are respectively denoted $M_1$, $M_2$, and $M_3$, and highlighted in green in the right panel of Figure \ref{fig:first-example-BD}.
    \qedex
\end{example}

\begin{figure}
    \centering
    \includegraphics[width=0.40\linewidth]{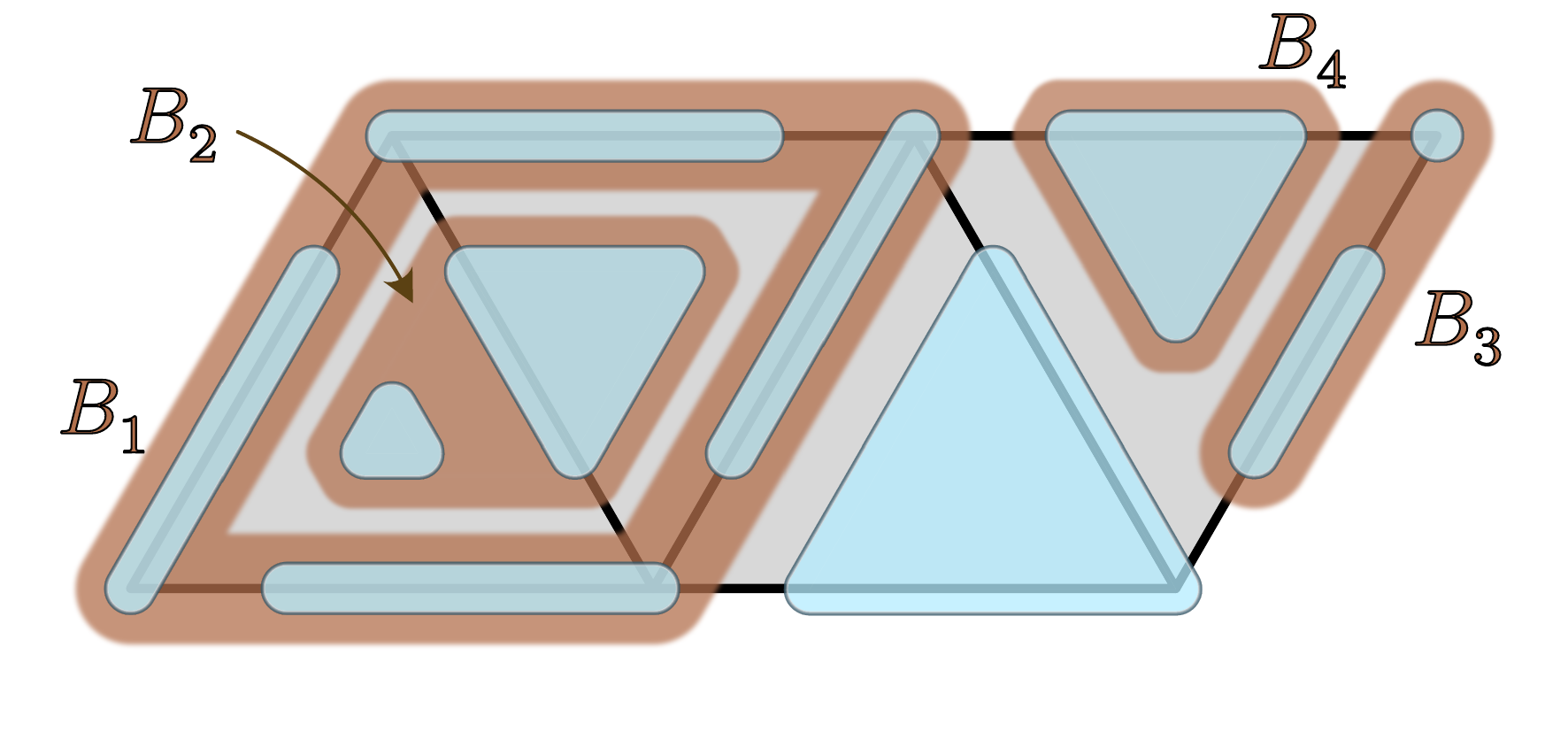}
    \includegraphics[width=0.09\linewidth]{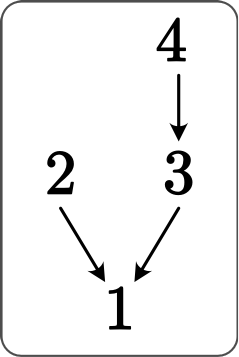}
    \includegraphics[width=0.40\linewidth]{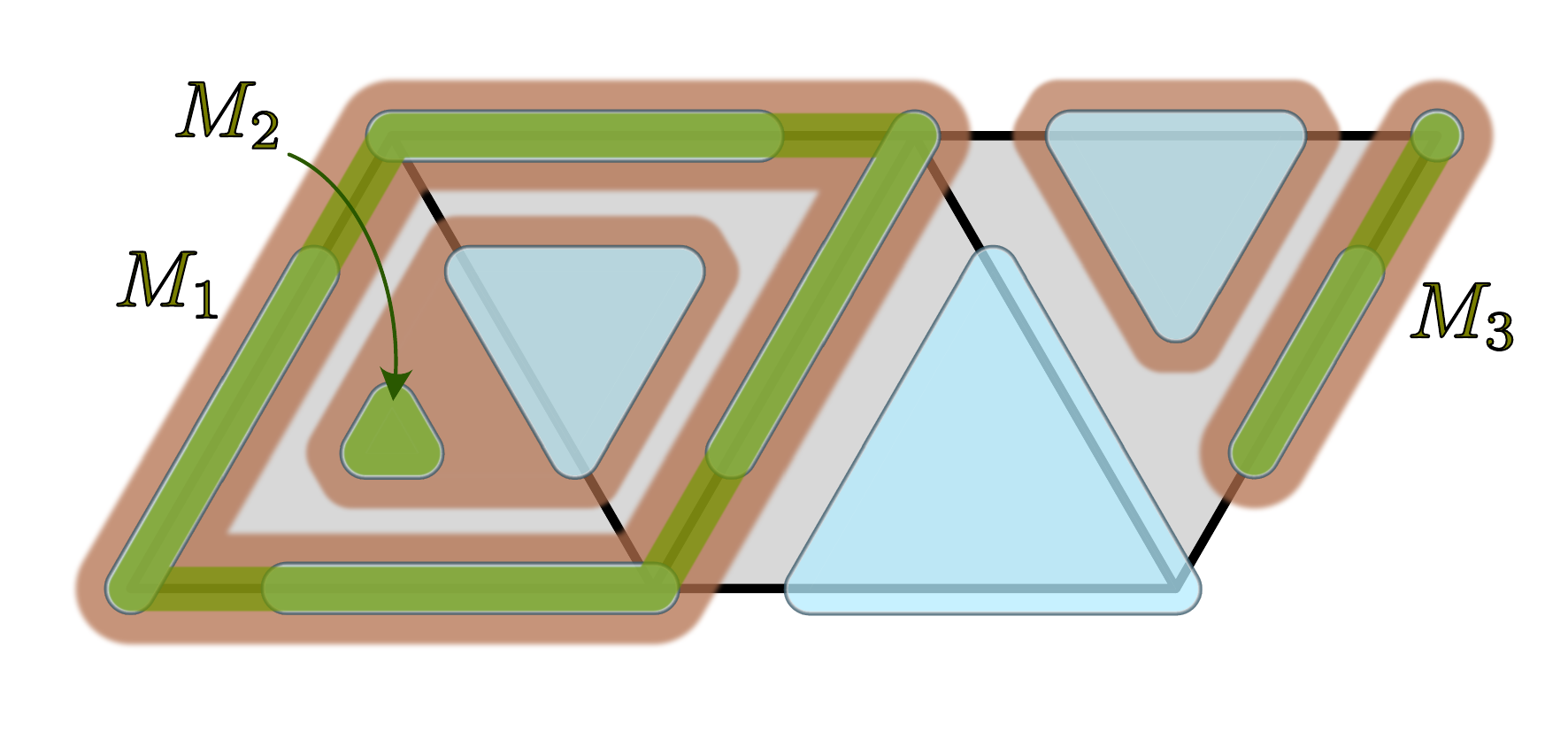}
    \caption{Examples of isolating blocks (brown sets) and isolated invariant sets (green sets) for a multivector field $\cV$ from Figure~\ref{fig:first-example-MVF}.
    Four isolating blocks on the left panel form a block decomposition $\BD$ of $\cV$.
    The green sets on the right panel indicate the invariant parts of the isolating blocks in $\BD$ forming a Morse decomposition.
    The graph on the middle panel represents the flow induced partial order for $\BD$.
    }
    \label{fig:first-example-BD}
\end{figure}

Note, that in the theory of continuous flows, an isolating block $\bl$ is defined as a compact set with the closed exit set and no internal tangencies.
In the combinatorial setting the only way to escape an isolating block is through its mouth, which by Proposition~\ref{prop:iso_block_is_lcl_Vcomp} has to be closed as well.
Moreover, no path starting in an isolating block $B$ can go to $\mo B$ and directly return to $B$.
    This can be viewed as a counterpart of the ``no internal tangencies'' condition.
The closedness of the isolating block, however, must be abandoned due to sparsity inherent in finite topological spaces.

\subsection{Morse and block decomposition}\label{subsec:morse-block-decomposition}

Let $\varphi$ be a full solution in $\cV$.
We define the \emph{ultimate backward and forward images}\index{ultimate image}\index{ultimate image!backward}\index{ultimate image!forward} of $\varphi$:
\begin{align*}
    \uimm\varphi &\coloneqq  \bigcap_{t<0}\varphi((-\infty, t]),\\
    \uimp\varphi &\coloneqq  \bigcap_{t>0}\varphi([t, +\infty)).
\end{align*}
Clearly, since the space $X$ is finite, the ultimate images are always non-empty.

\begin{definition}\label{def:morse_decomposition}(Morse decomposition)  
    \cite[Definition 7.1]{LKMW2022}
    A collection $(\cM, \PP) \coloneqq  \{M_p \mid p\in\PP\}$ of mutually disjoint, non-empty \underline{isolated invariant sets} is called a \emph{Morse decomposition}\index{Morse decomposition} of $S\subset X$
    in $\cV$ if there exists a partial order $(\PP,\leq)$ such that
    \begin{enumerate}[label=(M\arabic*)]
        \item\label{it:morse_decomposition_uims}
            for every $\varphi\in\esolV(S)$
            there exist $p,q\in\PP$ such that $\uimm\varphi\subset M_p$ and $\uimp\varphi\subset M_q$, 
        \item\label{it:morse_decomposition_paths} 
            if there exists $\rho\in\pathsV(S)$ with $\pbeg\rho\in M_p$ and $\pend\rho\in M_q$ for some $p,q\in\PP$ then
                either 
                \begin{enumerate}[label=(\alph*)]
                    \item\label{it:morse_decomposition_paths_le} $p>q$, or
                    \item\label{it:morse_decomposition_paths_eq} $p=q$ and $\im\rho\subset M_p$. 
                \end{enumerate}
    \end{enumerate}
\end{definition}

The primary goal of this work is to study the evolution of a Morse decomposition.
However, in order to study continuation it is preferable to shift our attention from isolated invariant sets to isolating blocks.
Therefore, we introduce the concept of a block decomposition, 
    which we obtain  by replacing isolated invariant sets with isolating blocks in Definition \ref{def:morse_decomposition}.

\begin{definition}\label{def:block_decomposition}(Block decomposition)
    \cite[Definition 7.2.5]{MroWan2025}
    A collection $(\BD, \PP) \coloneqq  \{B_p \mid p\in\PP\}$ of mutually disjoint, non-empty 
    \underline{isolating blocks} is called a \emph{block decomposition}\index{block decomposition} of $S\subset X$
    in $\cV$ if there exists a partial order $(\PP,\leq)$ such that
    \begin{enumerate}[label=(B\arabic*)]
        \item\label{it:block_decomposition_uims}
            for every $\varphi\in\esolV(S)$
            there exist $p,q\in\PP$ such that $\uimm\varphi\subset \bl_p$ and $\uimp\varphi\subset \bl_q$, 
        \item\label{it:block_decomposition_paths} 
            if there exists $\rho\in\pathsV(S)$ with $\pbeg\rho\in \bl_p$ and $\pend\rho\in \bl_q$ for some $p,q\in\PP$ then 
                either 
                \begin{enumerate}[label=(\alph*)]
                    \item\label{it:block_decomposition_paths_le} $p>q$, or
                    \item\label{it:block_decomposition_paths_eq} $p=q$ and $\im\rho\subset \bl_p$. 
                \end{enumerate}
    \end{enumerate}
    If additionally
    \begin{enumerate}[label=(B\arabic*)]
          \setcounter{enumi}{2}
        \item\label{it:block_decomposition_partition}
            $S=\bigcup_{p\in\PP} B_p$,
    \end{enumerate}
    then we call $(\BD,\PP)$ a \emph{block partition}\index{block partition} of $S$.
\end{definition}

\begin{remark}\label{rem:MD_is_BD}
    By Corollary \ref{cor:iso_inv_is_iso_block} every Morse decomposition is also a block decomposition.
\end{remark}

The concept of block decomposition was introduced in \cite{MroWan2025}; 
    however the authors refer to its elements as simply blocks, not (combinatorial) isolating blocks.
Also note, that in various works on multivector fields, 
    the definition of the Morse decomposition varies depending on context.
In particular, Definition \ref{def:morse_decomposition} coincides with the one introduced in \cite{LKMW2022, LiMiMr2025}.
However, in papers focused on algorithms or computations \cite{DJKKLM2019, DeLiMrSl2024, DeLiHa2026}, ``Morse decomposition'' refers to what we call a block partition.
With the distinction between Morse and block decompositions we intend to avoid this ambiguity.

We write $\cM$ or $\cB$ for a Morse or a block decomposition, respectively, omitting $\PP$ when it can be deduced from the context.
We call the minimal order satisfying~\ref{it:morse_decomposition_paths} or \ref{it:block_decomposition_paths} the \emph{flow induced order}\index{flow induced order}.
Any extension of the partial order $(\PP,\leq)$ to a linear order is called an \emph{admissible linear order}\index{linear order!admissible} for $\cM$ or $\BD$.

There is a simple correspondence between block and Morse decompositions. Namely, as we prove in the next proposition, every $\BD$ induces the Morse decomposition 
\begin{align}\label{eq:induced-morse-decomposition}
    \indMDV{\BD}{\cV}&\coloneqq \{\inv_\cV(B_p) \mid p\in\PP \text{ such that } \inv_\cV(B_p)\neq\emptyset\}.
\end{align}
We omit $\cV$ when it is clear from the context and write $\indMD{\BD}$.
We say that a block decomposition $\BD$ \emph{covers}\index{block decomposition!covering} Morse decomposition $\MD$ if $\BDmd=\MD$.
By Remark~\ref{rem:MD_is_BD}, every Morse decomposition admits at least one (trivial) covering block decomposition. 

\begin{proposition}\label{prop:block_to_Morse_decomposition}
    Let $(\BD,\PP)$ be a block decomposition of $S$.
    Then $\indMD{\BD}$ is a Morse decomposition of $S$.
\end{proposition}
\begin{proof}
    Let $\varphi\in\esolV(S)$.
    By \ref{it:block_decomposition_uims}, there exist $p,q\in\PP$ such that $\uimm\varphi\subset \bl_p$ and $\uimp\varphi\subset \bl_q$.
    As a consequence of \cite[Proposition 6.5]{LKMW2022}, $\inv_\cV(\uimpm\varphi)=\uimpm\varphi$;
        therefore, we conclude \ref{it:morse_decomposition_uims}, because of
    \begin{align*}
        \inv_\cV(\uimm\varphi)\subset \inv_\cV(B_p)\in\indMD{\BD} 
            \text{\quad and \quad}
        \inv_\cV(\uimp\varphi)\subset \inv_\cV(B_q)\in\indMD{\BD}.
    \end{align*}

    Consider $p,q\in\PP$ and denote $M_p\coloneqq \inv_\cV\bl_p$ and $M_q\coloneqq \inv_\cV\bl_q$.
    Assume that there exists a path $\rho\in\pathsV(M_p, M_q, S)$.
    Clearly $M_p\subset \bl_p$ and $M_q\subset \bl_q$.
    Hence, if $p\neq q$ and $p>q$ in $\BD$ then $p\neq q$ and $p>q$ in the induced order on $\indMD{\BD}$.
    If $p=q$ then, in particular exist essential solutions 
        $\psi\in\esolV(\pbeg\rho, M_p)$ and $\psi'\in\esolV(\pend\rho, M_p)$.
    Let $\psi_+$ denotes the restriction of $\psi$ to $\zint{+\infty}$ and
        $\psi_-'$, the restriction of $\psi'$ to $\zintab{-\infty}{0}$.
    It is easy to verify that $\varphi\coloneqq \psi_-'\cdot\rho\cdot\psi_+$ is an essential solution in $\bl_p$.
    Therefore, $\im\rho\subset\im\varphi\subset\inv_\cV M_p$,
    which shows \ref{it:morse_decomposition_paths} and concludes the proof.
\end{proof}

\begin{example}\label{ex:first-example-MD}
    The four isolating blocks in Figure \ref{fig:first-example-BD} (left) form a block decomposition $\BD=\{B_1,B_2,B_3,B_4\}$ of $K$ for multivector field~$\cV$,
        and the three isolated invariant sets shown in the right panel form the Morse decomposition 
        $\MD=\{M_1,M_2,M_3\}$. 
    In particular, $\BD$ covers $\MD$, that is $\BDmd=\MD$.
    The graph in the middle of Figure \ref{fig:first-example-BD} shows the flow induced partial order on $\PP$.
    \qedex
\end{example}

A Morse decomposition $\cM$ of $S$ is said to be the \emph{finest}\index{Morse decomposition!finest} if 
    for any other Morse decomposition $\cM'$ of $S$ we have $\MD\inscr\MD'$.
We define the \emph{finest block decomposition}\index{block decomposition!finest} and the 
\emph{finest block partition}\index{block partition!finest} analogously.
In particular, the finest Morse decomposition and the finest block decomposition coincide.

Finally, we evoke the result showing that the finest block partition, which always exists, can be easily obtained from $\dgV$, which in turn, gives the finest Morse decomposition.
The following theorem is a direct consequence of the proof of \cite[Theorem~7.3]{LKMW2022} and Proposition \ref{prop:block_to_Morse_decomposition}.
\begin{theorem}[Decomposition by \scc]\label{thm:scc_finest_BD}
    Let $\cV$ be a multivector field on $X$.
    Then the family of strongly connected components of the graph $\dgV$ forms the finest block partition $\BD$ of $X$ with respect to $\cV$. 
    Moreover, $\indMD{\BD}$ is the finest Morse decomposition of $X$.
\end{theorem}

\begin{example}\label{ex:first-example-BP}
    Figure \ref{fig:first-example-BP} illustrates another block and Morse decomposition for the multivector field $\cV$ from Example \ref{ex:first-example-MVF}.
    Here, we have block decomposition $\BD'=\{B'_1, B'_2, B'_3, B'_4, B'_5, B'_6, B'_7\}$.
    In fact, $\BD'$ is the finest block partition of $X$ 
        (although, not the finest block decomposition, which would consists of $\{B'_1, B'_3, B'_4, B'_7\}$), because the union of blocks in $\BD'$ gives the entire complex, 
        and there is no room for a further refinement.
    Note, that breaking $B'_1$ into smaller pieces would violate both conditions \ref{it:block_decomposition_uims} and \ref{it:block_decomposition_paths}.
    The Morse decomposition $\MD'=\{M'_1,M'_2,M'_3,M'_4\}$ is the finest Morse (and block) decomposition and it is covered by~$\BD'$, that is $\cM'=\indMDV{\cB'}{}$.
    In particular, compared to $\MD$ from Example~\ref{ex:first-example-MD}, we have $\MD'\inscr\MD$.
    Note, that in case of $\BD$ and $\BD'$ we have neither $\BD\inscr\BD'$ nor $\BD'\inscr\BD$.
    \qedex
\end{example}
\begin{figure}
    \centering
    \includegraphics[width=0.49\linewidth]{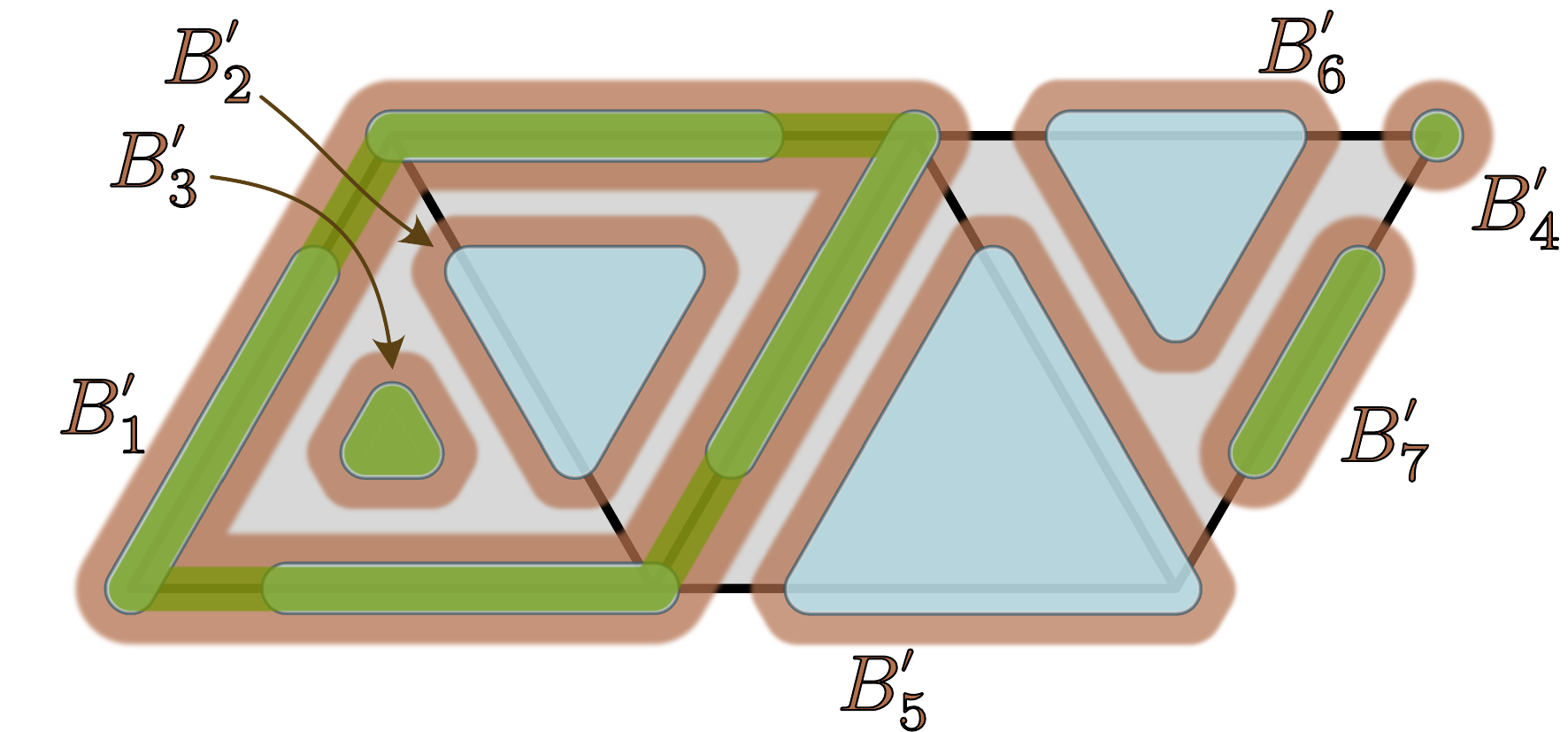}
    \includegraphics[width=0.49\linewidth]{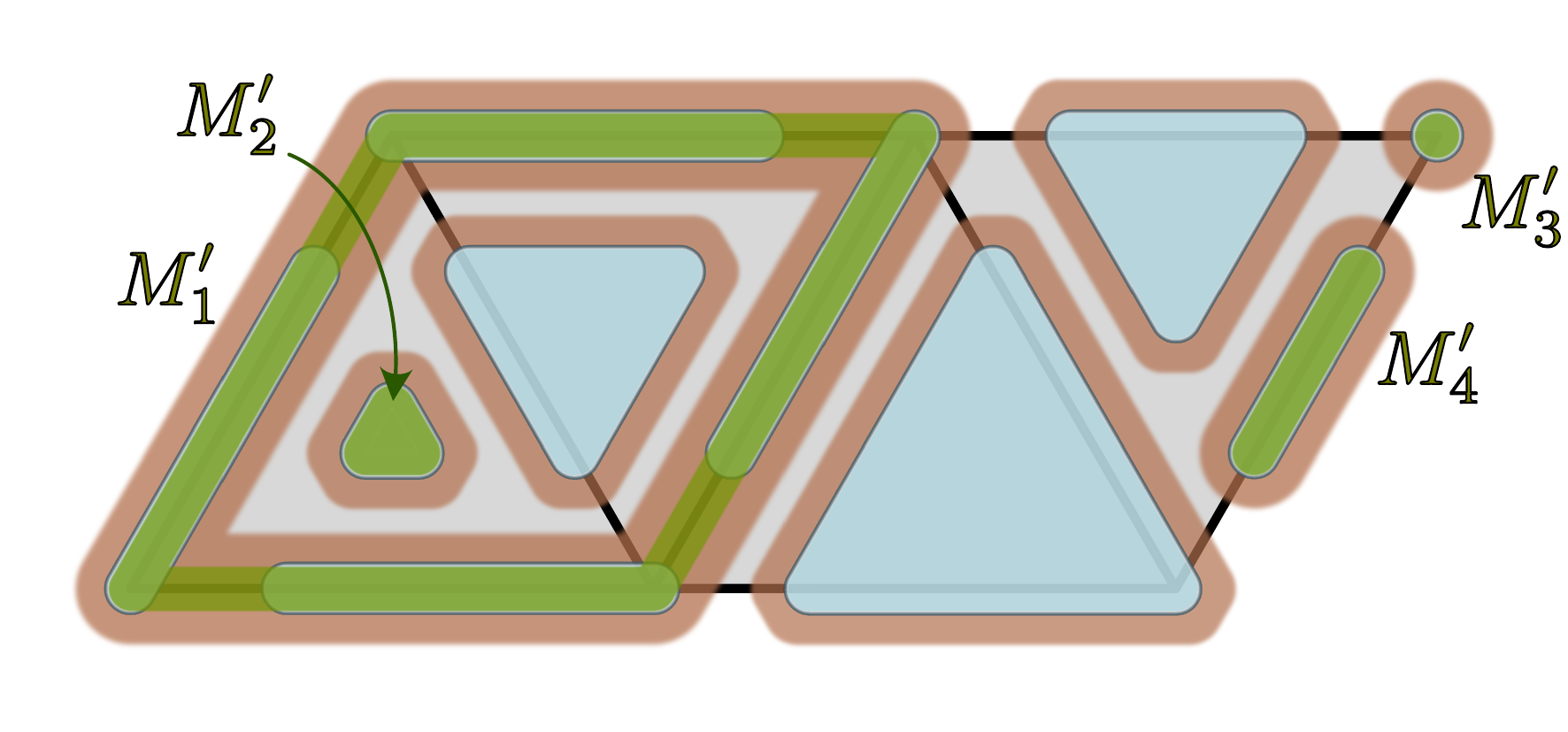}
    \caption{Example of the finest block partition (brown sets in the left panel) and 
        the corresponding finest Morse decomposition (green sets in the right panel)
        for a multivector field $\cV$ from Figure~\ref{fig:first-example-MVF}.
    }
    \label{fig:first-example-BP}
\end{figure}

Finally, we introduce a special type of Morse and block decomposition, namely the attractor-repeller pair.
An isolated invariant set $A$ is an \emph{attractor}\index{attractor} (relative) to $S$ if $\mvm(A)\cap S=A$.
Equivalently, an isolated invariant set is an attractor in $S$ if it is closed in $S$, that is $S\cap\cl A=A$ \cite[Theorem~6.2]{LKMW2022}.
Analogously an isolated invariant set $R$ is a \emph{repeller}\index{repeller} (relative) to $S$ if $\mvm^{-1}(R)\cap S=R$.
Equivalently, an isolated invariant set $R$ is a repeller in $S$ if it is open in $S$, that is $S\cap\opn R=R$ \cite[Theorem~6.3]{LKMW2022}.

\begin{definition}[Attractor-repeller pair]
    Let $S$ be an isolated invariant set.
    Then a pair of isolated invariant sets $A, R\subset S$ is 
        an \emph{attractor-repeller pair}\index{attractor-repeller pair} (or an \emph{AR-pair}\index{AR-pair|seealso{attractor-repeller pair}}) if 
        $A$ is an attractor in $S$ and $R=\inv_\cV S\setminus A$.
    In particular, $R$ is called the \emph{dual repeller}\index{repeller!dual} in $S$.
    It is easy to see that an AR-pair $(A,R)$ of $S$ form a Morse decomposition of $S$, 
        which we also call an \emph{AR-decomposition}\index{AR-decomposition}.

    For our purposes we allow $A$ and/or $R$ to be empty.
\end{definition}

\begin{proposition}\label{prop:AR-block-decomposition}
    Let $\BD=\{\bl_a, \bl_r\}$ be a two-element block decomposition of an isolating block $B$ such that $r\not<a$.
    Then $(M_a,M_r)\coloneqq(\inv_\cV B_a, \inv_\cV B_r)$ is an AR-pair for $S\coloneqq\inv_\cV B$.
\end{proposition}
\begin{proof}
    By Proposition~\ref{prop:block_to_Morse_decomposition}, the pair $\cM\coloneqq\{M_a,M_r\}$ is a Morse decomposition of $S$ as long as both are not empty.
    However, if $M_a$ and/or $M_b$ is empty the below argument remains virtually the same.

    To show that $M_a$ is an attractor, suppose that there exists an $x\in(\mvm(M_a)\cap S)\setminus M_a$.
    Since $S$ is invariant there exists $\varphi\in\esolV(x,S)$.
    We also have $p\in\{a,r\}$ such that $\uimp(\varphi)\subset M_p$, because $\cM$ is a Morse decomposition of $S$.
    If $p=a$ then we can take $y\in M_a$ such that $x\in\mvm(y)$ and $\rho\subset\varphi$ such that $\pbeg{\rho}=x$ and $\pend{\rho}\in M_a$;
        but then, path $y\cdot\rho$ implies $x\in\im\rho\subset M_a$ contradicting \ref{it:morse_decomposition_paths}\ref{it:morse_decomposition_paths_eq}.
    Putting $p=r$, we can construct a similar path, but with $\pend\rho\in M_r$, 
        then $y\cdot\rho$ and \ref{it:morse_decomposition_paths}\ref{it:morse_decomposition_paths_le} imply $r< a$, again a contradiction.
    Hence, $M_a$ is an attractor in $S$.

    To show that $\inv_\cV(S\setminus A)=M_r$ we notice first that $\bl_r\subset\bl\setminus A$ immediately implies 
    $M_r=\inv_\cV{\bl_r}\subset\inv_\cV{\bl\setminus A}$.
    To see the other inclusion consider $\varphi\in\esolV(S\setminus M_a)$.
    Since $\cM$ is a Morse decomposition we have $\uimpm_\cV\varphi\subset M_r$, thus $\im\varphi\subset M_r$ by \ref{it:morse_decomposition_paths}\ref{it:morse_decomposition_paths_eq}.
\end{proof}

\subsection{Conley Index}\label{subsec:conley-index}
\begin{definition}(Index pair) \cite[Definition~5.1]{LKMW2022}\label{def:idxpair}
A pair of closed sets $(P,E)$ such that $E\subset P$ is an \emph{index pair}\index{index pair} for the isolated invariant set $S$ if the following conditions hold:
    \begin{enumerate}[label=(IP\arabic*)]
    \item\label{it:idxpair_PE} $F_\cV(P\setminus E)\subset P$ (the exit set condition),
    \item\label{it:idxpair_E} $F_\cV(E)\cap P\subset E$ (the positive invariance condition),
    \item\label{it:idxpair_inv} $S = \inv_\cV(P\setminus E)$ (the invariant part condition).
\end{enumerate}
\end{definition}

It is easy to verify that $P\setminus E$ is an isolating block.
When we say that $(P,E)$ is an index pair for an isolating block $B$ we mean that $(P,E)$ is an index pair for $\inv B$ and $B\subset P\setminus E$.
The following proposition is an immediate consequence of Proposition \ref{prop:iso_block_is_lcl_Vcomp} and \cite[Proposition 9]{DeLiMrSl2022}.

\begin{proposition}
\label{prop:iso_block_forms_ipair}
    Let $\bl$ be an isolating block in $\cV$. Then $(\cl\bl, \mo\bl)$ is an index pair for $\inv_\cV\bl$.
    In particular, by Corollary \ref{cor:iso_inv_is_iso_block}, $(\cl S, \mo S)$ is the minimal index pair for an isolated invariant set $S$.
\end{proposition}

\begin{proposition}\label{prop:invPE}
    Let $(P,E)$ be a pair of closed sets such that $E\subset P$ 
    and $P\setminus E$ is $\cV$-compatible.
    Then $(P,E)$ is an index pair for $S\coloneqq \inv_\cV (P\setminus E)$.
\end{proposition}
\begin{proof}
    To see \ref{it:idxpair_PE}:
    let $x\in P\setminus E$ and $y\in F_\cV(x)=[x]_\cV\cup\cl x$.
    If $y\in [x]_\cV$ then $y\in P$, because $P\setminus E$ is $\cV$-compatible.
    If $y\in\cl x$ then $y\in P$, because $P$ is closed.
    Therefore, $y\in P$.
    To see \ref{it:idxpair_E}:
    let $x\in E$ and $y\in F_\cV(x)$.
    If $y\in [x]_\cV \cap P$ then necessarily $y\in E$ because $P\setminus E$ is $\cV$-compatible and $x\not\in P\setminus E$. 
    If $y\in\cl x$ then $y\in E$, because $E$ is closed.
    Therefore, $y\in E$.
    Condition \ref{it:idxpair_inv} is given by the assumption.
\end{proof}

\begin{theorem}\cite[Theorem 5.16]{LKMW2022}\label{thm:con_idx_well_defined}
    Let $(P_1,E_1)$ and $(P_2,E_2)$ be two index pairs for an isolated invariant set $S$ in $\cV$. 
    Then $H(P_1,E_1)\cong H(P_2,E_2)$.
\end{theorem}

\begin{definition} (Conley index)\label{def:conley_index}
\cite[Section 5.2]{LKMW2022}
    The \emph{Conley index}\index{Conley index} of an isolated invariant set $S$ is defined as $\con(S)\coloneqq [H_0(P,E), H_1(P,E), \ldots]$, where $\ipair{}$ is an index pair for $S$ and the relative homology is calculated over the field $k$.
    We also denote each component of $\con(S)$ as $\con_i(S) \coloneqq H_i(P,E)$.
\end{definition}
Note that the Conley index is well defined due to Theorem \ref{thm:con_idx_well_defined};
and that every $\con_i(S)$ is isomorphic to the vector space $k^{d_i}$, where $d_i = \dim H_i(P,E)$.

\begin{example}
    Conley indices for the Morse sets in $\MD$ from Example \ref{ex:first-example-MD} are as follows: 
    $\con(M_1)=[\fk,\fk,0]$, 
    $\con(M_2)=[0,0,\fk]$, and $\con(M_3)=[0,0,0]$.
    In case of $\MD'$ from Example \ref{ex:first-example-BP} we have 
        $\con(M'_3)=[\fk, 0,0]$ and $\con(M'_4)=[0,\fk, 0]$.
    Morse sets $M'_1$ and $M'_2$ coincide with $M_1$ and $M_2$.
    \qedex
\end{example}

Theorem~\ref{thm:con_idx_well_defined} can be proved by constructing a sequence of index pairs related by isomorphisms.
We show such a sequence explicitly, as it will come in handy later. 
Moreover, the construction already hints at connections with persistence theory.
To do that, let us recall Theorem~\ref{thm:ipair-inclusion-isomorphism} and the notion of push forward.
\begin{theorem}\cite[Theorem~28]{DeLiMrSl2022}\label{thm:ipair-inclusion-isomorphism}
    If $(P,E)$ and $(P',E')$ are index pairs for $S$ in $\cV$ such that $(P,E)\subset (P',E')$ then the inclusion induces an isomorphism in homology.
\end{theorem}
The \emph{push forward}\index{push forward} of a set $A$ in $Y\subset X$ with respect to~$\cV$ is defined by:
\begin{equation}\label{eq:push_forward}
    \pf_\cV(A, Y) \coloneqq  \{x\in Y \mid \exists_{a\in A}\, \exists_{\rho\in\paths_\cV(Y)}\ \pbeg\rho=a \text{ and } \pend\rho = x \}.
\end{equation}
The crucial property of the push forward is captured by the following proposition, which can be easily proved by adapting the proof of \cite[Proposition 5.11]{LKMW2022}.
\begin{proposition}\label{prop:push_forward_closed_vcomp}
    Let $\bl$ be a locally closed, $\cV$-compatible subset of $X$.
    Then, for any set $A\subset\bl$, the push forward $\pf_\cV(A, \cl\bl)$ is closed.
    Moreover, $\pf_\cV(A,\bl)$ is locally closed and $\cV$-compatible.
\end{proposition}

Assume that $(P,E)$ and $(P',E')$ are two index pairs for $S$ in $\cV$.
Now take $\bl$, an isolating block for $S$, such that $\bl\subset P\setminus E$ and $\bl\subset P'\setminus E'$ (by Corollary~\ref{cor:iso_inv_is_iso_block} such a $\bl$ always exists, that is, $\bl=S$).
Then, the sequence
\begin{align}
    \begin{split}\label{eq:connection_ip_sequence}
    (P,E)\ \supset\ &(\pf_\cV(\cl \bl, P), E\cap\pf_\cV(\mo \bl, P))\ \\
        \subset\ &(\pf_\cV(\cl \bl, P), \pf_\cV(\mo \bl, P))\
        \supset\ (\cl \bl, \mo \bl)
    \end{split}
\end{align}
consists of index pairs (\cite[Theorems 10 and 15]{DeyMrozekSlechta2022}) and the inclusions induce isomorphisms in homology (Theorem \ref{thm:ipair-inclusion-isomorphism}).
Symmetrically, we construct another sequence connecting $(\cl \bl, \mo \bl)$ with $(P',E')$ and 
concatenate them to obtain a filtration from $(P,E)$ to $(P',E')$ as we schematically present in diagram~\eqref{eq:connecting_ip_sequence}, which we call the \emph{connecting sequence}\index{connecting sequence}. 
\begin{equation}
    \includegraphics[width=0.6\textwidth]{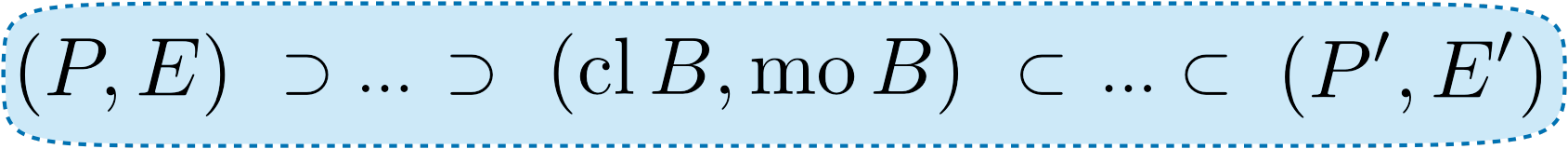}
        \vspace{0.25cm}
    \label{eq:connecting_ip_sequence}
\end{equation}

\subsection{Combinatorial continuation}
\label{subsec:combinatorial_continuation}
Let $\cV$ and $\cV'$ be two multivector fields on~$X$.
Whenever $\cV\sqsubseteq\cV'$ we say that $\cV$ is a \emph{refinement}\index{refinement} of $\cV'$, and symmetrically, $\cV'$ is a \emph{coarsening}\index{coarsening} of $\cV$.

We denote the collection of all possible multivector fields on $X$ by $\smvf(X)$.
The pair $(\smvf(X), \sqsubseteq)$ forms a partial order, and therefore, by Theorem \ref{thm:alexandrov_theorem}, we can interpret it as a finite topological space 
    where upper sets correspond to open sets.
In particular, the minimal open set containing $\cV\in\smvf(X)$ consists of all its refinements;
    we denote it $\opn_\sqsubseteq\cV$.\footnote{Note that $\cV\inscr\cV'$ implies that $\cV$ is higher in the poset (i.e. $\cV\geq\cV'$), which might appear counterintuitive. 
However, we follow the standard convention in the multivector fields literature, where open sets are identified with upper sets, 
    a choice that makes multivectors, when drawn on a simplicial complex,
        geometrically resemble continuous isolating blocks.
}

An example of the $\smvf(X)$ is presented in Figure~\ref{fig:smvf-example}.
For the sake of clarity, we restrict the example to multivector fields with multivectors that are connected.
The dynamical interpretation of a disconnected multivector is unclear, however none of our proofs require that property.

Note that in finite setting,
    the $\opn_\sqsubseteq\cV$ is the best possible approximation of an ``$\epsilon$-neighborhood''.
Therefore, we can think of a multivector field $\cV'\in\opn_\sqsubseteq\cV$ as a \emph{combinatorial perturbation}\index{combinatorial perturbation} of $\cV$.
Hence, with the use of Proposition~\ref{prop:iso_block_is_lcl_Vcomp}, 
    we mimic the stability of an isolating block present in the continuous theory of flows (see \cite[Proposition~1.1]{MischMro_Conley_2002}).

\begin{figure}[t]
    \centering
    \includegraphics[width=0.9\linewidth]{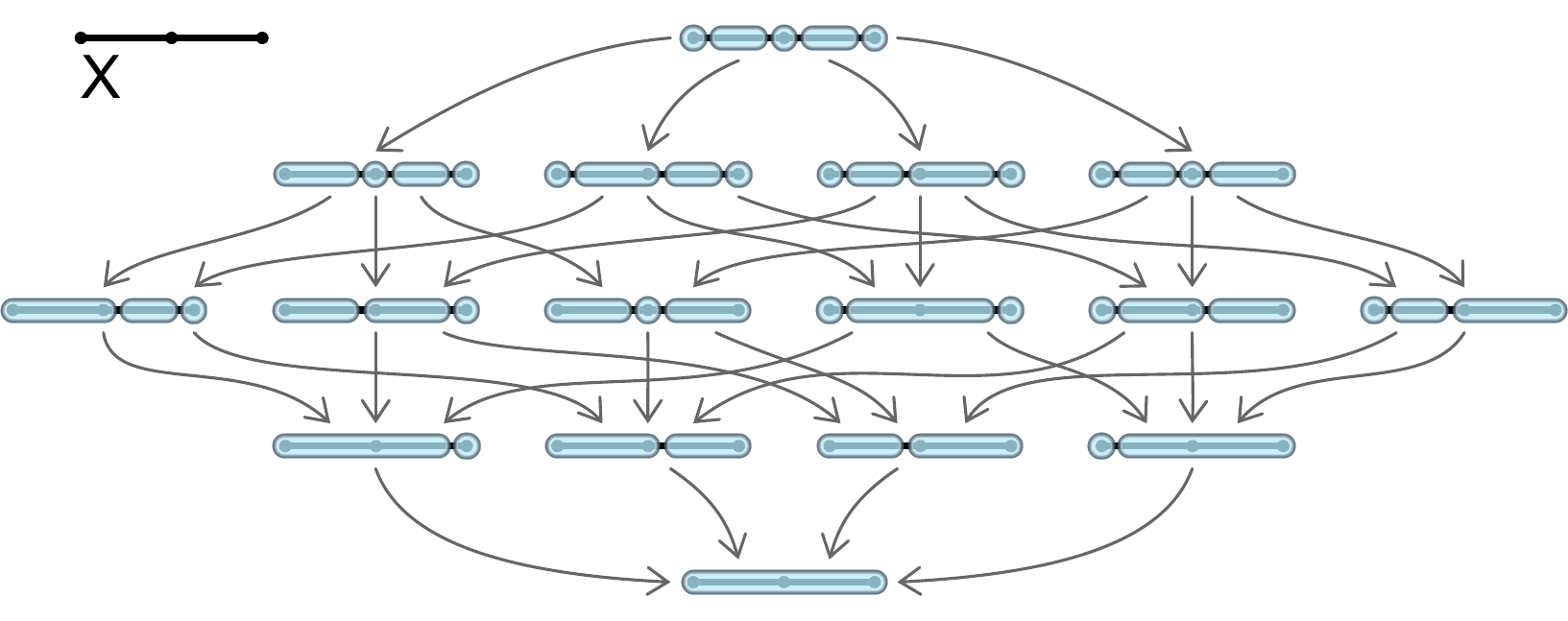}
    \caption{The poset of $(\smvf(X), \inscr)$ for simplicial complex~$X$.}
    \label{fig:smvf-example}
\end{figure}

\begin{proposition}[Stability of an isolating block]\label{prop:iso_block_stability}
    Let $\cV\in\smvf(X)$.
    An~isolating block $N$ for $\cV$ is also an isolating block for every $\cV'\in\opn_\sqsubseteq\cV$. 
\end{proposition}

We call a fence of multivector fields $\zzV=\{\cV_\tdyn\}_\tdyn=\mvfseq{\tdynmax}\subset \smvf(X)$, 
    that is a sequence such that 
    $\cV_{\tdyn}\sqsubseteq\cV_{\tdyn+1}$ or $\cV_{\tdyn}\sqsupseteq\cV_{\tdyn+1}$ for all $\tdyn\in\zint{\tdynmax-1}$,
    a (\emph{continuously}) \emph{parameterized combinatorial multivector field}\index{multivector field!parameterized}.
We leave it as an exercise to the reader that one can indeed construct a continuous map $\cV(\tdyn):[0,1]\rightarrow\smvf(X)$\footnote{Continuous with respect to the standard topology on the real interval, 
    and the finite topology on $\smvf(X)$ induced by relation $\inscr$.}
    generating the sequence.
This leads to the definition of a combinatorial continuation of an isolated invariant set.

\begin{definition}[Combinatorial continuation of an isolated invariant set]
\cite[Definition 10]{DeLiMrSl2022}
	\label{def:continuation}
    Let $\cV_0\inscr\cV_1$ or $\cV_0\ovscr\cV_1$.
    An isolated invariant set $S_0$ in $\cV_0$ \emph{continues} to $S_1$ in $\cV_1$ if there exists a set $B$, 
        which is an isolating block both in $\cV_0$ and $\cV_1$, 
        and $\inv_{\cV_0}B=S_0$ and $\inv_{\cV_1}B=S_1$.
\end{definition}

The above definition differs from the one introduced in \cite{DeLiMrSl2022}, but they are equivalent through Proposition~\ref{prop:iso_block_forms_ipair}.
Moreover, the concept of continuation can be easily extended to any parametrized multivector field $\zzV$.
In particular, $S_0$ in $\cV_0$ continues to $S_T$ in $\cV_T$ if there is a 
    sequence $S_0, S_1,\ldots,S_T$, where $S_\lambda$ is an isolated invariant set in $\cV_\lambda$ and 
    $S_\lambda$ continues to $S_{\lambda+1}$ for each $\lambda$. 
Intuitively, if an isolated invariant set continues to another, then they play the same qualitative role in the corresponding dynamical systems; 
    in particular, their Conley indices are isomorphic.

\begin{theorem}\cite[Theorem 22]{DeLiMrSl2022}
	\label{thm:continuation_iso_conley_index}
	If the isolated invariant set $S$ continues to $S'$ then $\con(S)\cong\con(S')$.
\end{theorem}

\section{Transition Diagram for a Zigzag Filtration of Block Decompositions}\label{sec:transition-diagram}

The notion of continuation identifies isolated invariant sets at different steps of a parameterization 
    and relates their Conley indices directly via isomorphisms.
In~\cite{DeLiMrSl2022}, we explored how persistence can be used to capture changes in a Conley index.
Here, we take the next step and develop a framework that allows us to track all Conley indices simultaneously, 
    providing additional insight into their mutual interactions and encoding the nature of these changes.

\subsection{Filtration of block decompositions}
\label{subsec:zigzag-filtration-BD}

Let $\zzV=\{\cV_\tdyn\}_{\tdyn\in\tInt}$
be a parameterized multivector field
    and 
    $\zzMD=\{(\MD_\tdyn,\cV_\tdyn)\}_{\tdyn\in\tInt}$
    be the corresponding Morse decompositions, such that, for each $\tdyn$, $\MD_\tdyn$ is a Morse decomposition for $\cV_\tdyn$.
To track the changes we require the existence of a sequence of covering block decompositions $\zzBD=\{(\BD_\tdyn,\cV_\tdyn)\}_{\tdyn\in\tInt}$ 
    (that is, $\BDmdt{\tdyn}=\MD_\tdyn$ for each $\tdyn\in\tInt$)
    forming a zigzag filtration as defined below.

\begin{definition}[Zigzag filtration of block decompositions]
    Let $\mvfseq{\tdynmax}$ be a parameterized multivector field on $X$.
    A sequence of pairs $\zzBD=\{(\BD_\tdyn,\cV_\tdyn)\}_{\tdyn\in\tInt}$, where $\tInt=\zint{\tdynmax}$, such that 
        $\BD_\tdyn$ is a block decomposition for $\cV_\tdyn$ is called a \emph{zigzag filtration of block decompositions}\footnote{
        Since elements of the sequence are families of sets, this is not a zigzag filtration in the standard sense.
        Nevertheless, we keep the name because the philosophy is analogous. 
        This can be seen as a higher-level form of filtration.
    }\index{zigzag filtration of block decompositions}
    (or simply a \emph{zigzag filtration} if it is clear from the context)
    if for all $\tdyn\in\zint{\tdynmax-1}$ either
    \begin{align*}
        \BD_{\tdyn}\inscr\BD_{\tdyn+1} 
            \quad\text{and}\quad
                \cV_{\tdyn}\inscr\cV_{\tdyn+1}
    \end{align*}
    or
    \begin{align*}
        \BD_{\tdyn+1}\inscr\BD_{\tdyn}
            \quad\text{and}\quad
                \cV_{\tdyn+1}\inscr\cV_{\tdyn},
    \end{align*}
    which we also denote as $(\BD_\tdyn, \cV_\tdyn)\inscr (\BD_{\tdyn+1}, \cV_{\tdyn+1})$ or
        $(\BD_\tdyn, \cV_\tdyn)\ovscr (\BD_{\tdyn+1}, \cV_{\tdyn+1})$, respectively.
   
   The sequence $\zzBD$ is called a \emph{filtration}\index{filtration of block decompositions} if all relations are in the same direction.
    We denote the indexing set corresponding to $\BD_\tdyn$ by $\PP_\tdyn$ and 
        an element of $\BD_\tdyn$ with index $p\in\PP_\tdyn$ by $\bl_{p,\tdyn}$.
    We usually write $M_{p,\tdyn}\coloneqq\inv_{\cV_{\tdyn}} \bl_{p,\tdyn}$. 
    If non-empty, $M_{p,\tdyn}$ is a Morse set in the corresponding Morse decomposition $\MD_\lambda$.
    However, one should keep in mind that this set might be empty.
\end{definition}

The canonical example of $\zzMD$ is a sequence of the finest Morse decompositions for the multivector fields in $\zzV$. 
The simplest strategy to build a corresponding zigzag filtration $\zzBD$ is to take the sequence of finest block partitions corresponding to~$\zzV$.
We use this canonical choice in our examples as it is also natural from an algorithmic perspective, but all presented results work for non-canonical zigzag filtrations as well.
Since a block decomposition carries all information about the underlying Morse decomposition (Proposition~\ref{prop:block_to_Morse_decomposition}), we focus mainly on $\zzBD$.

\begin{proposition}\label{prop:refinement_of_V_gives_inscribed_BDs}
	Let $\BD$ and  $\BD'$ be the finest block partitions for $\cV$ and $\cV'$, respectively. 
    If $\cV\in\opn_\inscr\cV'$ then $\BD\inscr\BD'$.
\end{proposition}
\begin{proof}
    Note that $G_\cV\subset G_{\cV'}$, that is, whenever $(x,y)$ is an edge in $G_\cV$ then it is in $G_{\cV'}$ as well.
    Therefore, the assertion follows directly from 
        Theorem~\ref{thm:scc_finest_BD}.
\end{proof}

\begin{example}\label{ex:main-example-part-1}
    Figure \ref{fig:main-example-MVFs-v2} shows a sequence forming a parameterized multivector field $\zzV\coloneqq\cV_0\ovscr\cV_1\inscr\cV_2\ovscr\cV_3\ovscr\cV_4$.
    The central column in Figure~\ref{fig:main-example-MVFs-BD-v2} illustrates the corresponding finest block partitions $\BD_0$, $\BD_1$, $\BD_2$, $\BD_3$, and $\BD_4$,
    which, by Proposition~\ref{prop:refinement_of_V_gives_inscribed_BDs}, form the following zigzag filtration of block decompositions: 
    \begin{align*}
        \zzBD\coloneqq
            (\BD_0, \cV_0)\ovscr(\BD_1, \cV_1)\inscr(\BD_2, \cV_2)\ovscr(\BD_3, \cV_3)\ovscr(\BD_4, \cV_4).
    \end{align*}
    The central column in Figure~\ref{fig:main-example-MVFs-BD-v2} presents the isolating blocks highlighted in brown, and the corresponding Morse sets highlighted in green.
    The block decompositions consist of 2, 3, 2, 3, and 7 isolating blocks, respectively, the associated Morse decompositions, defined as 
        $\MD_\tdyn\coloneqq \indMDV{{\BD_\tdyn}}{\cV_\tdyn}$ (see definition~\eqref{eq:induced-morse-decomposition})
        contain  $1$, $2$, $2$, $2$ and $3$ Morse sets, respectively.
    The corresponding flow induced partial orders are presented in Figure~\ref{fig:main-example-morse-graphs}.
    \qedex
\end{example}

\begin{figure}
    \centering
    \includegraphics[width=0.33\linewidth]{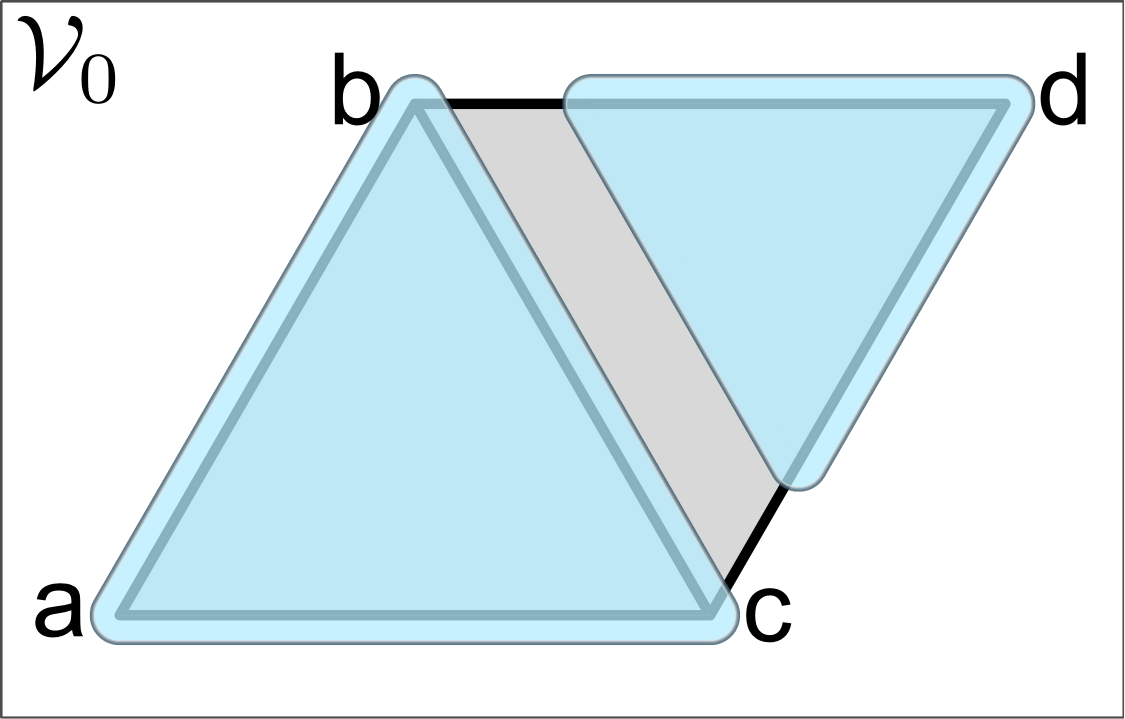}
    \hspace{-8pt}
    \includegraphics[width=0.33\linewidth]{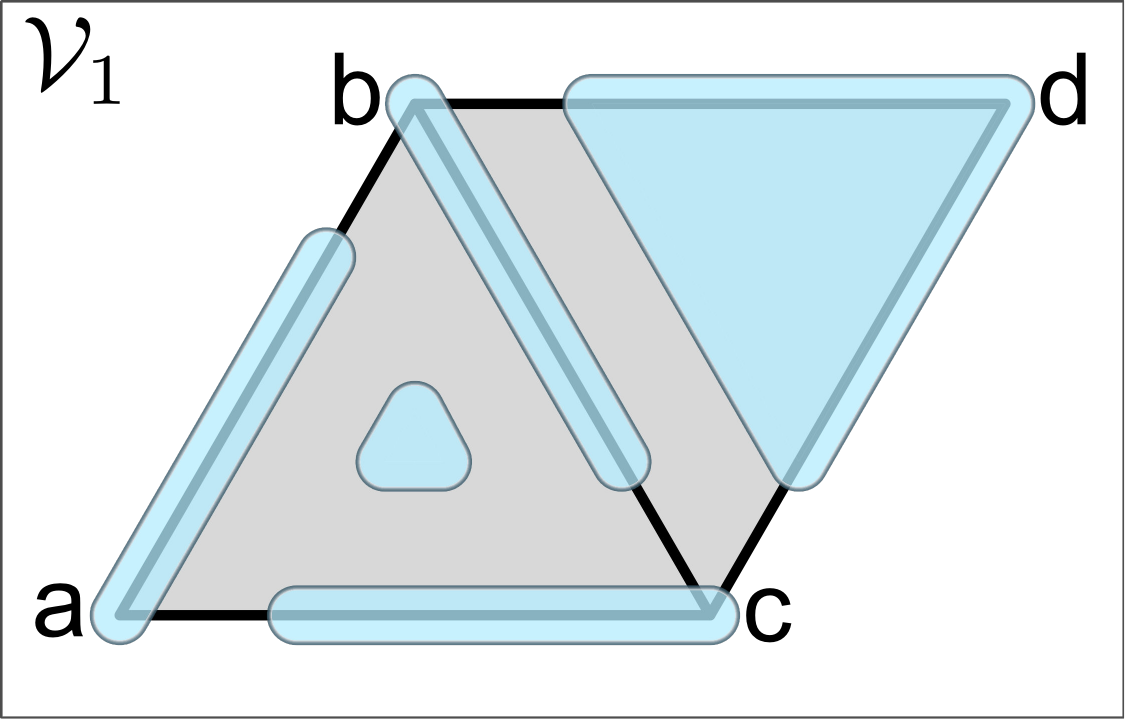}
    \hspace{-8pt}
    \includegraphics[width=0.33\linewidth]{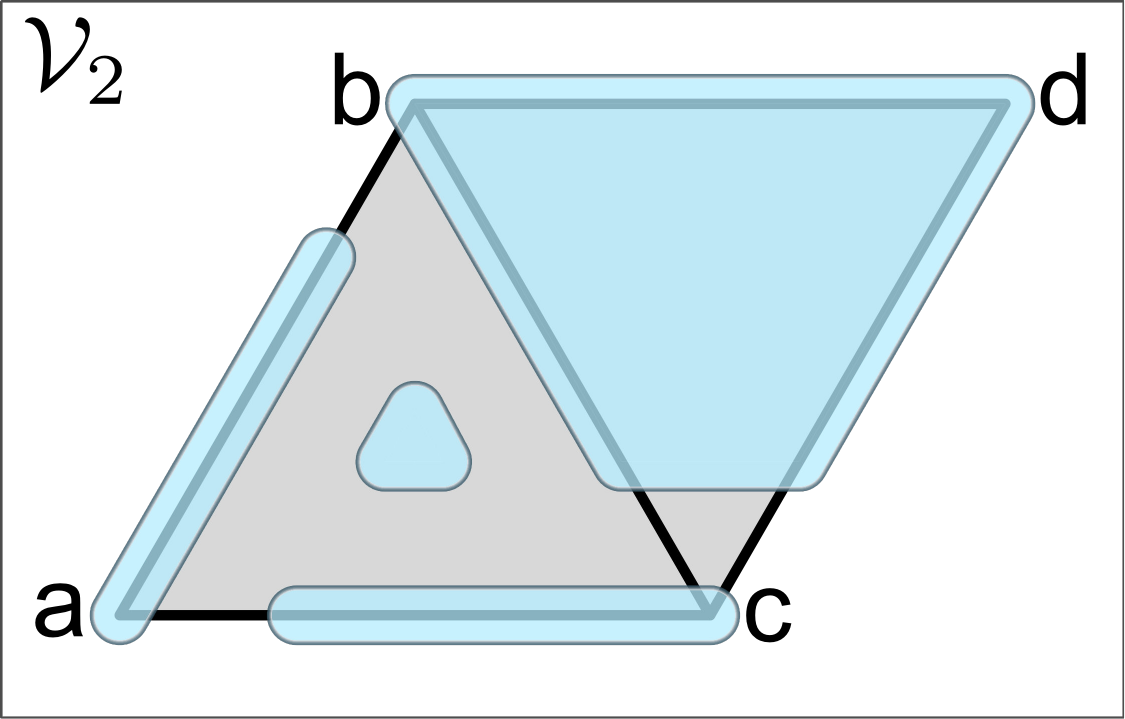}
    \vspace{-1.25pt}
    
    \includegraphics[width=0.33\linewidth]{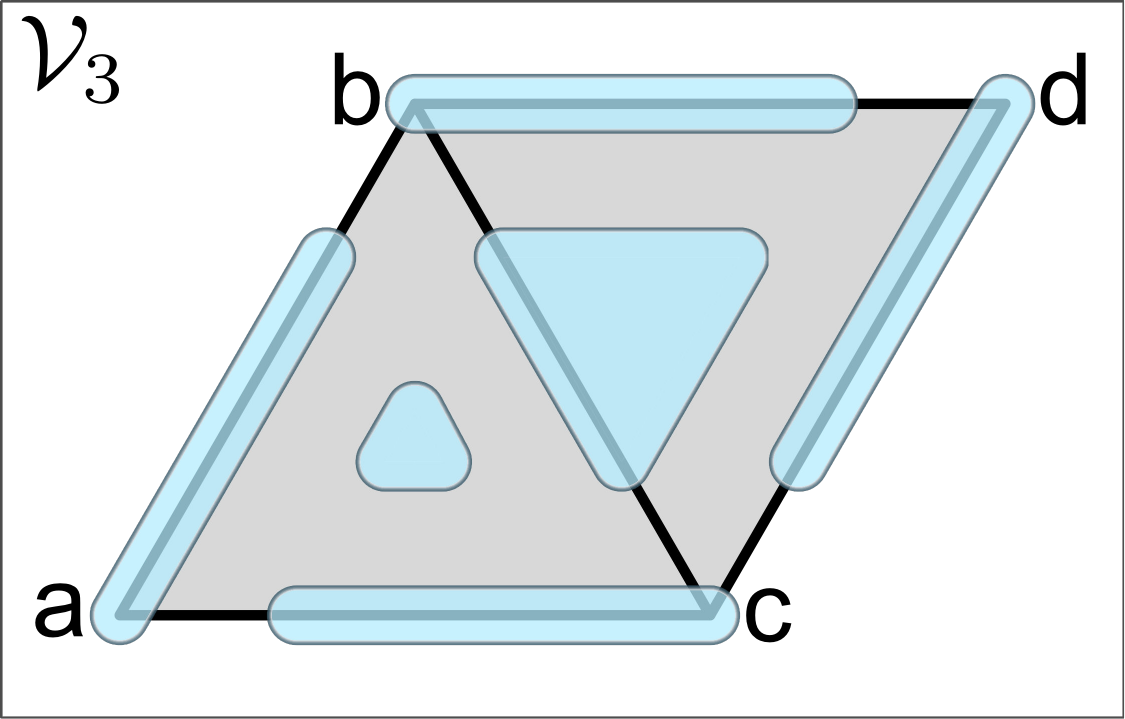}
    \hspace{-8pt}
    \includegraphics[width=0.33\linewidth]{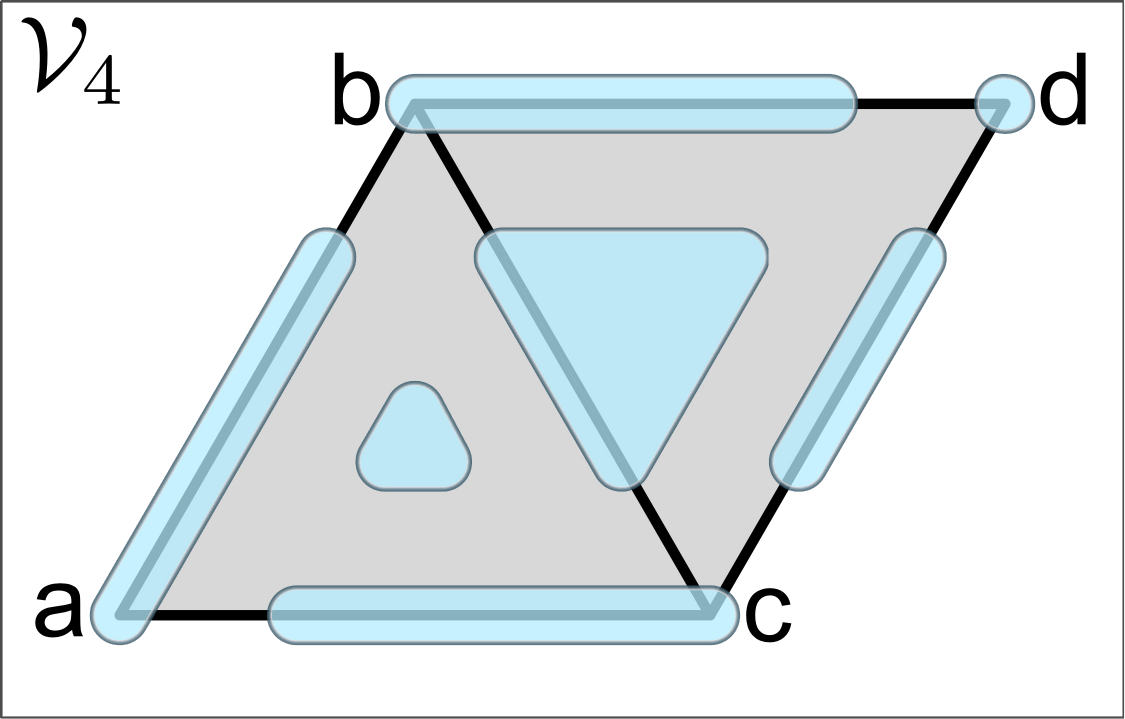}
    
    \caption{From top left to bottom right, multivector fields $\cV_0$, $\cV_1$, $\cV_2$, 
        $\cV_3$, and $\cV_4$ on a simplicial complex $K$.
    In particular, $\cV_0\ovscr\cV_1\inscr\cV_2\ovscr\cV_3\ovscr\cV_4$.
    }
    \label{fig:main-example-MVFs-v2}
\end{figure}
\begin{example}
    Another example of a parameterized multivector field was given in Section~\ref{subsec:sphere-example}.
    In Figure~\ref{fig:sphere-example-combinatorial-mvf}
        we have the sequence $\cV_0\ovscr\cV_1\inscr\cV_2\ovscr\cV_3\ovscr\cV_4$.
    The finest block partition of $\cV_0$ consists of 7 blocks, each formed by a single multivector.
    The finest block partition of $\cV_1$ consists of 8 isolating blocks, two formed by the 0-cells at the poles, 5 formed by gray regular multivectors, and the 8th one consists of the collection of 4 orange multivectors.
    The block partitions for the remaining steps should be also easy to find.
    In total we have a zigzag filtration of the following form:
    \[
           \quad 
           \qquad 
           (\BD_0, \cV_0)\ovscr(\BD_1, \cV_1)\inscr(\BD_2, \cV_2)\ovscr(\BD_3, \cV_3)\ovscr(\BD_4, \cV_4).
           \quad\qquad{\lozenge}
    \]
\end{example}

Whenever we have inscribed block decompositions $(\BD,\cV)\inscr(\BD',\cV')$ with corresponding index sets $\PP$ and $\PP'$, we can define the \emph{indexing map}\index{indexing map} 
$\idxmap{}:\PP \rightarrow\PP'$ such that $\idxmap{}(p)\coloneqq r$ if $\BD\ni\bl_{p}\subset \bl_{r}\in\BD'$.
We leave it as an easy exercise to the reader to verify that $\idxmap{}$ is an order preserving map between the flow induced partial orders.

In the case of zigzag filtration $\zzBD$, we distinguish two types of indexing maps:
    for $\BD_{\tdyn}\sqsubseteq\BD_{\tdyn+1}$ we have
    the \emph{$\tdyn$-forward map}\index{$\tdyn$-forward map}
    denoted and defined as:
    \[
        \idxfwd{\tdyn}:\PP_\tdyn\ni p \mapsto r\in\PP_{\tdyn+1}
    \]
    such that $\BD_\tdyn\ni\bl_{p,\tdyn}\subset\bl_{r,\tdyn+1}\in\BD_{\tdyn+1}$.
    Analogously, for $\BD_{\tdyn}\sqsupseteq\BD_{\tdyn+1}$ we have the \emph{$\tdyn$-backward map}\index{$\tdyn$-backward map}:
    \[
        \idxbck{\tdyn}:\PP_{\tdyn+1}\ni r \mapsto p\in\PP_{\tdyn}
    \]
    such that $\BD_{\tdyn+1}\ni\bl_{r,\tdyn+1}\subset \bl_{p, \tdyn}\in\BD_{\tdyn+1}$.
    Whenever we refer to $\zzBD$ we assume that the corresponding indexing sets $\PP_\tdyn$, and
        $\tdyn$-forward/backward maps, $\idxfwd{\tdyn}$ and $\idxbck{\tdyn}$, are implied.

\renewcommand{\arraystretch}{0.0}
\begin{figure}
    \begin{tabular}{c@{}c@{}c}
   &\includegraphics[width=0.29\linewidth]{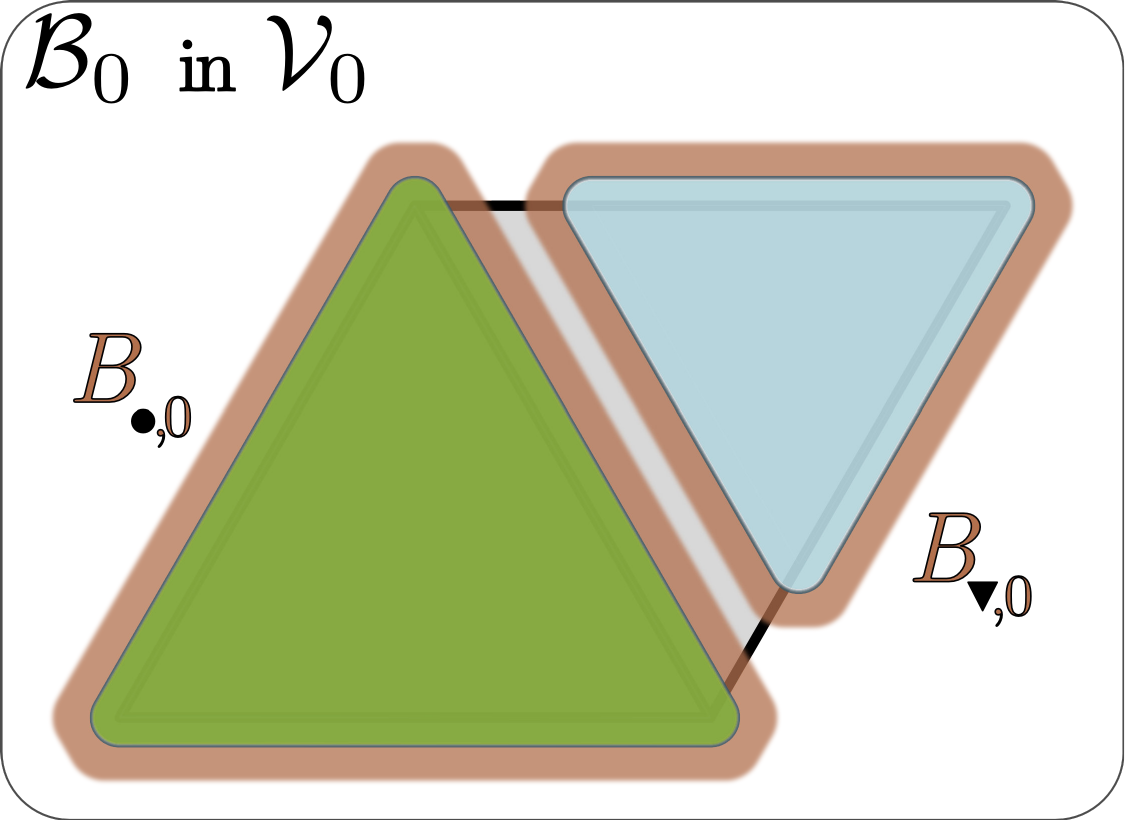}&\\
    \includegraphics[width=0.29\linewidth]{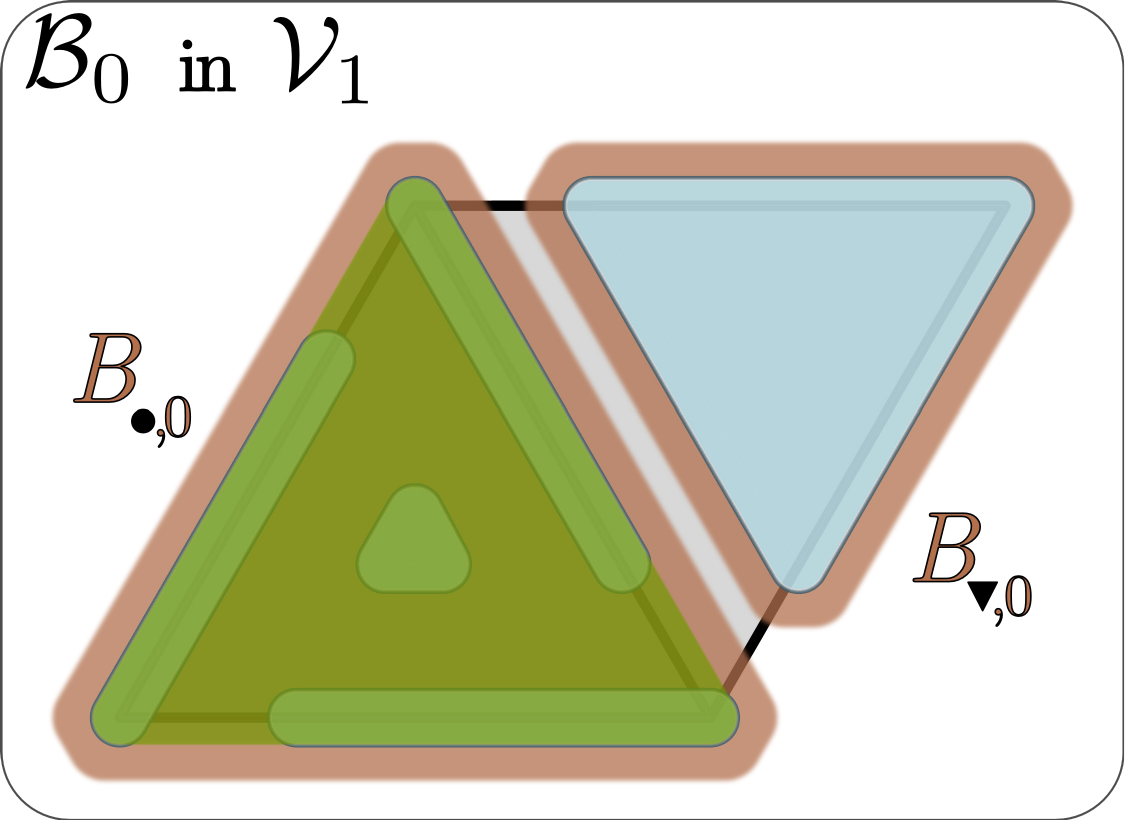}&
    \includegraphics[width=0.29\linewidth]{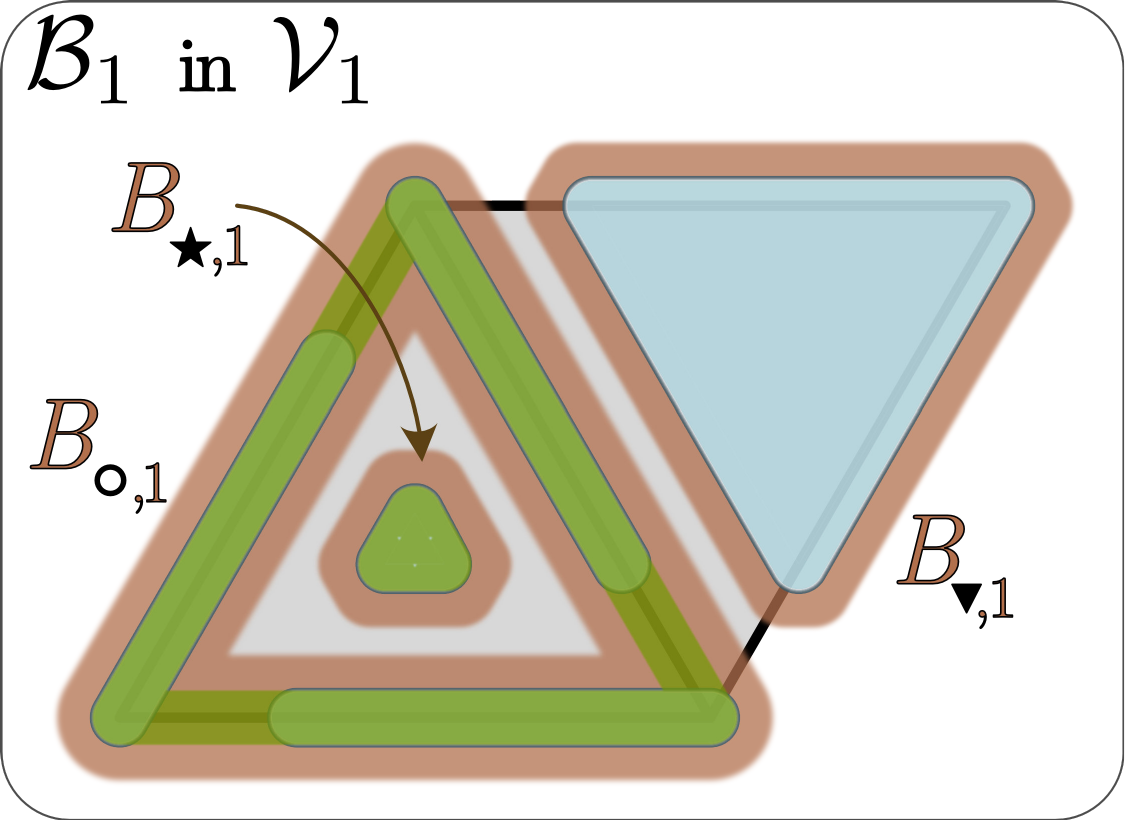}&
    \includegraphics[width=0.29\linewidth]{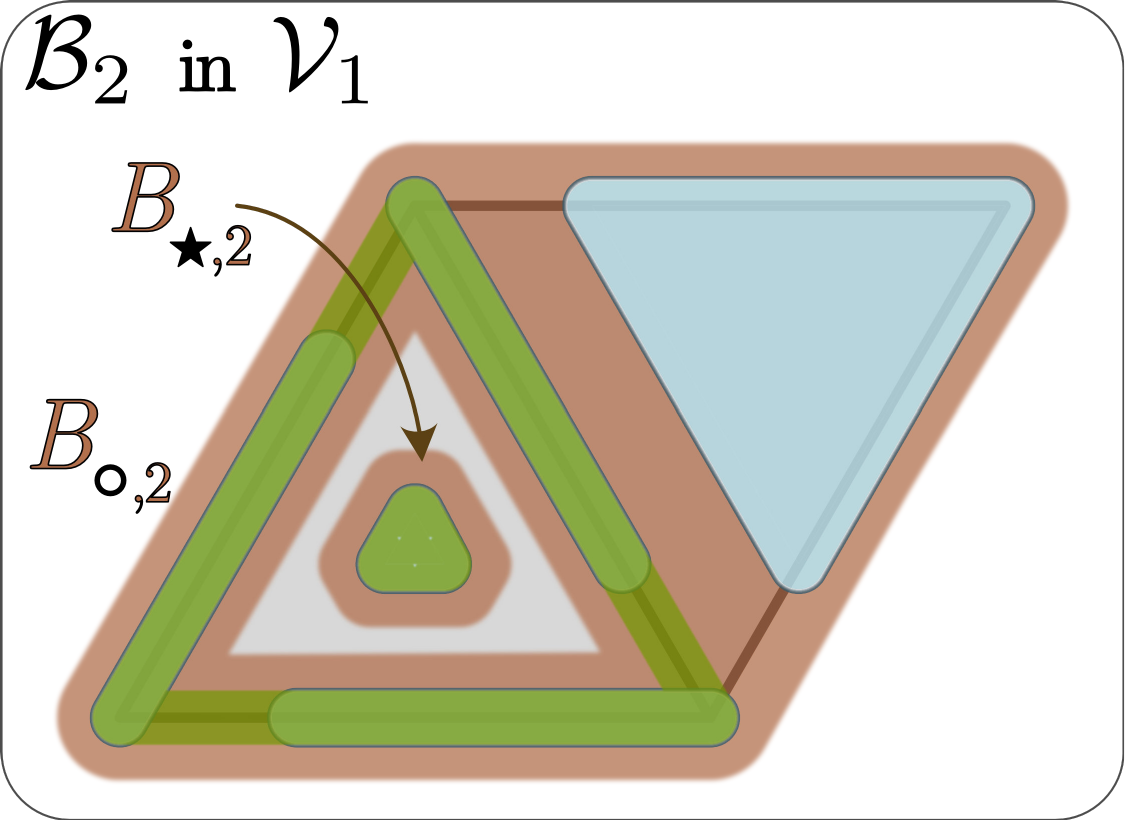}\\
   &\includegraphics[width=0.29\linewidth]{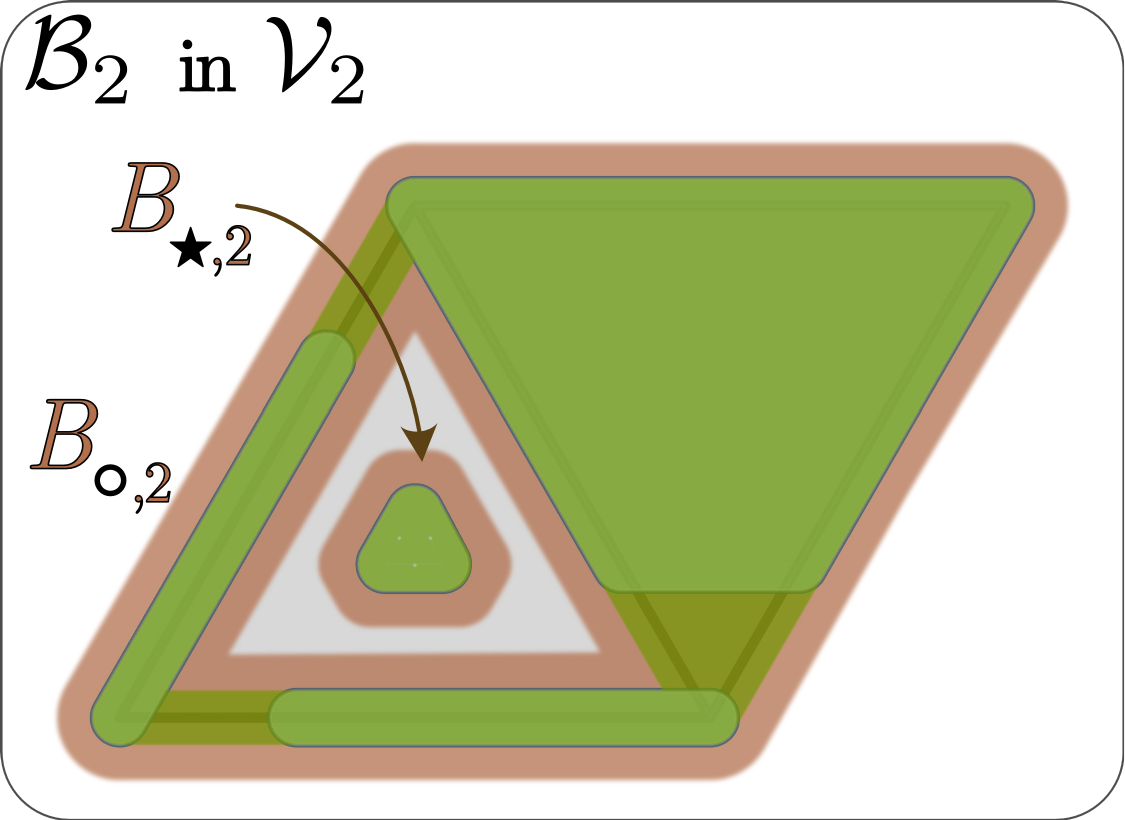}&\\
    \includegraphics[width=0.29\linewidth]{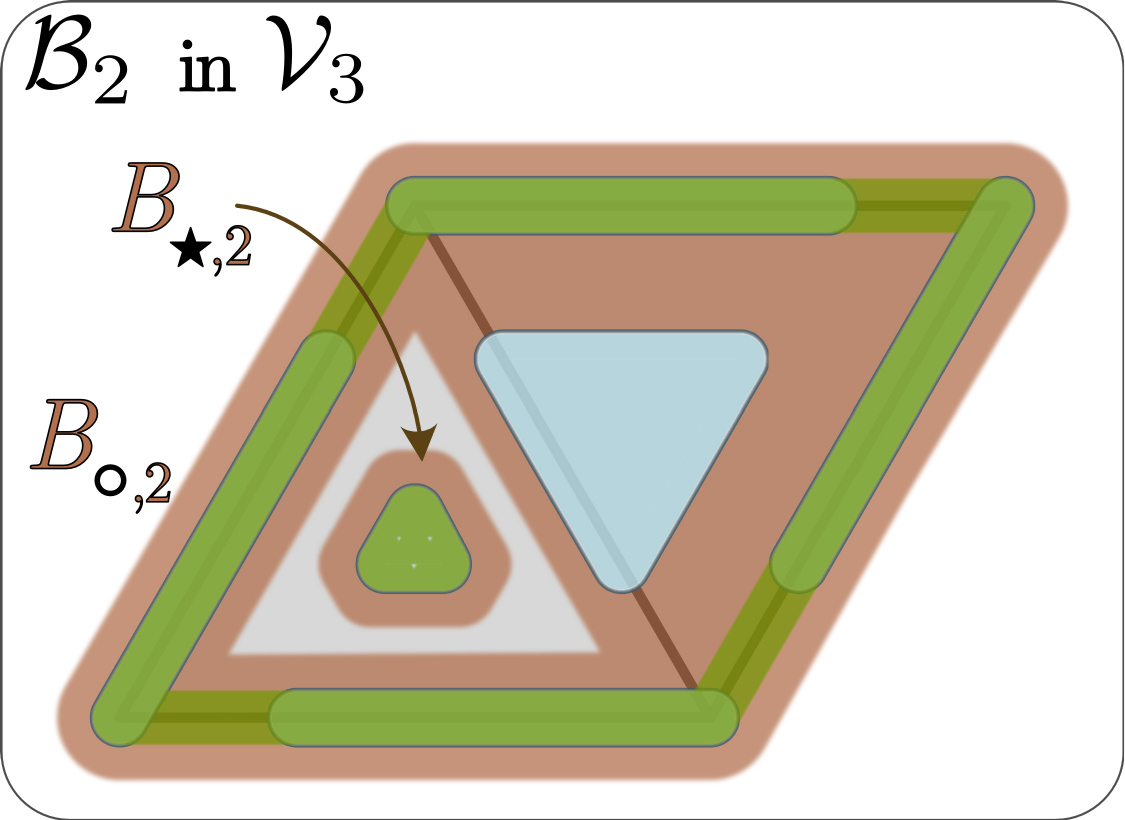}&
    \includegraphics[width=0.29\linewidth]{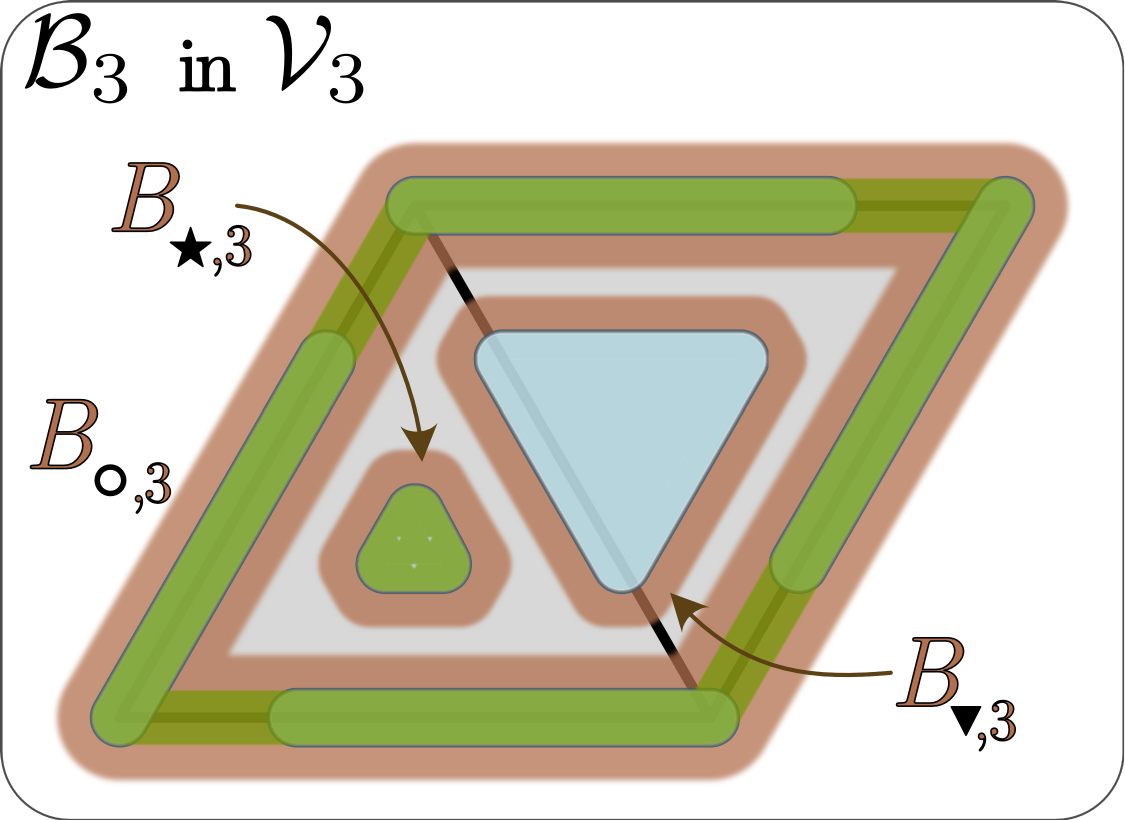}&\\
    \includegraphics[width=0.29\linewidth]{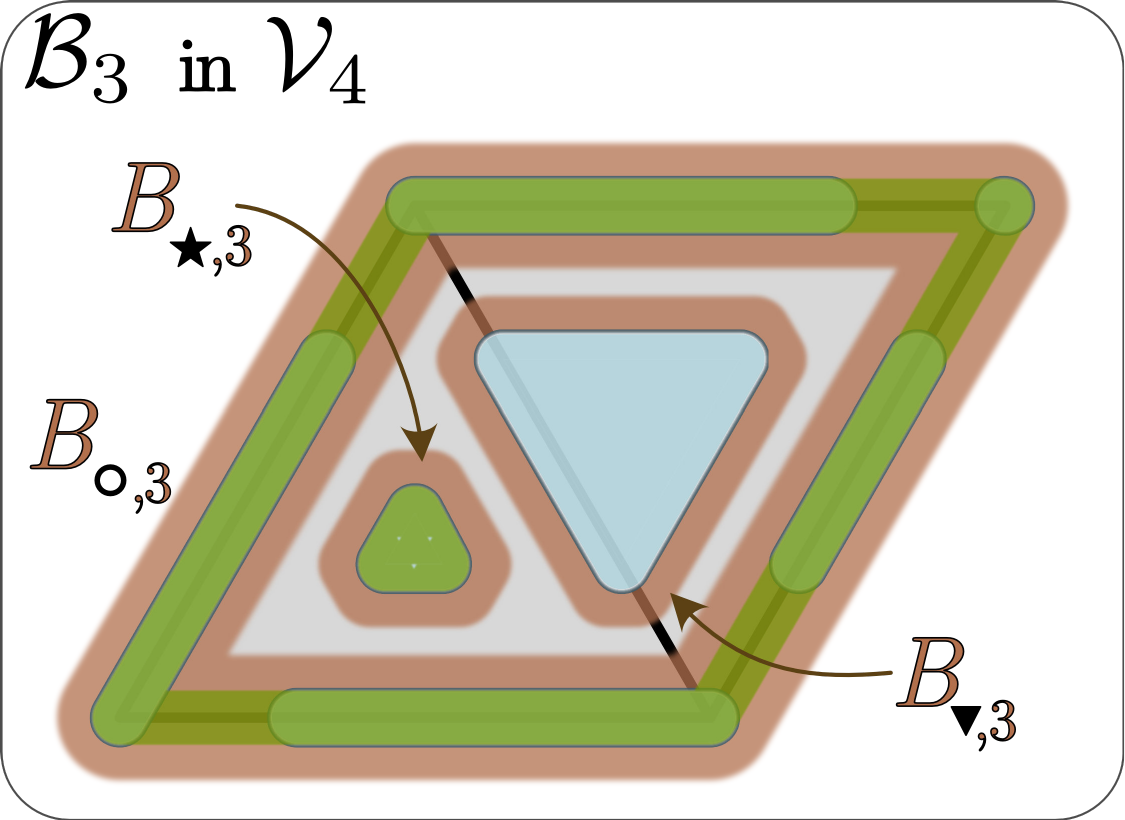}&
    \includegraphics[width=0.29\linewidth]{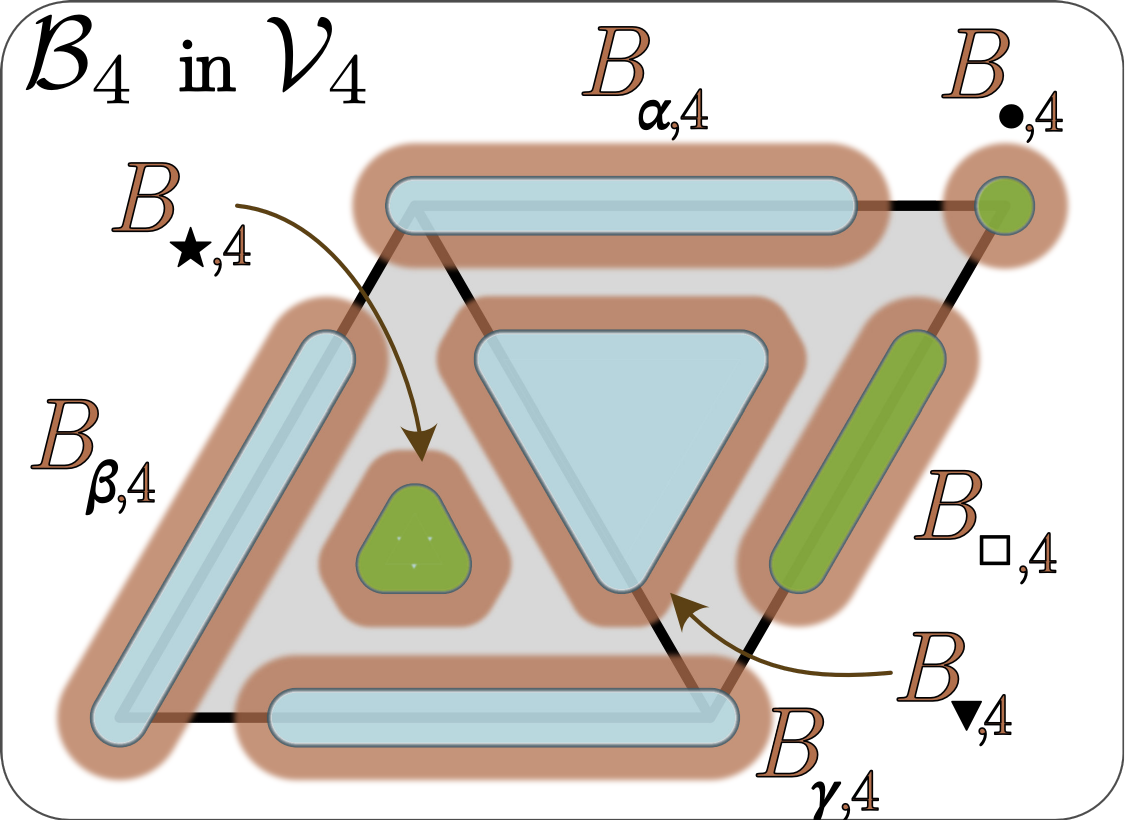}
    \end{tabular}
    \centering
    \caption{
        The central column shows the finest block partitions for multivector fields in Figure~\ref{fig:main-example-MVFs-v2}.
        The brown sets represent isolating blocks and the green subsets---their invariant parts.
        The left column contains block partition $\BD_\lambda$ for multivector $\cV_{\lambda+1}$ whenever $\cV_{\lambda}\inscr\cV_{\lambda+1}$.
        The right column presents block partition $\BD_2$ of multivector $\cV_1$.
    }
    \label{fig:main-example-MVFs-BD-v2}
\end{figure}
\vspace{0.5cm}
\begin{figure}
    \centering
    \includegraphics[width=0.66\linewidth]{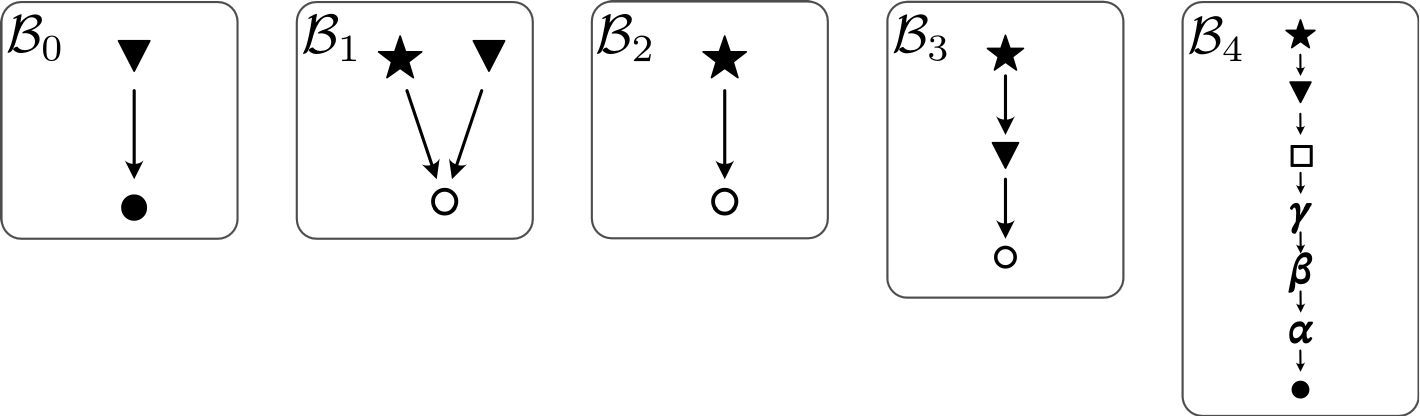}
    \caption{Flow induced partial orders corresponding to block decompositions in Example~\ref{ex:main-example-part-1}.}
    \label{fig:main-example-morse-graphs}
\end{figure}

\begin{example}\label{ex:main-example-part-2-indexing-maps}
    For the zigzag filtration $\zzBD$ from Example~\ref{ex:main-example-part-1} we have four indexing maps:
    $\idxbck{0}$, $\idxfwd{1}$, $\idxbck{2}$, and $\idxbck{3}$, 
    as shown in Figure~\ref{fig:idxbck-idxfwd-v2}.
    \qedex
\end{example}

\begin{figure}[h]
    \begin{tikzcd}[font=\small, row sep=0.25em, column sep=1.20em]
    \BD_0&\idxbck{0}&\BD_1&\idxfwd{1} &\BD_2&\idxbck{2} &\BD_3&\idxbck{3} &\BD_4\\
    && && && && \bl_{\star,4}\\
    \bl_{\blacktriangledown,0}\arrow[rr, hookleftarrow] 
        && \bl_{\blacktriangledown,1}\arrow[rrdd, hook] && 
        && \bl_{\star,3}\arrow[rru, hookleftarrow] && \bl_{\blacktriangledown,4}\\
    && \bl_{\star,1}\arrow[rr, hook] 
        && \bl_{\star,2}\arrow[rru, hookleftarrow]
        && \bl_{\blacktriangledown,3}\arrow[rru, hookleftarrow] && \bl_{\square,4}\\
    \bl_{\bullet,0}\arrow[rr, hookleftarrow]\arrow[rru, hookleftarrow] 
        && \bl_{\circ,1}\arrow[rr, hook] 
        && \bl_{\circ,2}\arrow[rr, hookleftarrow]\arrow[rru, hookleftarrow]
        && \bl_{\circ,3}\arrow[rru, hookleftarrow]\arrow[rr, hookleftarrow]
            \arrow[rrd, hookleftarrow]
            \arrow[rrdd, hookleftarrow]
            \arrow[rrddd, hookleftarrow]
        && \bl_{\gamma,4}\\
    && && && && \bl_{\beta,4}\\
    && && && && \bl_{\alpha,4}\\
    && && && && \bl_{\bullet,4}
    \end{tikzcd}
    \caption{Indexing maps 
        $\idxbck{0}$, $\idxfwd{1}$, $\idxbck{2}$, and $\idxbck{3}$ for the zigzag filtration of block decompositions $\zzBD$ from Figure~\ref{fig:main-example-MVFs-BD-v2}.
        }
    \label{fig:idxbck-idxfwd-v2}
\end{figure}       

We begin the analysis of the zigzag filtration by observing a direct relationship between two successive block decompositions.
We focus mostly on the $\inscr$ case, as the $\ovscr$ case is symmetric.

\begin{proposition}\label{prop:bd-after-refinement-is-bd}
    If $(\BD_1,\cV_1)\inscr(\BD_2,\cV_2)$, then $\BD_2$ is a block decomposition for~$\cV_1$.
\end{proposition}
\begin{proof}
    By Proposition~\ref{prop:iso_block_stability}, an isolating block $\bl_{q,2}\in\BD_2$ is an isolating block both in $\cV_1$ and $\cV_2$.
    Since $\dgVz\subset\dgVo$ we have $\paths_{\cV_1}(X)\subset\paths_{\cV_2}(X)$.
    Therefore, condition \ref{it:block_decomposition_paths} is preserved for $\BD_2$ with respect to $\cV_1$.

    Consider $\varphi\in\esol_{\cV_1}(X)$.
    In particular, $\varphi\in\sol_{\cV_2}(X)$.
    By~\ref{it:block_decomposition_uims} for $\BD_1$, there exist $p,q\in\PP_1$ such that 
        $\uimp_{\cV_1}\varphi\subset\bl_{p,1}$ and 
        $\uimm_{\cV_1}\varphi\subset\bl_{q,1}$.
    Moreover, 
        $\bl_{p,1}\subset\bl_{\idxfwd{1}(p),2}\in\BD_2$, and 
        $\bl_{q,1}\subset\bl_{\idxfwd{1}(q),2}\in\BD_2$.
    Hence, \ref{it:block_decomposition_uims} for $\BD_2$ in $\cV_1$ is satisfied.
\end{proof}

Proposition~\ref{prop:bd-after-refinement-is-bd} tells us that $\BD_2$ is a block decomposition both in $\cV_1$ and $\cV_2$.
Therefore, by Definition~\ref{def:continuation}, for every $\bl_{p,2}\in\BD_2$ its 
    invariant part in $\cV_2$---that is, $\inv_{\cV_2}\bl_{p,2}$---continues to 
        $\inv_{\cV_1}\bl_{p,2}$ in $\cV_1$, 
    and therefore, shares the same dynamical properties, in particular the Conley index (Theorem~\ref{thm:continuation_iso_conley_index}).
This allows us to study the transition from $\BD_1$ to $\BD_2$ on a common ground, that is, within multivector field~$\cV_1$.
The next theorem captures in more detail the relationship between individual isolating blocks of $\BD_1$ and $\BD_2$.
Let us first define the \emph{connection set}\index{connection set} for a family of subsets $\cA$ of $X$ within $Y\subset X$:
\begin{equation}
    \label{eq:connection_set}
	\cset{\cV}(\cA, Y) \coloneqq  \bigcup_{A, A'\in \cA} \{\im\rho \mid \rho\in\paths_\cV(Y),\ \pbeg\rho\in A,\ \pend\rho\in A'\}.
\end{equation}

\begin{theorem}
\label{thm:continuation-preimage-iota}
    Let $(\BD_1,\cV_1)\inscr(\BD_{2},\cV_{2})$.
    Let $q\in\PP_{2}$ and $\QQ:=\idxfwd{1}^{-1}(q)$.
    Then:
    \begin{enumerate}[label=(\alph*)]
        \item\label{it:continuation-preimage-iota-BD} $\BD_\QQ\coloneqq\{\bl_{p,1}\in\BD_1\mid p\in \QQ\}$ is a block decomposition of $\bl_{q,2}$ with respect to~$\cV_1$;
        moreover 
            $\bl_{\QQ,1}\coloneqq
                \cset{\cV_1}(\{\bl_{p,1}\mid p\in \QQ\}, X)\subset\bl_{q,2}$,
        \item\label{it:continuation-preimage-iota-MQ} 
            $M_{\QQ,1}\coloneqq\inv_{\cV_1}\bl_{q,2}$ is an invariant set in $\cV_1$ isolated by $\bl_{q,2}$;
            moreover, 
            $M_{\QQ,1}=\cset{\cV_1}(\{M_{p,1}\mid p\in \QQ\}, X)$,
        \item\label{it:continuation-preimage-iota-continuation} $M_{\QQ,1}$ in $\cV_{1}$ continues to $M_{q,2}$ in $\cV_{2}$.
    \end{enumerate}
\end{theorem}

If $\idxfwd{\tdyn}^{-1}(q)$ is empty, then the theorem becomes trivial. 
In particular, there is no isolated invariant set in $\cV_1$ corresponding to $M_{q,2}$; 
    or, in other words, $M_{q,2}$ continues to the empty set in $\cV_1$.
If $\idxfwd{\tdyn}^{-1}(q)$ is a singleton, that is $\QQ=\{p\}$, 
    then $M_{\QQ,1}=M_{p,1}$,
    and therefore, the Morse set $M_{p,1}$ corresponding to $\bl_{p,1}$ in $\cV_1$ simply
    continues to $M_{q,2}$ corresponding to $\bl_{q,2}$ in $\cV_2$.
    
The situation becomes interesting when $\QQ$ has more than one element.
We know that it is the aggregated Morse set $M_{\QQ,1}$ in $\cV_1$ that continues to $M_{q,2}$ in $\cV_{2}$.
Again, they share the Conley index through a common index pair (see Proposition~\ref{prop:iso_block_forms_ipair}),
    therefore it is enough to study the contribution of individual Morse sets in $\cM_{1}$ 
    to the Conley index of $M_{\QQ,1}$ within $\cV_1$.
In the following section we introduce split diagrams that quantify this contribution.

\begin{example}\label{ex:main-example-part-3-continuation-theorem}
    The left and the right column in Figure~\ref{fig:main-example-MVFs-BD-v2} illustrate Proposition~\ref{prop:bd-after-refinement-is-bd}.
    For instance, since $\BD_0\ovscr\BD_1\inscr\BD_2$, it follows that $\BD_0$ and $\BD_2$ are also proper block decompositions for $\cV_1$.
    Therefore, we will study the transition from $\BD_0$ to $\BD_1$ and from $\BD_1$ to $\BD_2$ within $\cV_1$.

    To illustrate Theorem~\ref{thm:continuation-preimage-iota}, consider the step $(\BD_1,\cV_1)\inscr(\BD_2,\cV_2)$ 
    of the zigzag filtration from Example~\ref{ex:main-example-part-1} and the corresponding indexing map $\idxfwd{1}:\PP_1\rightarrow\PP_2$.
    Let $q\coloneqq\circ\in\PP_2$ and $\QQ\coloneqq\idxfwd{1}^{-1}(\circ)=\{\circ, \blacktriangledown\}\subset\PP_1$.
    By point~\ref{it:continuation-preimage-iota-BD}, 
        $\BD_{\QQ,1}=\{
        \bl_{\circ, 1}, \bl_{\blacktriangledown, 1}\}$ 
        is a block decomposition of set $\bl_{\circ, 2}$ in $\cV_1$.
    In this case, the invariant part of $\bl_{\circ, 2}$ with respect to $\cV_1$, 
        denoted $M_{\QQ,1}$, is the same as for $\bl_{\circ,1}$, denoted $M_{\circ,1}$ 
        (the highlighted green empty triangle in the second row, the central and right column, respectively).
    By point~\ref{it:continuation-preimage-iota-continuation},
        $M_{\circ,1}$ continues to $M_{\circ,2}$; 
        in particular, because they share the same isolating block $\bl_{\circ,2}$.

    As another example consider step $(\BD_3,\cV_3)\ovscr(\BD_4,\cV_4)$ and the map $\idxbck{3}:\PP_4\rightarrow\PP_3$.
    Let $q\coloneqq\circ\in\PP_3$ and $\QQ\coloneqq\idxbck{3}^{-1}(\circ)=\{\square, \gamma, \beta, \alpha, \bullet\}\subset\PP_4$.
    By~\ref{it:continuation-preimage-iota-BD}, 
        $\BD_\QQ=\{
        \bl_{\square, 4}, \bl_{\gamma, 4}$, $\bl_{\beta, 4}$, $\bl_{\alpha, 4}$, $\bl_{\bullet, 4}\}$ 
        is a block decomposition of $\bl_{\circ, 3}$ with respect to $\cV_4$.
    By~\ref{it:continuation-preimage-iota-MQ},
        $M_{\QQ,4}$---the invariant part of $\bl_{\circ, 3}$ in $\cV_4$--- coincides with $\bl_{\circ, 3}$ and consists of two critical multivectors--- the vertex $\{d\}$ and the edge $\{cd\}$---as well as of the connections between them---the three regular multivectors---forming together the empty quadrangle (highlighted in green in the bottom left panel in Figure~\ref{fig:main-example-MVFs-BD-v2}).
    By~\ref{it:continuation-preimage-iota-continuation}, $M_{\circ, 3}$ in $\cV_3$ continues to $M_{\QQ,4}$ in $\cV_4$.
    As collections of cells $M_{\circ, 3}$ and $M_{\QQ,4}$ are the same, but they represent different dynamics.
    In particular, $M_{\circ,3}$ behaves like a periodic orbit in $\cV_3$, 
        while $M_{\QQ,4}$ contains heteroclinic connections between equilibria in~$\cV_4$.
    Moreover, $M_{\QQ,4}$ can be further decomposed into a finer block decomposition.
    \qedex
\end{example}

Before proving Theorem~\ref{thm:continuation-preimage-iota}, we introduce Proposition~\ref{prop:path-to-esol} and Lemma~\ref{lem:subBD}.
Lemma~\ref{lem:subBD} will also come in handy later.  

\begin{proposition}\cite[Corollary of Proposition~3.10]{LiMiMr2025}\label{prop:path-to-esol}
    Let $\rho\in\pathsV(S,S',X)$, where $S$ and $S'$ are isolated invariant sets.
    Then $\rho$ extends to an essential solution~$\varphi$ with $\uimm_\cV\varphi\subset S$ and $\uimp_\cV\varphi\subset S'$.
\end{proposition}

\begin{lemma}
    \label{lem:subBD}
    Let $A\subset X$ be an isolating block in $\cV$ and $\BD$ be a block decomposition of $X$.
    Define $\QQ:=\{q\in\PP\mid \bl_q\subset A\}$ and $\BD_\QQ:=\{\bl_p\in\BD\mid p\in \QQ\}$.
    If $\bl_p\cap A=\emptyset$ for every $p\in\PP\setminus\QQ$ then $\BD_\QQ$
         is a block decomposition of $A$.
    
    Additionally, $\inv_\cV A = \cset{\cV}(\BDmdt{\QQ}, A)=\inv_\cV\cset{\cV}(\BD_\QQ, A)$.
\end{lemma}
\begin{proof}
    To show that $\BD_\QQ$ is a block decomposition of $A$ consider $\varphi\in\esolV(A)$.
    Since $\BD$ is a block decomposition, there exists $p\in\PP$ such that $\uimp_{\cV}\varphi\subset \bl_{p}$.
    The fact that $\uimp_{\cV}\subset A$ implies that $p\in\QQ$.
    The same argument holds for $\uimm_\cV\varphi$, therefore \ref{it:block_decomposition_uims} is satisfied.
    Condition \ref{it:block_decomposition_paths} follows immediately by taking the restriction of the partial order on $\PP$ to $\QQ$, 
        because $\pathsV(A)\subset\pathsV(X)$.

    We show first that $\inv_\cV A \subset \cset{\cV}(\BDmdt{\QQ}, A)$.
    Consider $\varphi\in\esolV(A)$.
    Since $\BDmdt{\QQ}$ is a Morse decomposition of $A$, by~\ref{it:morse_decomposition_uims}, there are $q,q'\in\QQ$ such that 
        $\uimm_\cV\varphi\subset M_q\subset\bl_q$ and 
        $\uimp_\cV\varphi\subset M_{q'}\subset\bl_{q'}$, where $M_q\coloneqq\inv_\cV\bl_q\in\BDmdt{\QQ}$.
    Therefore, every point in $\im\varphi$ belongs to a path from $M_q$ to $M_{q'}$; 
        hence, $\im\varphi\subset \cset{\cV}(\BDmdt{\QQ}, A)$.
    
    To see $\cset{\cV}(\BDmdt{\QQ}, A)\subset\inv_\cV\cset{\cV}(\BD_{\QQ}, A)$, fix $q,q'\in\QQ$ and consider a path $\rho\in\pathsV(M_q, M_{q'}, A)$. 
    By Proposition~\ref{prop:path-to-esol}, we can extend $\rho$ to an essential solution~$\varphi$ with $\uimm\varphi\subset M_q\subset\bl_q$ and $\uimp\varphi\subset M_{q'}\subset\bl_{q'}$.
    Thus, $\im\rho\subset\im\varphi\subset\inv_\cV\cset{\cV}(\BD_{\QQ}, A)$.

    Finally, let $\varphi\in\esolV(\inv_\cV\cset{\cV}(\BD_\QQ, A))$.
    Clearly, $\varphi\in\esolV(A)$, and therefore, $\im\varphi\subset \inv_\cV A$, which shows the equalities in the second statement.
\end{proof}

\begin{proof}[Proof of Theorem~\ref{thm:continuation-preimage-iota}]
    The first statement in \ref{it:continuation-preimage-iota-BD} is a special case of Lemma~\ref{lem:subBD}.
    To see the second, consider a path $\rho\in\paths_{\cV_1}(\blo{p}, \blo{p'}, X)$ for some $p,p'\in\QQ$.
    Since $\rho$ is also a path in $\cV_2$ we have $\rho\in\paths_{\cV_2}(\blt{q}, \blt{q}, X)$.
    Thus, $\im\rho\subset\blt{q}$ by \ref{it:block_decomposition_paths}\ref{it:block_decomposition_paths_eq}.

    To see \ref{it:continuation-preimage-iota-MQ}, note that by Proposition~\ref{prop:iso_block_stability}, $\blt{q}$ is an isolating block in $\cV_1$;
        thus, by definition $\blt{q}$ isolates $M_{\QQ,1}$.
    The second part follows from the fact that $\blo{\QQ}\subset\blt{q}$ (by~\ref{it:continuation-preimage-iota-BD}) and the second part of Lemma~\ref{lem:subBD}.

    To show~\ref{it:continuation-preimage-iota-continuation}, note that 
        \ref{it:continuation-preimage-iota-MQ} implies that 
     $\bl_{q,2}$ is an isolating block for $M_{\QQ,1}$ in $\cV_1$,
     and by definition $\bl_{q,2}$ is an isolating block for $M_{q,2}$ in $\cV_2$.
     Since $M_{\QQ,1}$ and $M_{q,2}$ share the same isolating block, they are related by continuation.
\end{proof}

\subsection{Transition diagram for a basic zigzag filtration}
\label{subsec:basic-transition-diagram}
In this section we focus on a special type of zigzag filtration of block decompositions.
Namely, we say that $\zzBD$ is a \emph{basic zigzag filtration}\index{zigzag filtration of block decompositions!basic}
if for every $\tdyn\in\Lambda$, depending on the type of the map, we have:
\begin{align*}
    \abs{\idxbck{\tdyn}^{-1}(p)}\leq 2 \text{ for every } p\in\PP_{\tdyn}\text{,}
    \quad \text{or} \quad
    \abs{\idxfwd{\tdyn}^{-1}(r)}\leq 2 \text{ for every } r\in\PP_{\tdyn+1}.
\end{align*}

In other words, at each step of the zigzag filtration, 
    an isolating block can split into at most two other isolating blocks when $\BD_\tdyn\ovscr\BD_{\tdyn+1}$, or
    an isolating block can merge with at most one other to create a new 
        larger isolating block in step $\lambda+1$ when $\BD_\tdyn\inscr\BD_{\tdyn+1}$.
Therefore, each observed combinatorial bifurcation is a split into an attractor-repeller pair.

\begin{definition}[Index triple]
    \label{def:index_triple}
    Let $\BD=\{B_a, B_r\}$ be a block decomposition for an isolating block $\bl$ both in $\cV$;
        thus, inducing an attractor-repeller pair (see Proposition~\ref{prop:AR-block-decomposition}).
    Then, a triple $N_0\subset N_1\subset N_2$ of closed sets is called an \emph{index triple}\index{index triple} if 
        $(N_1, N_0)$, $(N_2,N_1)$ and $(N_2,N_0)$ form index pairs for 
        $M_a\coloneqq\inv_\cV B_a$, $M_r\coloneqq\inv_\cV B_r$ and $M\coloneqq\inv_\cV B$, respectively.
    Moreover, we assume that $B_a\subset N_1\setminus N_0$,  $B_r\subset N_2\setminus N_1$, and $B\subset N_2\setminus N_0$.
    In particular, the triple induces diagrams called  
        \emph{AR-split diagrams}\index{AR-split diagram} (see diagram~\eqref{eq:ip_index_triple}).
\end{definition}

\begin{align}
    \includegraphics[width=0.35\textwidth]{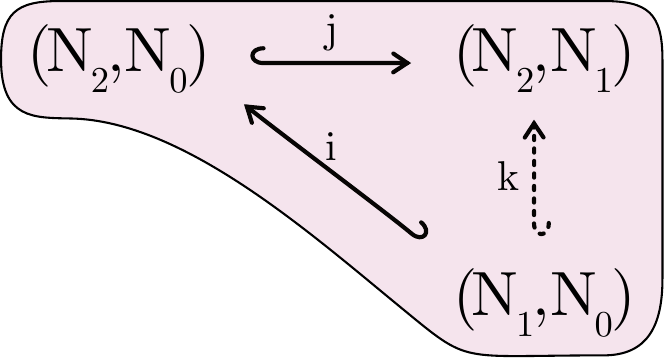}
    \qquad\quad
    \includegraphics[width=0.35\textwidth]{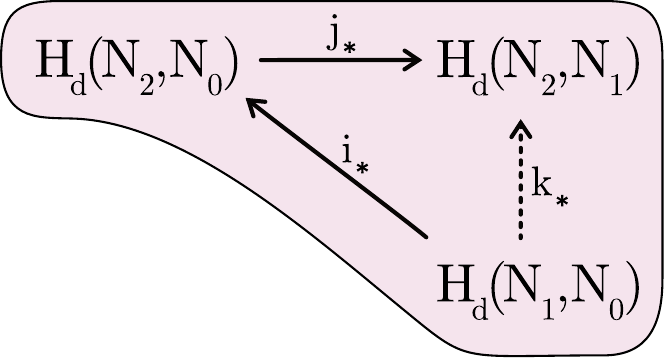}
    \label{eq:ip_index_triple}
\end{align}


It is easy to notice that an index triple forms a long exact sequence, as shown below, from which we can 
    relate the Conley indices of sets involved in the AR-decomposition, that is $M$, $M_a$ and $M_r$.
\begin{align}\label{eq:long_exact_sequence_index_triple}
    \ldots 
    \overset{i_\ast^d}{\longrightarrow}
    H_d(N_2, N_0) \overset{j_\ast^d}{\longrightarrow}
    H_{d}(N_2, N_{1}) \overset{\partial_\ast^d}{\longrightarrow}
    H_{d-1}(N_{1}, N_0) \overset{i_\ast^{d-1}}{\longrightarrow}
    \ldots
\end{align}

\begin{theorem}[The AR-split]\label{thm:index_triple}
Consider an index triple $N_0\subset N_1\subset N_2$ and the inclusion induced
    maps $i_\ast^d,j_\ast^d$ and $\partial_\ast^d$ as in the long exact sequence \eqref{eq:long_exact_sequence_index_triple}. 
    Then, we have the following properties:
\begin{enumerate}[label=(\alph*)]
    \item\label{it:0homomorphism} $k_\ast^d\coloneqq j_\ast^d\circ i_\ast^d = 0$ for all $d\in\NN$,
    \item\label{it:index_triple_split} $H_d(N_2, N_0) \cong \im i_\ast^d \oplus \frac{H_d(N_2, N_0)}{\ker j_\ast^d}$ for all $d\in\NN$,
    \item\label{it:index_triple_isomorphism} 
        $h_\ast^d: \coker j_\ast^d \rightarrow \ker i_\ast^{d-1}$,
        defined as the restriction of $\partial_\ast^d$ to $\coker j_\ast^d$,
        is an isomorphism for all $d\in\NN$.
\end{enumerate}
\end{theorem}
\begin{proof}
    Property \ref{it:0homomorphism} is straightforward since $N_{1}$ becomes the relative part under the map~$k$.

    To see \ref{it:index_triple_split}, note that we work with vector spaces; therefore, the long exact sequence splits.
    In particular, we have 
    $H_d(N_2, N_0) \cong \im i_\ast^d \oplus \frac{\ker j_*^d}{\im i_\ast^d} \oplus \frac{H_d(N_2, N_0)}{\ker j_*^d}$,
    but we can omit the middle term, because exactness provides that
    $\ker j_*^d=\im i_\ast^d$.

    To prove \ref{it:index_triple_isomorphism}, we use the following sequence of isomorphisms:
    \begin{equation*}
        \coker j_{\ast}^d \cong \frac{H_d(N_2, N_{1})}{\im j_{\ast}^d}
            \cong \frac{H_d(N_2, N_{1})}{\ker \partial_{\ast}^d}
            \cong \im \partial_{\ast}^d
            \cong \ker i_{\ast}^{d-1},
    \end{equation*}
    where the first equality follows by definition of $\coker$, the second and the fourth equality are given by the exactness of the sequence, and the third by the first isomorphism theorem.
\end{proof}

We will revisit Theorem~\ref{thm:index_triple} in Section~\ref{sec:cm-barcodes} after introducing the persistence language, 
    but for now, let us observe what it implies for Conley indices when we encounter an AR-split.
In particular, based on the right diagram in~\eqref{eq:ip_index_triple} we construct the diagram: 

\begin{equation*}
    \begin{tikzcd}[font=\small, row sep=normal]
        H(N_2,N_0)\arrow[r, phantom, "\cong"] &
            X_1\oplus X_2 \arrow[r, "0\oplus f"]
            & V_1 \oplus V_2 \arrow[r, phantom, "\cong"]
                \arrow[d, "h_\ast", to path={(\tikztostart.230) -- (\tikztotarget.50) \tikztonodes}]
                & H(N_2,N_1)\\
            & & W_1 \oplus W_2 \arrow[lu, "g\oplus 0"]
                \arrow[r, phantom, "\cong"]& H(N_1,N_0)
    \end{tikzcd}
\end{equation*}
where:
\begin{align*}
    X_1&\cong \im i_\ast, &       X_2&\cong \coim j_\ast =\frac{H(N_2, N_0)}{\ker j_\ast},\\
    V_1&\cong \coker j_\ast, &    V_2&\cong \im j_\ast,\\
    W_1&\cong \coim i_\ast = \frac{H(N_1, N_0)}{\ker i_\ast},&  W_2&\cong \ker i_\ast.
\end{align*}
The maps $f:X_2\rightarrow V_2$, $g:W_1\rightarrow X_1$, and $h:V_1\rightarrow W_2$ are induced by $j_\ast$, $i_\ast$, and $\partial_\ast$, respectively.
Note that $f$ and $g$ are of degree $0$, while $h$ of degree $-1$.
All maps are isomorphisms, $f$ and $g$ by the first isomorphisms theorem, and $h_\ast$ is given by Theorem~\ref{thm:index_triple}\ref{it:index_triple_isomorphism}.

\begin{corollary}\label{cor:index_triple_consequences}
Let 
    $N_0\subset N_1\subset N_2$
    be an index triple for isolated invariant set $M$ and its AR-decomposition $\{M_a, M_r\}$.
Then, Theorem \ref{thm:index_triple} implies that:
\begin{enumerate}[label=({\alph*})]
    \item\label{it:index_triple_time_consistency} 
        $\con(M_a)$ and $\con(M_{r})$ do not ``share'' generators, that is, any generator in $\con(M_a)$ is mapped into $0$ in $\con(M_r)$ through the connecting maps, because $k_\ast\cong(0\oplus f)\circ(g\oplus 0)=0$,
    \item\label{it:index_triple_con_distribution} 
        each generator of $\con(M)$ is present either in the attractor $M_{a}$ or in the repeller $M_r$;
        none of them vanishes during the split, 
            that is, $X_1$ continues as $W_1$ and $X_2$ continues as $V_2$,
    \item\label{it:index_triple_con_pairing}
        the generators of $\con_d(M_{r})$ and $\con_{d-1}(M_{a})$ that are not ``inherited'' from $\con(M)$, that is $W_1$ and $V_2$, are coupled via isomorphism $h_\ast^d$.
        That is, whenever new generators are born during the AR-split, they appear in pairs.
        Specifically, for every basis element of degree $d-1$ in $\con(M_a)$ which is not present in $\con(M)$, that is, an element of $W_2$,
            there exists a matching dual basis element of degree $d$ in $\con(M_r)$ which is likewise not present in $M$ (that is, an element of $V_1$).
\end{enumerate}
\end{corollary}

\begin{example}\label{ex:main-example-part-4-index-split}
    Consider the splitting of $\bl_{\bullet,0}$ into $\bl_{\circ,1}$ and $\bl_{\star,1}$ in the step $\BD_0\ovscr\BD_1$ from Example~\ref{ex:main-example-part-1}.
    In this case, the involved isolating blocks equal the corresponding Morse sets $M_{\bullet,0}$,  $M_{\circ,1}$ and $M_{\star,1}$.
    Sets $\bl_{\circ,1}$ and $\bl_{\star,1}$ form a block decomposition of $\bl_{\bullet,0}$ in $\cV_1$.
    One can easily check that sets $N_0\coloneqq\emptyset$, 
        $N_1\coloneqq\bl_{\circ,1}$ and 
        $N_2\coloneqq\bl_{\bullet,0}$ satisfy the index triple assumptions.
    Thus, we get the following diagram:
    \begin{equation*}
        \begin{tikzcd}[font=\small, row sep=normal, column sep=1.5em]
            (\bl_{\bullet,0}, \emptyset) \arrow[r, equal]
                & \ipairleqp{2}{}\arrow[r, hook]\arrow[rd, hookleftarrow]
                & \ipairpp{2}{1}{} \arrow[r, equal]
                & (\bl_{\bullet,0}, \bl_{\circ,1})\\
                & 
                & \ipairpp{1}{0}{}\arrow[u, hook,dashed,swap]\arrow[r, equal]
                & (\bl_{\circ,1}, \emptyset)
        \end{tikzcd}
    \end{equation*}       
    The relations between Conley indices are captured by the following diagram: 
    \begin{equation*}
        \begin{tikzcd}[font=\small, row sep=normal, column sep=2.0em]
            \con(M_{\bullet,0})\arrow[r]\arrow[rd, leftarrow]
                &  \con(M_{\star,1})\\ 
                & \con(M_{\circ,1})\arrow[u, dashed,swap]
        \end{tikzcd}
    \quad\cong\quad
        \begin{tikzcd}[font=\small, row sep=normal, column sep=2.5em]
         [ \fk, 0, 0]\arrow[r, "{[0, 0, 0]}"] 
            & {[0,0, \fk]} \\
        & {[\fk, \fk, 0]}\arrow[u, dashed,swap, "0"]\arrow[lu, "{[\id, 0, 0]}"]
        \end{tikzcd}        
    \end{equation*}

    The AR-split theorem 
    states that degree 0 generator of $\con(M_{\bullet,0})$ has been ``passed'' to $\con(M_{\circ,0})$. 
    Moreover, degree 1 and 2 generators of $\con(M_{\circ, 1})$ and $\con(M_{\star, 1})$, respectively, 
        form a coupled pair of generators which are born through the breakdown of $M_{\bullet,0}$.
    \qedex
\end{example}

The AR-split diagram captures dependencies of Conley indices of isolated invariant sets involved in the AR-decomposition. 
Therefore, it will serve as an elementary building block for the analysis of basic zigzag filtrations $\zzBD=\{\BD_\tdyn\}_{\tdyn\in\tInt}$ with $\tInt=\{0,1,\ldots,\tdynmax\}$. 
In particular, we compose all AR-split diagrams of the AR-splits occurring in $\zzBD$ into a single diagram called a \emph{transition diagram}\index{transition diagram} for $\zzBD$.
We denote it by $\zzTD$ and define it constructively using the following procedure:

\begin{enumerate}[label=\arabic*.]
    \item \textbf{Transition step:} 
        for each successive pair $\BD_{\tdyn}\inscr\BD_{\tdyn+1}$ in $\zzBD$ we relate the Conley indices associated with the isolating blocks in $\BD_\tdyn$ with those in $\BD_{\tdyn+1}$ using the split diagrams.
        We have three cases (which are symmetric for the $\BD_{\tdyn}\ovscr\BD_{\tdyn+1}$ case):
        \begin{enumerate}[label=1.\arabic*.]
            \item\label{it:transition-diagram-basic-split} 
            If $\idxfwd{\tdyn}^{-1}(q)=\{p_0, p_1\}$, that is $\bl_{q,\tdyn+1}\in\BD_{\tdyn+1}$ splits into $B_{p_0, \tdyn},B_{p_1, \tdyn}\in\BD_\tdyn$,
             we construct an AR-split diagram
            (we show an explicit construction in Section~\ref{subsec:construction-of-the-trans-diag}).
            \item\label{it:transition-diagram-continuation} If $\idxfwd{\tdyn}^{-1}(q)=\{p\}$ then we take any index pair $(P^q,E^q)$ for 
                $M_{q,\tdyn+1}$ in $\cV_{\tdyn+1}$
                that is also an index pair for 
                $M_{p,\tdyn}$ in $\cV_{\tdyn}$ 
                such that $\bl_{q,\tdyn+1}\subset\ipairdb{}{q}$ 
                (it always exists, for example $(\cl\bl_{q,\tdyn+1},\mo\bl_{q,\tdyn+1})$ satisfies this condition). 
            \item If $\idxfwd{\tdyn}^{-1}(q)=\emptyset$ then we take any index pair $(P^q,E^q)$ for 
                $M_{q,\tdyn+1}$ in $\cV_{\tdyn+1}$
                such that $\bl_{q,\tdyn+1}\subset\ipairdb{}{q}$ 
                (this case can only happen if $\BD_\lambda$ is not a block partition of~$X$).
        \end{enumerate}
    \item \textbf{Aligning step:} 
        The first step provides index pairs for every isolating block occurring in $\zzBD$.
        In particular, 
            for each isolating block in $\BD_0$ we construct exactly one index pair from the \textbf{Step 1} for $\BD_0$ and $\BD_1$,
            the same holds for $\BD_\tdynmax$, using the step for $\BD_{\tdynmax-1}$ and $\BD_\tdynmax$.
        Any other block $\bl_{p,\tdyn}$ has two index pairs constructed in the process, 
            one from the transition from $\BD_{\tdyn-1}$ to $\BD_{\tdyn}$, 
            and the other from the transition from $\BD_{\tdyn}$ to $\BD_{\tdyn+1}$.
        We denote the corresponding index pairs as $\ipairl{p,\tdyn}$ and $\ipairr{p,\tdyn}$, respectively.
        Both index pairs are for $\inv_{\cV_\tdyn}\bl_{p,\tdyn}$ in $\cV_\tdyn$, 
            thus, we connect them using the connecting sequence (see equation~\eqref{eq:connecting_ip_sequence}) to get a proper filtration of topological pairs.
\end{enumerate}

The above procedure explains the general philosophy of the transition diagram, 
    we provide a concrete recipe for the index pairs in Section~\ref{subsec:construction-of-the-trans-diag}.
After applying the homology functor to the constructed transition diagram we ob\-tain the
    \emph{Con\-ley-Morse persistence module}\index{Conley-Morse persistence module} whose decomposition 
    into intervals (strings) gives the \emph{Conley-Morse persistence barcode}\index{Conley-Morse persistence barcode}.
We will return to the decomposition in Section~\ref{sec:cm-barcodes} after necessary algebraic preparations in Section~\ref{sec:gentle-algebras-persistence}.

\begin{example}\label{ex:main-example-part-5-transition-diagram-basic}
    Consider the first four stages of the zigzag filtration from Example~\ref{ex:main-example-part-1}, that is $\BD_0\ovscr\BD_1\inscr\BD_2\inscr\BD_3$, which form a basic zigzag filtration.
    The last part of the filtration, that is $\BD_3\inscr\BD_4$, contains the splitting of $\bl_{\blacktriangledown, 3}$ into five isolating blocks, thus, making the filtration non-basic.   
    We will address this more general situation in the next section.
    
    In the first step of the construction of the transition diagram, we relate all pairs of successive block decompositions using the AR-split diagrams.
    The yellow blocks in Figure~\ref{fig:transition-diagram-basic} represent such comparisons for 
        $\BD_0\ovscr\BD_1$, $\BD_1\inscr\BD_2$ and $\BD_2\ovscr\BD_3$, from left to right, respectively.
    In each we have one AR-split diagram from step~\ref{it:transition-diagram-basic-split} and one equality from step~\ref{it:transition-diagram-continuation}.

    We construct index pairs for each pair of block decompositions independently;
        therefore $\ipairl{p,\tdyn}$ and $\ipairr{p,\tdyn}$ may differ.
    Thus, in the second step of the construction we join the comparison blocks using connection sequences introduced in Section~\ref{subsec:combinatorial_continuation} (see equation~\eqref{eq:connecting_ip_sequence});
        they are represented in Figure~\ref{fig:transition-diagram-basic} with the blue strips.
    
    Note that in the top part of Figure~\ref{fig:transition-diagram-basic} we indicate for which multivector field the given index pairs are well defined. 
    For instance, since we have $\BD_0\ovscr\BD_1$, the index pair $(P_{\bullet,0}, E_{\bullet,0})$ corresponding to $\bl_{\bullet,0}$ is also a proper index pair for $\cV_1$.
    It is the key fact that allows us to decompose $\bl_{\bullet,0}$ in $\cV_1$ into $\bl_{\star,1}$ and $\bl_{\circ,1}$, which we utilized in Theorem~\ref{thm:continuation-preimage-iota}.
    \qedex
\end{example}

\begin{figure}
    \centering
    \includegraphics[width=1.0\linewidth]{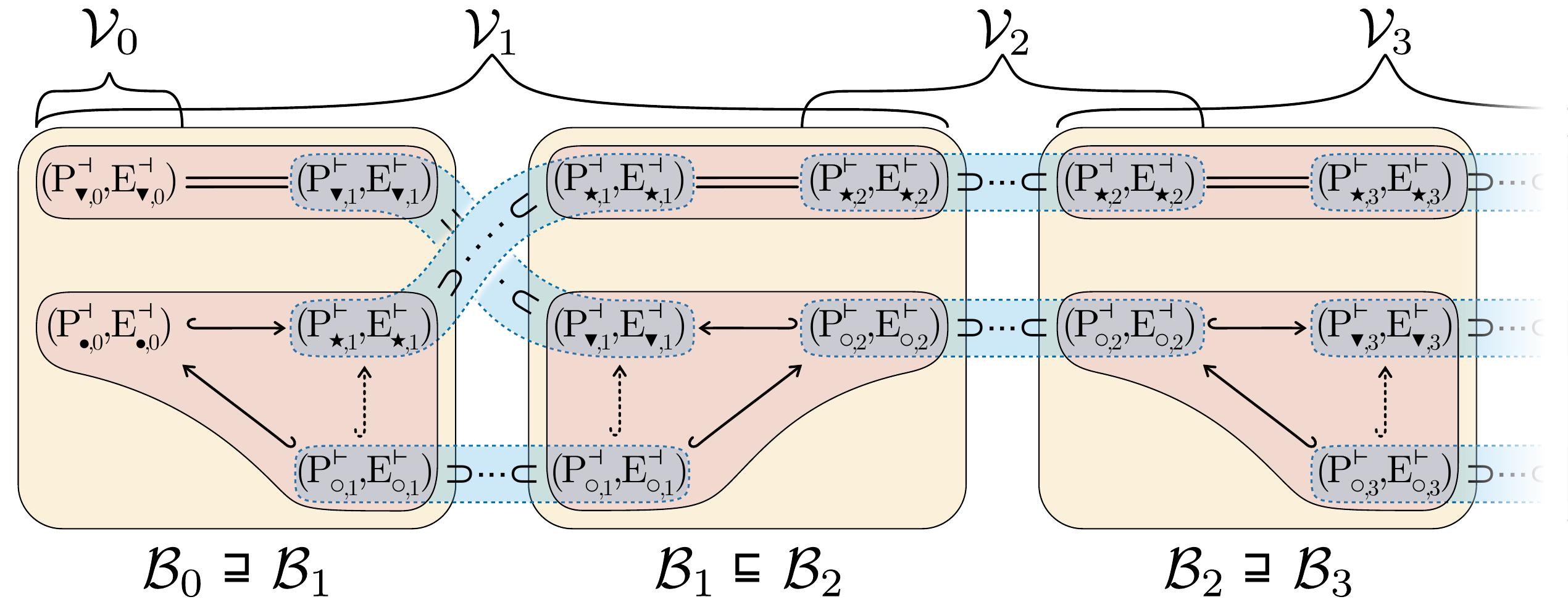}
    \caption{Transition diagram for the first four steps of the zigzag filtration $\zzBD$ from Example~\ref{ex:main-example-part-1} (see also Figure~\ref{fig:main-example-MVFs-BD-v2}).
    }
    \label{fig:transition-diagram-basic}
\end{figure}

We close the section with a straightforward, yet crucial property of acyclicity of the transition diagram. 

\begin{proposition}\label{prop:iptd_is_acyclic}
    The digraph obtained from a transition diagram by taking the set of index pairs as nodes and the directed arrows given by inclusions is acyclic.
    In particular, its transitive closure is a partially ordered set; we call it the \emph{transition diagram induced partial order}\index{partial order!transition diagram induced}.
\end{proposition}
\begin{proof}
    On the contrary, assume that there exists a loop in the induced graph.
    Consider a minimal loop.
    It corresponds to a sequence of index pairs 
        such that $\ipair{0} \subset\ipair{1}\subset\ldots \subset\ipair{n}=\ipair{0}$,
        which implies that all index pairs in the loop are equal.
    Necessarily, there exists an $i$ such that index pairs $\ipair{i}$ is at step $\lambda$ 
        and index pairs $\ipair{i-1}$ and $\ipair{i+1}$ are at step $\lambda-1$ and $\lambda+1$, respectively.
    This configuration is possible only if the three index pairs form an AR-split.
    Note that by Definition~\ref{def:index_triple}, 
        the index pairs cannot be equal, a contradiction.
\end{proof}

\subsection{Transition diagram for a non-basic zigzag filtration}
\label{subsec:TD-general-case}

Our strategy for the general, non-basic zigzag filtration $\zzBD$ is to reduce it to the basic case and to apply the procedure from the previous section.

\begin{definition}[AR-cascade]
    Consider two block decompositions such that $(\BD_0,\cV_0)\inscr(\BD_1,\cV_1)$.
	An~\emph{AR-cascade}\index{AR-cascade} from $\BD_0$ to $\BD_1$ is a 
        basic filtration of block decompositions
    \begin{align*}
        \zzBD_{0,1}\coloneqq (\BD_0,\cV_0)=
            (\BD_{0'},\cV_0)\inscr 
            (\BD_{0''},\cV_0)\inscr 
            \ldots\inscr
            (\BD_{0^{(n)}},\cV_0)= 
            (\BD_1,\cV_1).
    \end{align*}
\end{definition}
In other words, 
    whenever we observe a refinement of $\bl_{q,1}$ into more than two blocks in $\BD_0$,
    we decompose the split into a sequence of attractor-repeller decompositions leading to a basic filtration.
As it follows from the proposition below, it is always possible to merge two elements into a coarser isolating block.
An iterative application of the result leads to an AR-cascade. 

\begin{proposition}\label{prop:merging-two-blocks}
    Let $(\BD_0,\cV_0)\inscr(\BD_1,\cV_1)$, $q\in\PP_1$ and $\QQ\coloneqq\idxfwd{}^{-1}(q)$.
    Then, $\QQ$ is convex in $\PP_0$.
    Assume that $\abs{\QQ}\geq 2$ and $\{p,p'\}$ is a convex subset of $\QQ$.
    Then, $\BD_{0'}\coloneqq \BD_0\setminus\{\bl_{p,0},\bl_{p',0}\}\cup\{\cset{\cV_0}(\{\bl_{p,0},\bl_{p',0}\}, X)\}$
    is again a block decomposition.
    Moreover, $(\BD_0,\cV_0)\inscr(\BD_{0'},\cV_0)\inscr(\BD_1,\cV_1)$.
\end{proposition}

Consider an arbitrary zigzag filtration 
    $\mathfrak{B}=\{(\BD_\tdyn,\cV_\tdyn)\}_{\tdyn\in\Lambda}$, where 
    $\Lambda=\zint{\tdynmax}$.
A~\emph{simplified zigzag filtration for $\zzBD$}\index{zigzag filtration of block decompositions!simplified} is a basic zigzag filtration:
\begin{align*}
    \zzBD_{0,1} \zzBD_{1,2} \ldots \zzBD_{T-1,T},
\end{align*}
where $\zzBD_{\tdyn,\tdyn+1}$ is an AR-cascade of the following form:
\begin{itemize}
    \item in the $(\BD_\tdyn,\cV_\tdyn)\inscr(\BD_{\tdyn+1},\cV_{\tdyn+1})$ case: 
        \begin{align*}
            (\BD_\tdyn,\cV_\tdyn)=
                (\BD_{\tdyn'}, \cV_\tdyn)\inscr 
                (\BD_{\tdyn''},\cV_\tdyn)\inscr 
                \ldots\inscr
                (\BD_{\tdyn^{(n)}},\cV_\tdyn)\inscr 
                (\BD_{\tdyn+1},\cV_{\tdyn+1}).
        \end{align*}
    \item in the $(\BD_\tdyn,\cV_\tdyn)\ovscr(\BD_{\tdyn+1},\cV_{\tdyn+1})$ case: 
        \begin{align*}
            (\BD_\tdyn,\cV_\tdyn)\ovscr
                (\BD_{{\tdyn+1}^{(n)}},     \cV_{\tdyn+1})\ovscr 
                \ldots\ovscr
                (\BD_{{\tdyn+1}''},    \cV_{\tdyn+1})\ovscr 
                (\BD_{{\tdyn+1}' },\cV_{\tdyn+1})= 
                (\BD_{\tdyn+1},\cV_{\tdyn+1}).
        \end{align*}
\end{itemize}
From now on, we will study only basic zigzag filtrations of block decompositions as the general case can be always reduced to a basic one.

\begin{example}\label{ex:main-example-part-6-cascade}
    Consider $\BD_3\ovscr\BD_4$ step of the zigzag filtration in Example~\ref{ex:main-example-part-1}, where the isolating block $\bl_{\circ, 3}\in\BD_3$ splits into 
    $\{\bl_{\square, 4}, \bl_{\gamma, 4}, \bl_{\beta, 4}, \bl_{\alpha, 4}, \bl_{\bullet 4}\}\subset\BD_4$.
    By Proposition~\ref{prop:merging-two-blocks} we can turn $\BD_3\ovscr\BD_4$ into an AR-cascade 
        $(\BD_3, \cV_3)\ovscr (\BD_{4''}, \cV_4)\ovscr (\BD_{4'}, \cV_4)\ovscr (\BD_4, \cV_4)$.
    One possibility is presented in Figure~\ref{fig:main-example-MVFs-BD-cascade}.
    Block decomposition $\BD_{4'}$ is obtained by merging 
        $\bl_{\square, 4}$ and $\bl_{\gamma,4}$ into $\bl_{\square, {4'}}$, and 
        by merging
        $\bl_{\bullet, 4}$ and $\bl_{\alpha,4}$ into $\bl_{\bullet, 4'}$.
    The block decomposition $\BD_{4''}$ is obtained by merging $\bl_{\bullet, 4'}$ and $\bl_{\beta, 4'}$ from $\BD_{4'}$ into $\bl_{\bullet,4''}$.

    With the simplified zigzag filtration for $\zzBD$ we can finish the transition diagram we started to construct in Example~\ref{ex:main-example-part-5-transition-diagram-basic}.
    It is shown in Figure~\ref{fig:transition-diagram-cascade}. 
    \qedex
\end{example}                                             

\renewcommand{\arraystretch}{0.0}
\begin{figure}
    \begin{tabular}{c@{}c}
   \includegraphics[width=0.375\linewidth]{figures/main-example-v2/MVF4-BD3-MD.pdf}&
   \includegraphics[width=0.375\linewidth]{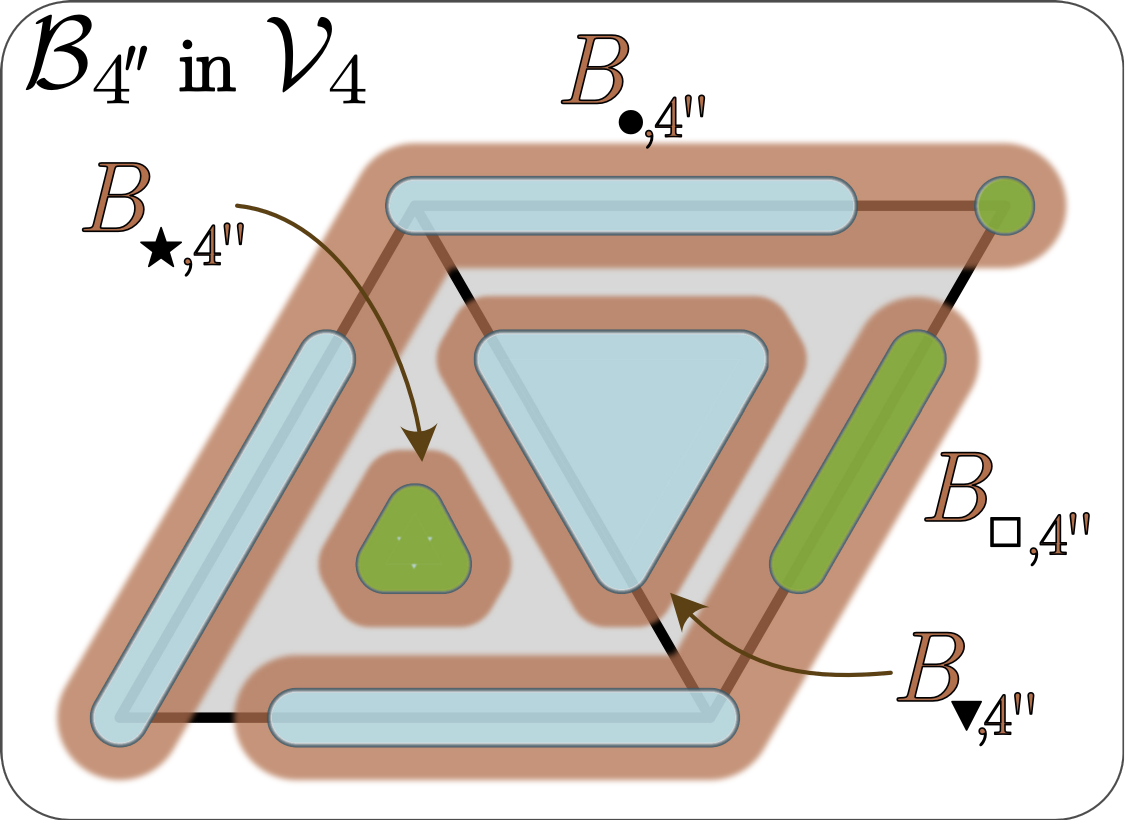}\\
   \includegraphics[width=0.375\linewidth]{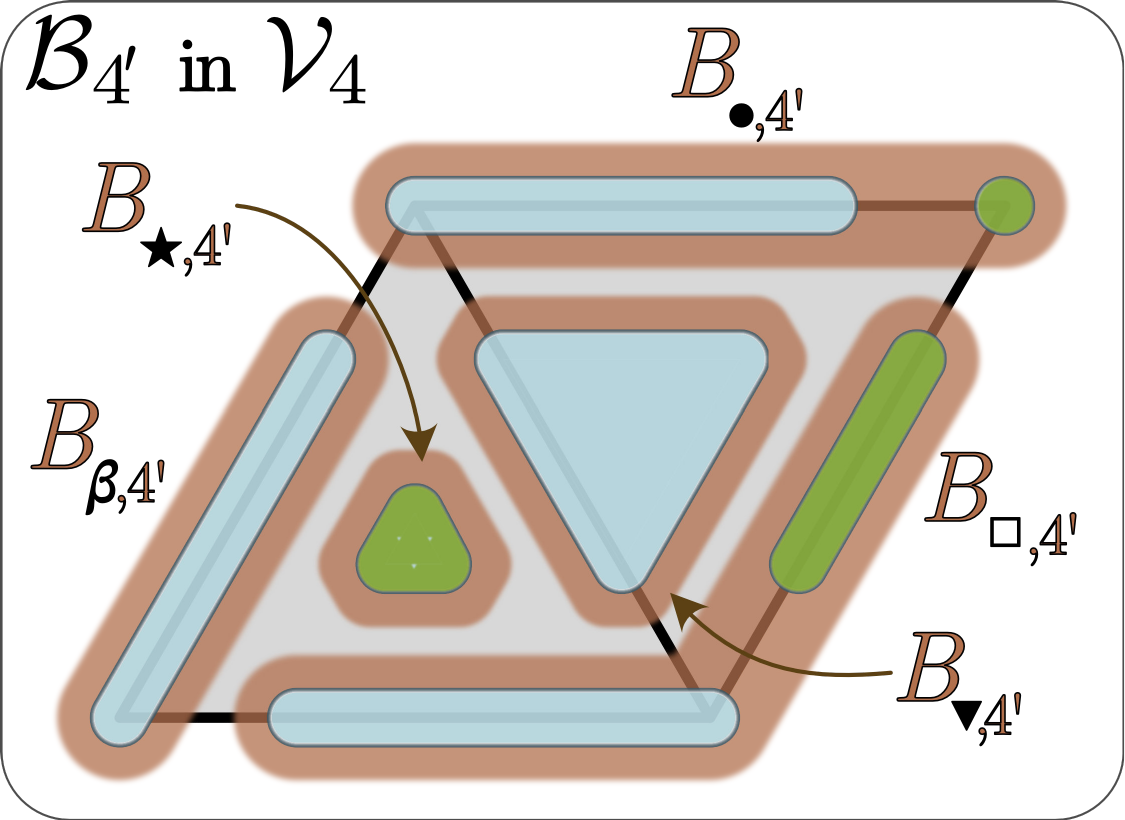}&
   \includegraphics[width=0.375\linewidth]{figures/main-example-v2/MVF4-BD4-MD.pdf}
    \end{tabular}
    \centering
    \caption{Splitting cascade for the step $\BD_3\ovscr\BD_4$ of zigzag filtration $\zzBD$ from Example~\ref{ex:main-example-part-1}.
    }
    \label{fig:main-example-MVFs-BD-cascade}

    \vspace{1cm}
    \centering
    \includegraphics[width=1.0\linewidth]{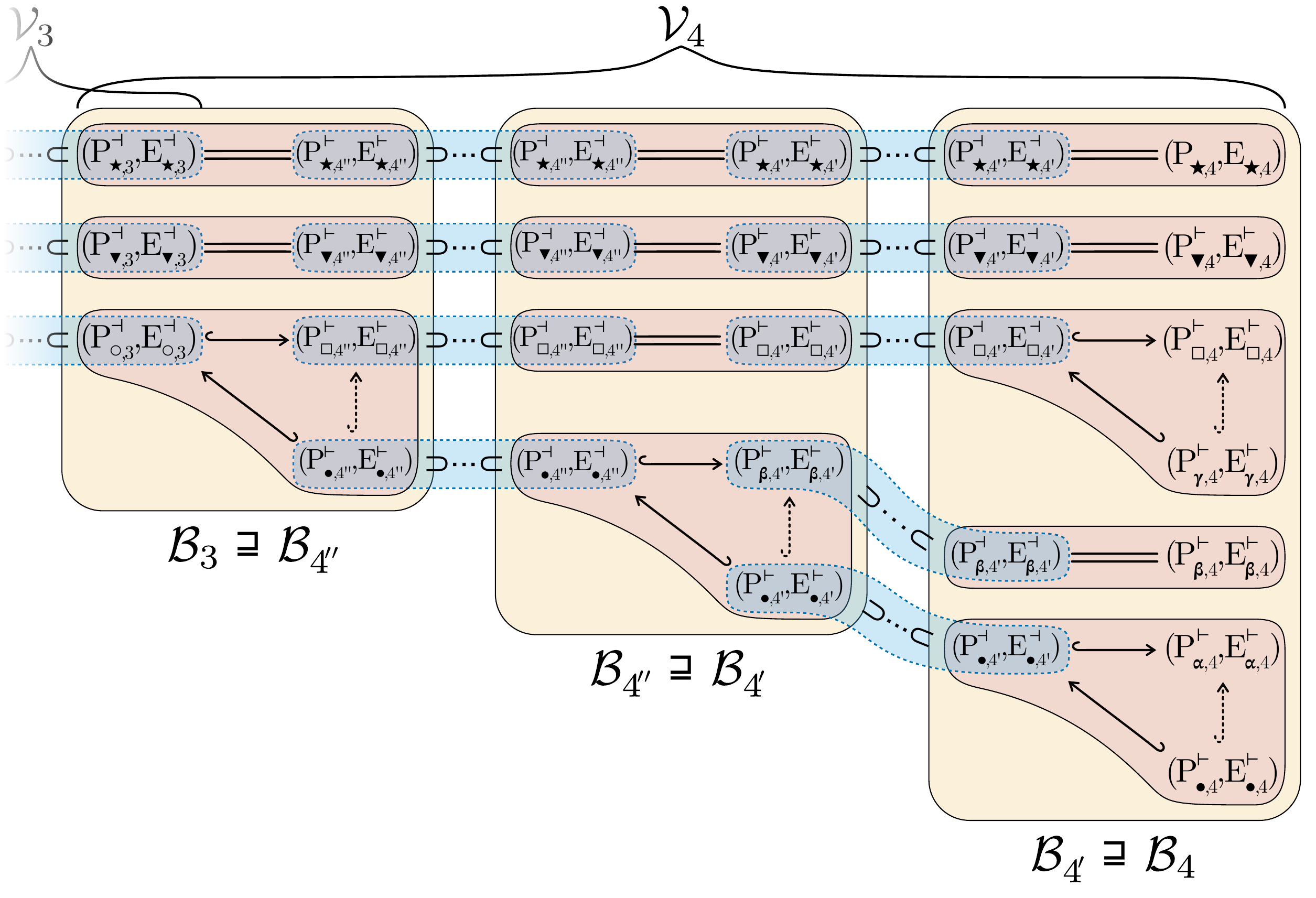}
    \caption{Transition diagram for the splitting cascade from Example~\ref{ex:main-example-part-6-cascade} (see also Figure~\ref{fig:main-example-MVFs-BD-cascade}).
    }
    \label{fig:transition-diagram-cascade}
\end{figure}

\begin{remark}
    One should view the above expansion of zigzag filtration analogously to the approach in standard persistence.
    In practice, for a filtration of complexes $K_0\subset K_1\subset\ldots\subset K_n$ we expand the sequence implicitly so that in each step a single cell is added. 
    The chosen order determines the basis and may change the pairing of cells in the algorithm.
    It does not affect the final persistence bars, because
        we treat the expanded parts as if they arrive at the same time.
    The same happens in our case; in order to make the construction computable we split the sequence into simpler steps and collapse them again at the very end.
    We discuss in Section~\ref{sec:discussion} how a choice of AR-cascade may affect the final outcome, 
        but a thorough study of this phenomenon is beyond the scope of this paper. 
\end{remark}

We prove Proposition~\ref{prop:merging-two-blocks} with the help of the following, more general lemma.

\begin{lemma}[Consolidation lemma]\label{lem:consolidation-BD}
    Let $\BD$ be a block decomposition of $X$ for $\cV$ and $\QQ\subset\PP$ be a convex subset.
    Define $\bl_{\QQ}\coloneqq\cset{\cV}(\{\bl_{p}\mid p\in \QQ\}, X)$.
    Then $\BD'\coloneqq\BD\setminus\{\bl_{q}\mid q\in\QQ\}\cup\{\bl_{\QQ}\}$ is also a block decomposition of $X$ for $\cV$ .
\end{lemma}
\begin{proof}
    An easy modification of the proof of \cite[Lemma~4.12]{LiMiMr2025} shows that $\bl_\QQ$ is an isolating block and that $\bl_\QQ\cap \bl_p=\emptyset$ for all $p\in\PP\setminus\QQ$.
    Thus, $\BD'$ is a family of mutually disjoint isolating blocks.
    
    Condition \ref{it:block_decomposition_uims} is satisfied, 
        because for a $\varphi\in\esolV(X)$ we can find $p\in\PP$ such that $\uimp\varphi\subset\bl_p\in\BD$. 
    If $p\in\QQ$ then $\uimp\varphi\subset\bl_\QQ\in\BD'$; otherwise we again have $\uimp\varphi\subset\bl_p\in\BD'$.

    To show that there exists a partial order satisfying \ref{it:block_decomposition_paths} suppose the contrary, 
        that is, the flow induced order on $\PP'$ contains a loop.
    Since there was no loop in $\PP$ for $\BD$ and the set of solutions remains the same, 
        the loop had to appear as a result of aggregating isolating blocks into $\bl_\QQ$.
    Thus, $\QQ$ is an element of the loop and we have a sequence of relations
        $\QQ<p_0<p_1<\ldots<p_k<\QQ$.
    It follows that there are $q,q'\in\QQ$ such that $q<p_0<q'$, but this contradicts the assumption that $\QQ$ is convex.
    Therefore, \ref{it:block_decomposition_paths} is also proved.
\end{proof}

\begin{proof}[Proof of Proposition~\ref{prop:merging-two-blocks}]
    Suppose that $\QQ$ is not convex in $\PP_0$.
    Then, for certain $p,p'\in\QQ$ and $r\in\PP\setminus\QQ$ we have $p>r>p'$.
    Therefore, there exist paths 
        $\rho\in\paths_{\cV_0}(\blz{p},\blz{r}, X)$ and 
        $\rho'\in\paths_{\cV_0}(\blz{r},\blz{p'}, X)$,
        which are also paths in $\cV_1$.
    Hence,
        $\rho\in\paths_{\cV_1}(\blo{q},\blo{\idxfwd{}(r)},X)$ and 
        $\rho'\in\paths_{\cV_1}(\blo{\idxfwd{}(r)},\blo{q},X)$.
    This implies $q<\idxfwd{}(r)<q$, which contradicts that $\BD_1$ is a block decomposition.

    Lemma~\ref{lem:consolidation-BD} provides that $\BD_{0'}$ is indeed a block decomposition of $X$ for $\cV_0$.

    Finally, the statement $\bl_{r,0'}\coloneqq\cset{\cV_0}(\{\bl_{p,0},\bl_{p',0}\}, X)\subset\blo{q}$ follows directly from Theorem~\ref{thm:continuation-preimage-iota}\ref{it:continuation-preimage-iota-BD}.
    Clearly, $\bl_{p,0},\bl_{p',0}\subset \bl_{r,0'}$.
    Thus, $\BD_0\inscr\BD_{0'}\inscr\BD_1$.
\end{proof}

\subsection{Construction of the transition diagram}
    \label{subsec:construction-of-the-trans-diag}

In this section we present an explicit construction of a transition diagram for a zigzag filtration.
The reader interested in the decomposition theorem leading to the Conley-Morse persistence diagram and its properties 
    can safely skip this section.

Let us begin with a simple zigzag filtration consisting of two block partitions
    $\zzBD\coloneqq (\BD_0,\cV_0)\inscr(\BD_1,\cV_1)$.
Let $\idxfwd{}:\PP_0\rightarrow\PP_1$ be the corresponding forward map.
Linear extensions on $\PP_0$ and $\PP_1$ are called \emph{filtration consistent linear orders}\index{linear order!filtration consistent}
    if $\idxfwd{}$ is \emph{order preserving}\index{order preserving map}, 
    that is
        for all $p,p'\in\PP_0$ relation $p<p'$ implies $\idxfwd{}(p)<\idxfwd{}(p')$.
The proposition below shows that the blocks
    that merge together when passing to $\BD_1$ are grouped together in the filtration consistent order.

\begin{proposition}
    If $p,p',p''\in\PP_0$, $p<p'<p''$ and $q\coloneqq\idxfwd{}(p)=\idxfwd{}(p'')$ then $\idxfwd{}(p')=q$.
\end{proposition}
\begin{proof}
    Assume the contrary that $\idxfwd{}(p')=q'\neq q$.
    Since $\idxfwd{}$ is an order preserving map, 
        $p<p'<p''$ implies $q<q'<q$, which is a contradiction.
\end{proof}

\begin{proposition}
    For any linear extension of $\PP_1$ there exists a filtration consistent linear order on $\PP_0$.
\end{proposition}
\begin{proof}
    Define a linear order on $\PP_0$ starting from the lowest element as follows:
    In increasing order, for each $q\in\PP_1$, consider the set $\idxfwd{}^{-1}(q)$.
    If it is nonempty, fix any linear order on that set consistent with the partial order on $\PP_0$ and append it to the growing sequence.
    If it is empty, skip it. 

    The obtained order on $\PP_0$ clearly is a linear extension. 
    By construction, $p<p'$ implies $\idxfwd{}(p)\leq\idxfwd{}(p')$.
\end{proof}
Assume that $\PP_0=\zintp{m}$ reflecting a linear order consistent with the filtration.
For a $p\in\PP_0$ denote $q_p\coloneqq \idxfwd{}(p)$.
Then we define the sequence of closed sets:
\begin{align}
\label{eq:Nsetssequence}
    N_{p}\coloneqq\begin{cases}
        N_{p-1}\cup \pf_{\cV_0}(B_{p,0}, X);\quad   &\text{if } p \neq \max\idxfwd{}^{-1}(q_p),\\
        N_{p-1}\cup \pf_{\cV_1}(B_{q_p,1}, X);\quad  &\text{if } p = \max\idxfwd{}^{-1}(q_p),\\
    \end{cases}
\end{align}
where $N_0\coloneqq\emptyset$.
This gives us a nested sequence of topological spaces.
The two cases in formula \eqref{eq:Nsetssequence} distinguish whether $B_p$ is the maximal element in the linear order among the sets merging into $B_{q_p,1}$. 
If $\BD_0$ and $\BD_1$ are block partitions, then the sequence becomes a filtration of $X$ consistent with the chosen linear order on $\PP_0$.
\begin{proposition}\label{prop:Nsequence-for-partitions}
    If $\BD_0$ and $\BD_1$ are block partitions then $N_p=\bigcup_{k\leq p} B_{k,0}$.
\end{proposition}

The next two propositions guarantee that the sets defined in formula~\eqref{eq:Nsetssequence} directly lead to index pairs that can be used to construct the transition diagram.

\begin{proposition}\label{prop:NpNpm-ipair-Mp}
    $(N_{p}, N_{p-1})$ is an index pair for $M_{p,0}\coloneqq\inv_{\cV_0}\bl_{p,0}$.
\end{proposition}

\begin{proposition}\label{prop:NpNpm-ipair-Mq}
    Consider a filtration $(\BD_0,\cV_0)\inscr(\BD_1,\cV_1)$.
    Let $q\in\PP_1$, $\QQ\coloneqq\idxfwd{}^{-1}(q)$, and $p\coloneqq\max\QQ$.
    Denote $M_{\QQ,0}\coloneqq\inv_{\cV_0}\bl_{q,1}$ and 
        $M_{q,1}\coloneqq\inv_{\cV_1}\bl_{q,1}$.
    If $\QQ\neq\emptyset$, then, $(P,E)\coloneqq(N_p, N_{p-\abs{\QQ}})$
    is a common index pair for $M_{\QQ,0}$ in $\cV_0$ and for $M_{q,1}$ in $\cV_1$.
\end{proposition}

\begin{example}\label{ex:main-example-part-7-concrete-transition-diagram}
    We apply Proposition~\ref{prop:Nsequence-for-partitions} to construct 
        index pairs for the transition diagram described in 
        Examples~\ref{ex:main-example-part-5-transition-diagram-basic} and~\ref{ex:main-example-part-6-cascade}. 
    According to the procedure, we need to choose consistent linear orders for every pair of consecutive block decompositions first.
    For the step $\BD_0\inscr\BD_1$ the only possibility is 
        $\bullet<\blacktriangledown$ for $\PP_0$ and $\circ<\star<\blacktriangledown$ for $\PP_1$. 
    For the step $\BD_1\ovscr\BD_2$ we also have no choice but  
        $\circ<\blacktriangledown<\star$ for $\PP_1$ and $\circ<\star$ for $\PP_2$. 
    Note that we have to choose different linear orders for $\PP_1$ at each step in to ensure consistency with the filtration.
    For the remaining steps we use a proper ordering which is already depicted in 
        Figures~\ref{fig:transition-diagram-basic} and~\ref{fig:transition-diagram-cascade}.

    In the running example we consider only block partitions which allows us to apply Proposition~\ref{prop:Nsequence-for-partitions} to easily obtain the index pairs.
    In particular, for the transition from $\BD_0$ to $\BD_1$ we have:
        $N_0\coloneqq\emptyset$, $N_1\coloneqq\bl_{\circ,1}$,  $N_2\coloneqq\bl_{\circ,1}\cup\bl_{\star,1}$, and $N_3\coloneqq K$.
    To construct the splitting diagrams between $\BD_1$ and $\BD_2$ we put 
        $N_0\coloneqq\emptyset$, $N_1\coloneqq\bl_{\circ,1}$,  $N_2\coloneqq\bl_{\circ,1}\cup\bl_{\blacktriangledown,1}$, and $N_3\coloneqq K$.
    The obtained transition diagram is presented in Figure~\ref{fig:transition-diagram-actual-blocks}.
    Instead of listing simplices, we draw the corresponding subcomplexes for each index pair.
    
    In practice, the connecting sequences (the blue strips) are not always needed.
    In our example, no additional intermediate index pair is necessary.
    Only two isolating blocks in $\BD_1$ have different index pairs,
        but they are still related by an inclusion.
    \qedex
\end{example}

\begin{figure}
    \centering
    \includegraphics[width=1.0\linewidth]{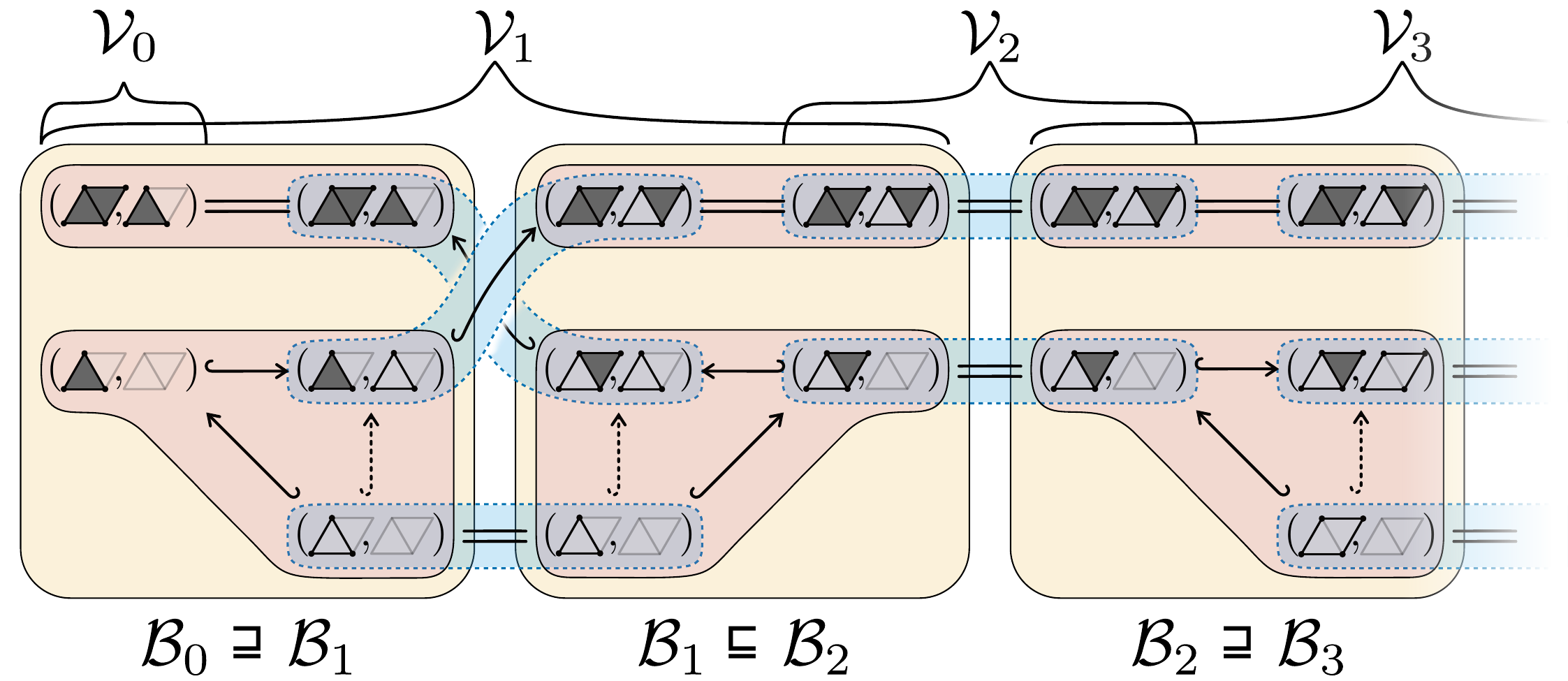}
    
    \vspace{0.5cm}
    
    \includegraphics[width=1.0\linewidth]{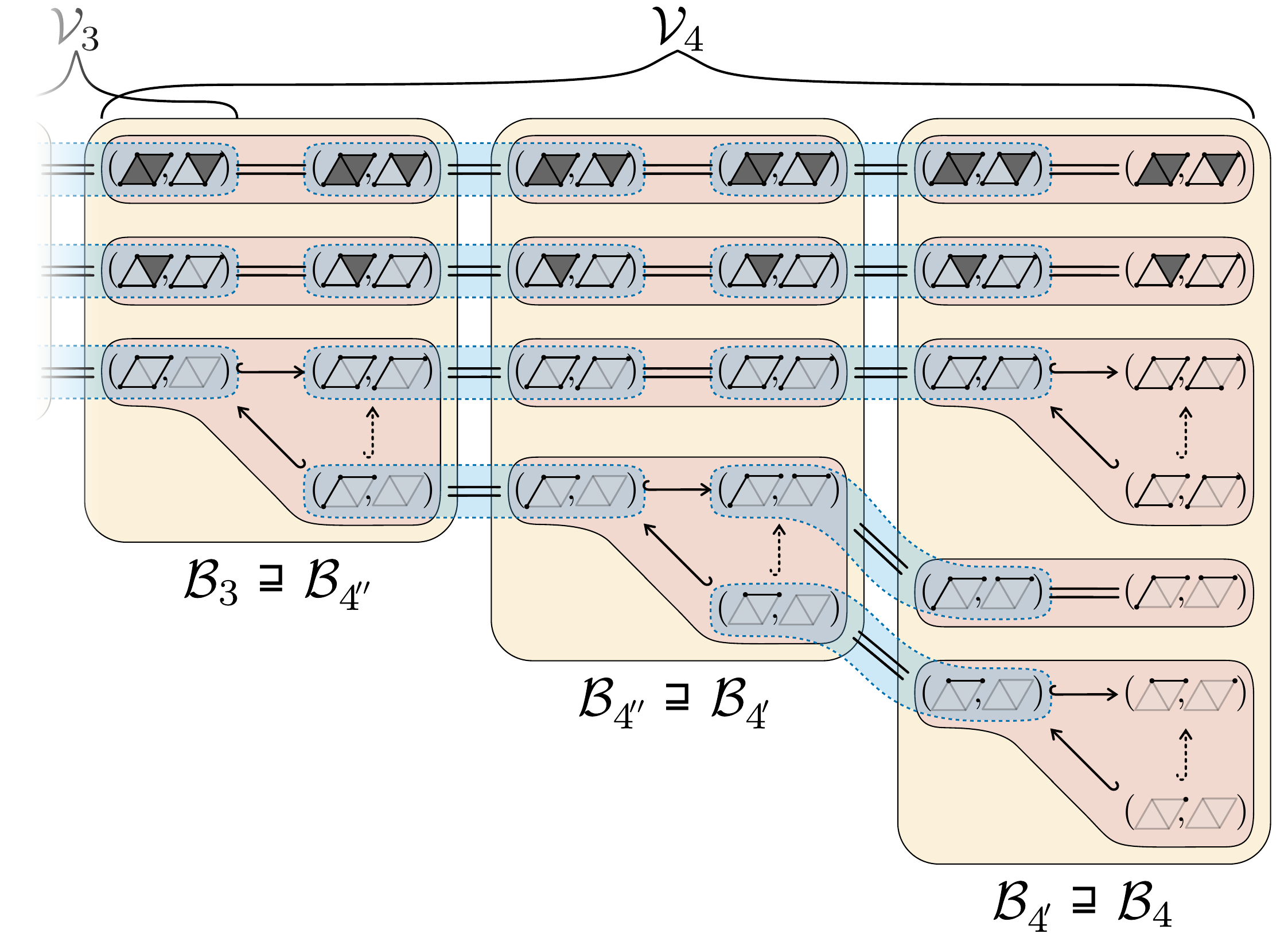}
    \caption{A transition diagram for the zigzag filtration $\zzBD$ from Example~\ref{ex:main-example-part-1}
        with concrete simplicial complexes representing  index pairs.
        }
    \label{fig:transition-diagram-actual-blocks}
\end{figure}

Proposition~\ref{prop:Nsequence-for-partitions} follows immediately from the next lemma.
\begin{lemma}\label{lem:Np-minus-Npm-Bp}
    The sets $\{N_p\}_p$ defined in formula~\eqref{eq:Nsetssequence} have the following properties:
    \begin{enumerate}[label=(\alph*)]
        \item\label{it:Np-minus-Npm-Bp-P0} 
            Let $p\in\PP_0$.
            Then $\blz{p}\subset N_p\setminus N_{p-1}$. 
        \item\label{it:Np-minus-Npm-PQ} 
            Let $q\in\PP_1$, $\QQ\coloneqq\idxfwd{}^{-1}(q)$ and $p\coloneqq\max\QQ$.
            Then $\blo{q}\subset N_p\setminus N_{p-\abs{\QQ}}$.
    \end{enumerate}
\end{lemma}

\begin{proof}
    By definition, $\blz{p}\subset\pf_{\cV_0}(\blz{p}, X)\subset N_p$.
    Suppose that there exists an $x\in \blz{p}\cap N_{p-1}$.
    Since $x\in N_{p-1}$ 
    we have two cases corresponding to formula~\eqref{eq:Nsetssequence};
        in the first one,
            there exists a $p'\in\PP$ such that $p'<p$ and $x\in\pf_{\cV_0}(\blz{p'},X)$;
        but $x$ is also in $\blz{p}$, therefore there is a path $\rho\in\paths_{\cV_0}(\blz{p'}, \blz{p},X)$, 
        and by \ref{it:block_decomposition_paths} we have $p<p'$, which is a contradiction.
    In the second case, there exists a $p'\in\PP$ such that $p'<p$, $q'\coloneqq\idxfwd{}(p')$ and 
        $x\in\pf_{\cV_1}(\blo{q'},X)$. 
    By the same argument, we can find a path $\rho'\in\paths_{\cV_1}(\blo{q'}, \blz{p},X)$. 
    Since $\blz{p}\subset\blo{q}$, 
        $\rho$ is also a path from $\blo{q'}$ to $\blo{q}$ in $\cV_1$.
    Thus, we obtain $q<q'$;
        but $p'<p$ implies $\idxfwd{}(p')=q'<q=\idxfwd{}(p)$, which gives a contradiction.
    Hence,~\ref{it:Np-minus-Npm-Bp-P0} is proved. 

    For the second statement it is clear that $\blo{q}\subset\pf_{\cV_1}(\blo{q},X)\subset N_p$.
    To see that $\blo{q}\cap N_{p-\abs{\QQ}}=\emptyset$ assume the contrary, that is, suppose there exists an $x\in \blo{q}\cap N_{p-\abs{\QQ}}$.
    Again, we have two cases;
        in the first one there exists a $p'\in\PP$ such that $p'\leq p-\abs{\QQ}$ and $x\in\pf_{\cV_0}(\blz{p'}, X)$.
    Therefore, there is a path $\rho\in\paths_{\cV_0}(\blz{p'}, \blo{q},X)$, 
        which implies $q<\idxfwd{}(p')$,
        but this contradicts the filtration consistent order assumption, 
        in particular $p'<p$ implies $\idxfwd{}(p')<\idxfwd{}(p)=q$.
    In the second case, there exists a $p'\in\PP$ such that $p'\leq p-\abs{\QQ}$, $q'\coloneqq\idxfwd{}(p')$ and 
        $x\in\pf_{\cV_1}(\blo{q'})$. 
    Therefore, $\paths_{\cV_1}(\blo{q'}, \blo{q},X)\neq\emptyset$, 
        and again we obtain $q<q'$, while the assumption $p'<p$ implies $\idxfwd{}(p')=q'<q=\idxfwd{}(p)$,
        which is a contradiction.
\end{proof}

In the following proofs we will also use the following remark which is an immediate corollary of Lemma~\ref{lem:subBD}.
\begin{remark}\label{rem:one-element-BD}
    Let $\BD=\{\bl\}$ be a one element block decomposition of $X$ then $\inv_\cV\bl=\inv_\cV X$.
\end{remark}

\begin{proof}[Proof of Proposition~\ref{prop:NpNpm-ipair-Mp}]
    We prove the statement by showing that $(N_p,N_{p-1})$ satisfies the conditions of Proposition~\ref{prop:invPE}.
    By Proposition~\ref{prop:push_forward_closed_vcomp}, 
        the set $N_{p}\setminus N_{p-1}$ is $\cV_0$-compatible as a difference of $\cV_0$-compatible sets; 
        it is also locally closed as a difference of closed sets.
    Hence, $N_{p}\setminus N_{p-1}$ is an isolating block by Proposition~\ref{prop:iso_block_is_lcl_Vcomp}. 
    Lemmas~\ref{lem:Np-minus-Npm-Bp}\ref{it:Np-minus-Npm-Bp-P0} and~\ref{lem:subBD} imply that $\{\blz{p}\}$ is a one-element block decomposition of $N_{p}\setminus N_{p-1}$.
    By Remark~\ref{rem:one-element-BD} we have $\inv_{\cV_0}N_{p}\setminus N_{p-1} = M_{p,0}$, 
        which finishes the proof.
\end{proof}

\begin{proof}[Proof of Proposition~\ref{prop:NpNpm-ipair-Mq}]
    Note that for the considered situation both $N_p$ and $N_{p-\abs{Q}}$ are obtained by applying the second rule of formula~\eqref{eq:Nsetssequence}.
    Therefore, by Proposition~\ref{prop:push_forward_closed_vcomp}, the set $P\setminus E$ is $\cV_1$-compatible and locally closed as a difference of $\cV_1$-compatible, closed sets.
    In particular, it is also $\cV_0$-compatible.
    Hence, it is an isolating block both in $\cV_0$ and~$\cV_1$. 

    To show that $(P,E)$ is an index pair for $M_{\QQ,0}$ in $\cV_0$, 
        note that by Theorem~\ref{thm:continuation-preimage-iota}\ref{it:continuation-preimage-iota-BD} 
        we know that $\BD_\QQ\coloneqq\{\blz{p}\mid p\in \QQ\}$ is a block decomposition of $\blo{q}$.
    By Lemma~\ref{lem:Np-minus-Npm-Bp}\ref{it:Np-minus-Npm-Bp-P0}
        $\bl_{p,0}\subset P\setminus E$ when $p\in\QQ$ and $\bl_{p,0}\cap P\setminus E=\emptyset$ when $p\not\in\QQ$.
    Therefore, Lemma~\ref{lem:subBD} implies that $\BD_\QQ$ is also a block decomposition of $P\setminus E$.
    Therefore, $\BD_\QQ$ is a block decomposition both for $P\setminus E$ and $\blo{q}$.
    Thus, we apply Lemma~\ref{lem:subBD} twice to obtain 
        $\inv_{\cV_0}P\setminus E = \inv_{\cV_0}\cset{\cV_0}(\BD_Q, \blo{q}) = \inv_{\cV_0}\blo{q}$.
    Hence, we get the thesis by Proposition~\ref{prop:invPE}.

    Now we focus on $M_{q,1}$.
    Lemmas~\ref{lem:Np-minus-Npm-Bp}\ref{it:Np-minus-Npm-PQ} and~\ref{lem:subBD} imply that $\{\blo{q}\}$ is a one-element block decomposition of $P\setminus E$ in $\cV_1$. 
    Therefore, by Remark~\ref{rem:one-element-BD} we have $\inv_{\cV_1}P\setminus E = M_{q,1}$, 
        and the result again follows by Proposition~\ref{prop:invPE}.
\end{proof}

\section{Persistence Modules and Gentle Algebras}
\label{sec:gentle-algebras-persistence}

In this section, we review some concepts from quiver representations and persistence modules.
In particular, we focus on the notion of a gentle algebra, a~well-known concept from quiver representation theory that we use to study the transition diagram.
Recall that, in this paper, all vector spaces are finite dimensional and defined over a fixed field~$\fk$.

\subsection{Quivers}\label{subsec:quivers}
In this section, we summarize the main definitions of quiver theory. 
For a more detailed introduction we refer to \cite{Simson_Skowroński_book_2007}.
A \emph{quiver}\index{quiver}, $\quiver$, consists of a~set of nodes, $\quiver_0$, a set of arrows, $\quiver_1$, and two maps $s,e \colon \quiver_1 \to \quiver_0$, such that $s$ and $e$  define the \emph{source}\index{source node} node and the \emph{target}\index{target node} node of each arrow, respectively.
We assume that $\quiver$ is a \emph{finite quiver}\index{quiver!finite}, which means that the sets $\quiver_0$ and $\quiver_1$ are finite.
Given $a,b \in \quiver_0$, a \emph{path}\index{path} with source $a$ and target $b$ is a sequence of arrows $\alpha_1 \ldots \alpha_n$ such that $s(\alpha_1) = a$, $e(\alpha_n) = b$ and $e(\alpha_i) = s(\alpha_{i+1})$ for $i \in\zintp{n-1}$.
Note that the convention for paths in quivers is different from paths in directed graphs: while the latter are defined on vertices, the former are defined on edges (see Section~\ref{subsec:prelim-graphs}).

The \emph{path algebra}\index{path algebra} of $\quiver$, denoted $\fk\quiver$, is a $\fk$-algebra whose underlying vector space has the set of all paths in $\quiver$ as a basis and whose product is
\begin{equation*}
    (\alpha_1 \ldots \alpha_n) \cdot (\beta_1 \ldots \beta_m) =   \alpha_1 \ldots \beta_m
\end{equation*}
if $e(\alpha_n) = s(\beta_1)$ or $0$ otherwise.
In order for $\fk\quiver$ to be well defined, we need to introduce trivial paths.
A \emph{trivial path} $\varepsilon_a$ for the vertex $a\in\quiver_0$ is defined as a path of length $0$ such that $s(\varepsilon_a) = e(\varepsilon_a) = a$.
Note that trivial paths are not elements of $\quiver_1$ but they are elements of $\fk\quiver$.
If $e(\alpha) = s(\beta) = a$, we define $\alpha \cdot \varepsilon_a = \alpha$ and $\varepsilon_a \cdot \beta = \beta$.

The product of basis elements is extended to any element of $\fk\quiver$ by distributivity.
We denote the ideal generated by all arrows in $\quiver_1$ by the arrow ideal $R$.
It can be decomposed as the following sum of $\fk$-vector spaces 
\[
    R = V_1 \oplus V_2 \oplus \ldots \oplus V_n \oplus \ldots
\]
where $V_n$ is the vector subspace of $\fk\quiver$ generated by (the sum) of paths of length~$n$.
In particular, there are no paths of zero length in $R$.
Given an ideal $I \subset \fk\quiver$, $(\quiver,I)$ is said to be a \emph{bound quiver}\index{quiver!bound} if there exists $m$ with $R^m \subseteq I \subseteq R^2$.
In other words, the ideal $I$ contains all paths which are long enough.

Lastly, a \emph{representation}\index{representation} $\perm{M}$ of a quiver associates to each $a \in \quiver_0$ a $\fk$-vector space, $\perm{M}(a)$, and to every arrow $\alpha$ from $q$ to $r$ a linear map $\perm{M}(\alpha) \colon \perm{M}(q) \to \perm{M}(r)$.
The \emph{evaluation}\index{evaluation} of $\perm{M}$ on the path $u = \alpha_1\ldots \alpha_n$ is given by the composition $\perm{M}(u) = \perm{M}(\alpha_n) \circ \ldots \circ \perm{M}(\alpha_1)$.
The concept of evaluation extends to any element of $\fk\quiver$ by linearity.
Given a bound quiver $(\quiver, I)$, a representation $\perm{M}$ is bound by $I$ if for any $u \in I$, $\perm{M}(u) = 0$.
We also define the direct sum of two representations $\perm{M}_1, \perm{M}_2$, as $\perm{M}_1 \oplus \perm{M}_2 \colon \quiver \to \vect$, where $(\perm{M}_1 \oplus\perm{M}_2)(a) = \perm{M}_1(a)\oplus\perm{M}_2(a)$ and $(\perm{M}_1 \oplus\perm{M}_2)(\alpha) = \perm{M}_1(\alpha)\oplus\perm{M}_2(\alpha)$.
We say a representation $\perm{M}$ is indecomposable if, whenever $\perm{M} \cong \perm{M}_1 \oplus\perm{M}_2$, we have that $\perm{M}_1 = 0$ or $\perm{M}_2 = 0$.

\subsection{Gentle algebras}\label{subsec:gentle-algebras}
Gentle algebras were introduced in the 80s as a generalization of the path algebras of some well-known quivers, usually denoted as $\mathbb{A}$ and $\tilde{\mathbb{A}}$, \cite{Assem1981, Assem1987}.
Since then, they have been applied to the study of different algebras, including cluster algebras~\cite{gentel_triangulations} and enveloping algebras of Lie algebras~\cite{Huerfano2006CATEGORIFICATIONOS}.
As we will see, they can be used to study the transition diagram.
We follow the definition appearing in \cite{Simson_Skowroński_book_2007}.
\begin{definition}\label{def:gentle_algebra}
    Let $\quiver$ be a quiver and $\fk\quiver$ its path algebra. 
    Then, $\fk\quiver/I$ is called a \emph{gentle algebra}\index{gentle algebra} if $(\quiver,I)$ is a bound quiver and has the following properties:
    \begin{enumerate}
        \item Each node of $\quiver$ is the source of at most two arrows and the target of at most two arrows.
        \item For each arrow $\alpha \in \quiver_1$, there is at most one arrow $\beta$ and one
        arrow $\gamma$ such that $\alpha\beta \notin I$ and $\gamma\alpha \notin I$.
        \item For each arrow $\alpha \in \quiver_1$, there is at most one arrow $\beta$ and one arrow $\gamma$ such that $\alpha\beta \in I$ and $\gamma\alpha \in I$.
        \item The ideal I is generated by paths of length two.
    \end{enumerate}
\end{definition} 

By abuse of terminology, we refer to the bound quiver $(Q, I)$ itself as a gentle algebra whenever the associated algebra $kQ/I$ is gentle.
Gentle algebras can be decomposed into the so-called strings and band modules~\cite{Butler1987AuslanderreitenSW}.
We introduce some necessary definitions before describing these modules.
Given $\alpha \in \quiver_1$, we define its \emph{formal reverse}, $\alpha^-$, such that $s(\alpha^-) = e(\alpha)$ and $e(\alpha^-) = s(\alpha)$.
We also set the operation $(\alpha^-)^- = \alpha$. 
Given a bound quiver $(\quiver,I)$, a \emph{string}\index{string} is a sequence $\alpha_1^{\ell_1} \ldots \alpha_n^{\ell_n}$ with $\ell_i \in \{-1, 1 \}$ such that 
$e(\alpha_i^{\ell_i}) = s(\alpha_{i+1}^{\ell_{i+1}})$ and $\alpha_i^{\ell_i} \neq {(\alpha_{i+1}^{\ell_{i+1}})}^-$ for $i\in\zintp{n-1}$, and satisfying that for any $1 \leq i < j \leq n$, neither the sequence $\alpha_i^{\ell_i} \ldots \alpha_j^{\ell_j}$ nor $\alpha_j^{-\ell_j}, \ldots, \alpha_i^{-\ell_i}$ is in $I$.
We also define two trivial strings $\{ \varepsilon_a, \varepsilon_a^-\}$ for each node $a$. 

In addition, we define $\stringset(\quiver,I)$ as the quotient set given by all possible strings modulo the relation $(\alpha_1^{\ell_1} \ldots \alpha_n^{\ell_n}) \sim (\alpha_n^{-\ell_n}, \ldots, \alpha_1^{-\ell_1})$.
Hence, any class $u \in \stringset(\quiver,I)$ is formed by a string and its formal reverse.
For the sake of simplicity we also write $u$ for the class $[u]$ in $S(\quiver, I)$.  
We say that $u \in \stringset(\quiver,I)$ passes through the arrow $\alpha$ if $\alpha$ appears in either of the two strings forming $u$.
Analogously, we say that $u$ passes through the node $a$ if it is the starting or ending node of an arrow for which $u$ passes through.

\begin{definition}
    Given $u \in \stringset(\quiver,I)$ define the \emph{string module}\index{string module}~$\stringmodule{u}$ as the representation
    \[
        \stringmodule{u}(a) = 
            \begin{cases*}
                k & if $u$ passes through $a$ \\
                0 & otherwise,
            \end{cases*}
    \]
    where $a \in \quiver_0$ and $\alpha \in \quiver_1$.
    Moreover, $\stringmodule{u}(\alpha)$ is the identity map between $\stringmodule{u}(s(\alpha))$ and $\stringmodule{u}(e(\alpha))$ if $u$ passes through $\alpha$, or the zero map otherwise.
\end{definition}

A band is a string with the same source and target.
Since bands are cyclic, their representations have a more complex behavior than the rest of strings.
However, as we will see in Section~\ref{sec:cm-barcodes}, bands do not appear in our setting;
thus, for the sake of simplicity, we omit further details on band modules.

The following theorem is a direct consequence of the description of all possible Auslander-Reiten sequences for gentle algebras given in \cite{Butler1987AuslanderreitenSW}.

\begin{theorem}\label{the:decomposition_quiver}
    The set of indecomposables for the representations of a gentle algebra $(\quiver,I)$ is formed by string modules and band modules. 
\end{theorem}

\subsection{Persistence modules}\label{sec:persistence}
We review here the most common notions in persistence theory.
For a more detailed introduction, the reader can refer to
\cite[Sections 3.4 and 4.3]{DW22} or \cite{chazal_structure_2016}.
Given a finite poset $\poset$,  we can consider the bound quiver $(\quiver, I)$ where $\quiver$ is the Hasse diagram of $\poset$ regarded as a quiver, and $I$ is the ideal of the path algebra of $\quiver$ generated by all commutative relations. 
With this notation, a \emph{persistence module}\index{persistence module} $\perm{M}$ over $\poset$ is just a representation of $\quiver$ bounded by $I$.
Note that if $\perm{M}$ is a persistence module, its evaluation over a path of $\quiver$ only depends on the source and target of the path.

A common example of a persistence module is a zigzag module, arising in topological data analysis \cite{carlsson_daSilva_zigzag_2010}.
A zigzag module is defined over a fence, that is a poset
\[
    a_1 \leftrightarrow a_2 \leftrightarrow \ldots \leftrightarrow a_m
\]
where every $\leftrightarrow$ is either $\leq$ or $\geq$.
Hence, the persistence module is a sequence of vector spaces and linear maps as follows,
\[
V_1 \overset{f_1}{\longleftrightarrow} V_2  \overset{f_2}{\longleftrightarrow} \ldots \overset{f_{m-1}}{\longleftrightarrow} V_m.
\]

\begin{example}\label{ex:zigzag}
    Consider the following zigzag module,
    \[\begin{tikzcd}
        V_1  & \arrow[l] V_2 & \arrow[l] V_3 \arrow[r] & V_4 & \arrow[l] V_5 .
    \end{tikzcd}\]
    Then, the underlying partial order is given by
    \[
        a_1 \; \leq \; a_2 \; \leq \; a_3 \; \geq \; a_4 \; \leq \; a_5,
    \]
    and the corresponding quiver is,
    \[\begin{tikzcd}
        a_1  & \arrow[l] a_2 & \arrow[l] a_3 \arrow[r] & a_4 & \arrow[l] a_5.
    \end{tikzcd}\]
    \qedex
\end{example}

The simplest type of a zigzag module is a zigzag interval.
Given a poset formed by just one fence with $m$ elements, and an interval $\zintab{l}{n} \subset \zintab{0}{m-1}$, the interval module $\mathbb{I}_{[l,n]}$ is defined as,
\begin{equation*}
    \perm{I}_{[l,n]}(a_i) = \begin{cases*}
        \fk \text{ if } i \in \zintab{l}{n} \\
        0 \text{ otherwise};
    \end{cases*}
    \qquad
\end{equation*}
and with identity maps between adjacent copies of $\fk$, and zero maps otherwise.
More generally, a zigzag module can be seen as the representation of a gentle algebra, and interval modules as its string modules.

\begin{example}\label{ex:interval}
    Consider the poset appearing in Example~\ref{ex:zigzag}.
    Then, the interval module $\perm{I}_{[2,4]}$ is
    \[\begin{tikzcd}
        0  & \arrow[l, "0", swap] k & \arrow[l, "\id", swap] k \arrow[r, "\id"] & k & \arrow[l] 0 .
    \end{tikzcd}\]
    \qedex
\end{example}

Interval modules are indecomposable and can be used to describe zigzag modules. 
This follows from a well-known result about the decomposability of quiver representations since the 70s~\cite{Gabriel1972}.
In the context of topological data analysis, zigzag modules were introduced in \cite{carlsson_daSilva_zigzag_2010}, together with a constructive proof of such a decomposition. 

\begin{theorem}\label{the:zigag_decomposition}
    Every persistence module can be uniquely decomposed, up to isomorphism, as a direct sum of indecomposables.
    Specifically, for zigzag modules these indecomposables are interval modules.
    \label{thm:decomposition}
\end{theorem}

In particular, given a zigzag module $\perm{M}$, there exists a multiset $S$ (i.e.~a set where multiple instances of the same element are allowed) of intervals such that $\perm{M} \cong \bigoplus_{I \in S} \perm{I}_I$.
The multiset $S$ is denoted as the barcode of $\perm{M}$.
In addition, string modules can also be seen as the interval modules of  specific zigzag modules, as explained in Section~\ref{sec:algorithm}.

\section{Conley-Morse Persistence Barcode} \label{sec:cm-barcodes}
    Let $\zzTD$ be a transition diagram for a zigzag filtration of block decompositions~$\zzBD$.
    Note that we have a natural arrangement of index pairs by columns corresponding to the stages of the construction.
    In particular, Proposition~\ref{prop:iptd_is_acyclic} tells us that there is a poset underlying every transition diagram, which we denote by $\poset$ from now on.
    By applying the homology functor to $\zzTD$, we obtain a persistence module $\perm{M}$ over $\poset$.
    It consists of Conley indices and linear maps induced by inclusions.
    We call it the \emph{Conley-Morse persistence module}\index{Conley-Morse persistence module}.    
    Before studying its decomposition, we introduce two crucial properties of $\zzTD$ and $\perm{M}$.
    
\begin{lemma}
    Let $\zzTD$ be a transition diagram.
    Consider two index pairs in the same column.
    If there is an arrow in the induced poset $\poset$ representing an inclusion $\ipair{p,\tdyn} \hookrightarrow \ipair{q,\tdyn}$ then the induced homomorphism in homology is $0$.
\end{lemma}
\begin{proof}
    By construction the vertical inclusions appear only within an AR-split diagram. 
    The splitting theorem (Theorem~\ref{thm:index_triple}) guarantees that the induced linear map equals $0$.
\end{proof}

\begin{remark}\label{rem:hasse}
    Since all vertical arrows in $\zzTD$ come from the concatenation of two inclusions in an AR-split diagram, there are no vertical arrows in the Hasse diagram of $\poset$.
    In addition, due to the construction of the transition diagram (see Section~\ref{subsec:basic-transition-diagram}),
    every node in the Hasse diagram is the source of at most one arrow and the target of at most one arrow coming from the preceding column.
    The same is true for the following column.
    See Figure~\ref{fig:node_example} for an illustration.
\end{remark}
 \begin{figure}[h]
     \[\begin{tikzcd}
        \ldots \arrow[r, hookleftarrow]
            & (P_{p,\tdyn}, E_{p,\tdyn}) \arrow[r, hook] & \ldots \\
        \ldots\arrow[ur, hook, swap]  & & \ldots\arrow[ul, hook']
     \end{tikzcd}\]
     \caption{Every node in the Hasse diagram of $\poset$ can have at most one arrow coming from, and another going to, a previous column; and one arrow coming from, and another going to, a later column.}
     \label{fig:node_example}
 \end{figure}
\begin{proposition}\label{prop:bdtd_gentle}
    Consider the poset $\poset$ induced by $\zzTD$, and the quiver $\quiver$ formed by the Hasse diagram of $\poset$.
    Consider the ideal $I$ generated by all 2-element paths $\beta\alpha$ forming AR-splits in $\zzTD$ (that is corresponding to $j_\ast\circ i_\ast$ in diagram~\eqref{eq:ip_index_triple}).
    Then, $(Q,I)$ is a gentle algebra.
\end{proposition}
\begin{proof}
    By Remark~\ref{rem:hasse} each point of $\quiver$ is the source and the target of at most two arrows.
    In addition, 
        by the structure of $\zzTD$ whenever $\beta\alpha$ is going back and forward (or forward and back) to the same column,
        it is a part of an AR-split diagram.
    Thus, every node is the ending of at most one arrow that goes back to a previous column, or forward to a later column, we have the following two properties:
    \begin{itemize}
        \item[-] For each arrow $\alpha \in \quiver_1$, there is at most one arrow $\beta$ and one arrow $\gamma$ such that $\alpha\beta \notin I$ and $\gamma\alpha \notin I$.
        \item[-] For each arrow $\alpha \in \quiver_1$, there is at most one arrow $\beta$ and one arrow $\gamma$ such that $\alpha\beta \in I$ and $\gamma\alpha \in I$.
    \end{itemize}
    Moreover, $(\quiver,I)$ is a bound quiver since $\quiver$ is finite and $I$ is generated by paths of length 2.
    Hence, $(\quiver,I)$ satisfies Definition~\ref{def:gentle_algebra} and is a gentle algebra.
\end{proof}

\begin{theorem}\label{thm:CMbarcode-decomposes}
    A Conley-Morse persistence module decomposes uniquely, up to isomorphism, into a direct sum of string modules.
\end{theorem}
\begin{proof}
    Let $\perm{M}$ be a Conley-Morse persistence module and $(Q,I)$ the gentle algebra from  Proposition~\ref{prop:bdtd_gentle}.
    Then, note that $\perm{M}(\alpha) \circ \perm{M}(\beta) = 0$ if $\beta\alpha$ comes from an AR-split, so $\perm{M}(u) = 0$ for any $u \in I$ and $\perm{M}$ is a representation of $Q$ bound by $I$.
    In particular, Theorem~\ref{the:decomposition_quiver} implies that its indecomposables are string and band modules.
    However, $\perm{M}$ cannot have band modules.
    To see this, note that having band modules implies that there should be at least one string, $u$, with the same source and target.
    In particular, $u$ must start and end in the same column.
    But then, there must be at least two consecutive arrows going back/forward to the same column contained in $u$, forming an AR-split.
    This contradicts $u$ being a path in $(Q,I)$, since it cannot contain any path contained in $I$.
    Hence, no string can be a band and all indecomposables in $\perm{M}$ must be string modules.
\end{proof}

\begin{remark}\label{rem:backward}
     A direct consequence of Proposition~\ref{prop:bdtd_gentle} and Theorem~\ref{thm:CMbarcode-decomposes} is that if~$\perm{S}_u$ is a non-null string submodule of a Conley-Morse persistence module, then $u$ does not contain any AR-split.
     In particular, a Conley-Morse persistence module decomposes into strings that cannot go backward and forward, i.e., once it passes a column it cannot return to it.
\end{remark}

\begin{definition}
    Let $\perm{M}$ be a Conley-Morse persistence module and $S$ the multiset of strings such that $\perm{M} \cong \bigoplus_{u \in S} \perm{S}_u$.
    Then $S$ is called the \emph{Conley-Morse persistence barcode}\index{Conley-Morse persistence barcode} of $\perm{M}$.
    \label{def:CMbarcode}
\end{definition}

\begin{example}
    We present in Figure~\ref{fig:main-example-CM-persistence-module} the Conley-Morse persistence module corresponding to the diagram in Figure~\ref{fig:transition-diagram-actual-blocks}.
    The corresponding Conley-Morse persistence barcode is represented by the bars.
    The green bar represents the generator of degree 0, the blue bar the generator of degree 1, and the orange one the generator of degree 2.
    \qedex
\end{example}

\begin{figure}[b]
    \includegraphics[width=0.9\textwidth]{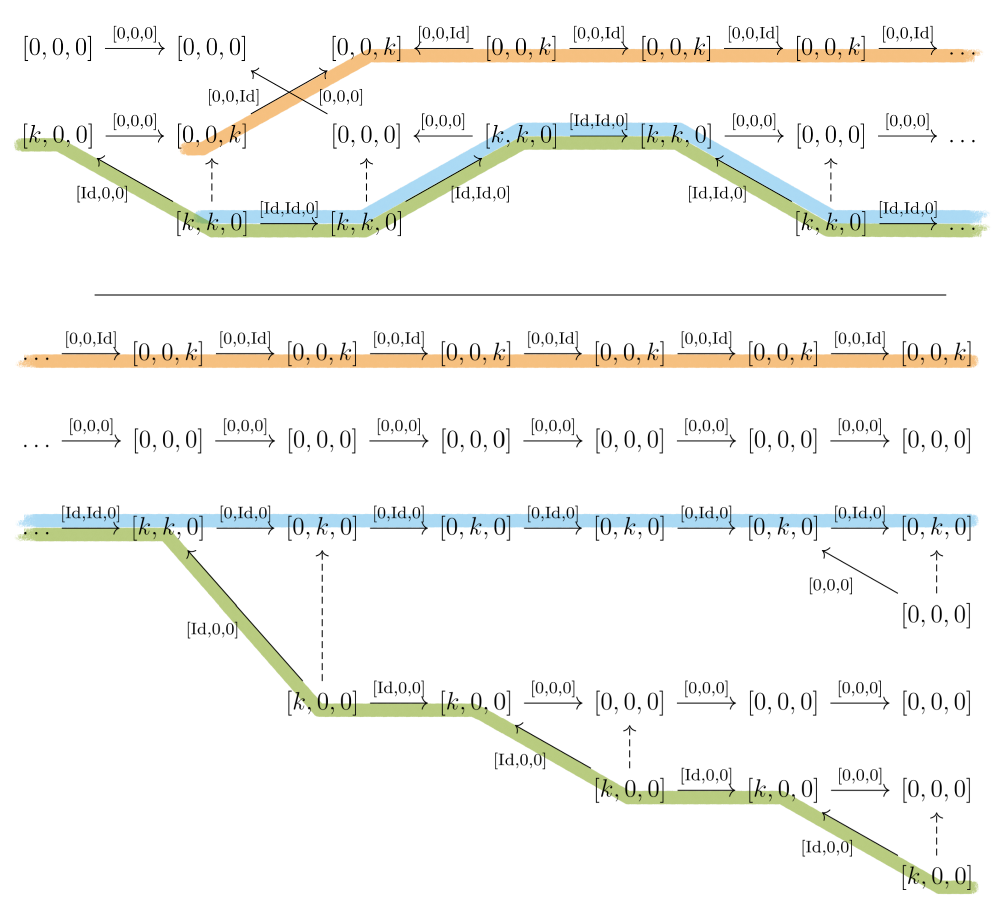}
    \caption{The Conley-Morse persistence module obtained from the diagram in Figure~\ref{fig:transition-diagram-actual-blocks}.
        The bars represent the string modules into which it decomposes.}
    \label{fig:main-example-CM-persistence-module}
\end{figure}

Let $\perm{M}$ be a Conley-Morse persistence module for the zigzag filtration of block decompositions $\zzBD\coloneqq\{(\BD_\tdyn, \cV_\tdyn)\}_{\tdyn\in\tInt}$, where $\tInt=\zint{\tdynmax}$.
Let $S$ be the corresponding multiset of strings of the corresponding Conley-Morse persistence barcode.
Recall that in general, the number of columns of the poset underlying $\perm{M}$ is greater than $\tdynmax+1$. 
However, analogously to the standard persistence algorithm, where every filtration is implicitly refined into a simplex-wise filtration,
we also treat the extra columns of $\perm{M}$ as an auxiliary maneuver facilitating the barcode computation.
Thus, we associate the lifespan of a string $u\in S$ with the corresponding values indexing the input sequence $\zzBD$. 
Note that Theorem~\ref{thm:CMbarcode-decomposes} assures that a string cannot ``go back'' with respect to the $\tdyn$ variable.
Therefore, the persistence of $u$ is the difference of the death index minus the birth index.
We can now revisit Theorem~\ref{thm:index_triple} and Corollary~\ref{cor:index_triple_consequences} and rephrase them in terms of persistence.

\begin{theorem}\label{thm:index-triple-consequences-persistence}
Let $\perm{M}$ be a Conley-Morse persistence module for the zigzag filtration of block decompositions $\zzBD\coloneqq\{(\BD_\tdyn, \cV_\tdyn)\}_{\tdyn\in\tInt}$, where $\tInt=\zint{\tdynmax}$,
    and let $S$ be the corresponding multiset of strings.
Then,
\begin{enumerate}[label=(${\alph*}$)]
    \item\label{it:index_triple_time consistency-2} 
        A string (interval) $u\in S$ cannot pass through the same time step $\tdyn\in\tInt$ twice,
        that is, it is spanned along the horizontal filtration.
    \item\label{it:index_triple_con_distribution-2} 
        There are two symmetric cases:
        \begin{itemize}
            \item If $\BD_{\tdyn}\inscr\BD_{\tdyn+1}$ (coarsening) 
                then every string $u\in S$ that is present at $\tdyn+1$ is also present in $\tdyn$, i.e., no bar is born through a coarsening.
            \item If $\BD_{\tdyn}\ovscr\BD_{\tdyn+1}$ (refinement) 
                then every string $u\in S$ that is present at $\tdyn$ is also present in $\tdyn+1$, i.e., no bar dies because of a refinement.
        \end{itemize}
    \item\label{it:index_triple_con_pairing-2}
        There are two symmetric cases:
        \begin{itemize}
            \item If $\BD_{\tdyn}\inscr\BD_{\tdyn+1}$ (coarsening) 
                then there is an even number of strings with the right endpoint at $\tdyn+1$, 
                    i.e., always an even number of bars dies through a coarsening.
            Moreover, each such string can be paired with another string of codimension~1.
            \item If $\BD_{\tdyn}\ovscr\BD_{\tdyn+1}$ (refinement), then there is an even number of strings with the left endpoint in column~$\tdyn$,
                i.e., always an even number of bars is born because of a refinement.
            Moreover, each such string can be paired with a string of codimension~1.
        \end{itemize}
\end{enumerate}
\end{theorem}
\begin{proof}
    As we will see, each of the properties \ref{it:index_triple_time consistency-2}, \ref{it:index_triple_con_distribution-2} and \ref{it:index_triple_con_pairing-2} is a consequence of properties \ref{it:0homomorphism}, \ref{it:index_triple_split} and \ref{it:index_triple_isomorphism} from Theorem~\ref{thm:index_triple}.
    In fact, property \ref{it:index_triple_time consistency-2} has already been proven since it is Remark~\ref{rem:backward} expressed with other words.
    
    We proceed now to prove \ref{it:index_triple_con_pairing-2}.
    Note that after applying the homology functor $H_d$ to $\zzTD$,
    the only linear functions that are not isomorphisms are the ones corresponding to the AR-splits,
    the remaining ones come from the connecting sequences, which are isomorphisms by Theorem~\ref{thm:ipair-inclusion-isomorphism}.
    Note that, in practice, by following the procedure for constructing the transition diagram in Section~\ref{sec:transition-diagram} the situation where a node participates in two AR-splits (like in Figure~\ref{fig:node_example}) does not happen, 
        instead we get two AR-splits connected by an isomorphism (like in Figure~\ref{fig:node_example_streched}).
    For simplicity, our proof adopts the second case 
        though it can be adjusted to the first, more general case.
    
     \begin{figure}[h]
         \[\begin{tikzcd}
            \ldots \arrow[r, hookleftarrow]
                & \ipairl{p,\tdyn}\arrow[r, hook] & \ipairr{p,\tdyn} \arrow[r, hook] & \ldots \\
            \ldots\arrow[ur, hook, swap]  & & & \ldots\arrow[ul, hook']
         \end{tikzcd}\]
         \caption{The provided construction of the transition diagram does not produce directly diagrams like in Figure~\ref{fig:node_example}.
         Instead, we have the above diagram.}
         \label{fig:node_example_streched}
     \end{figure}
    
    Therefore we can focus on the AR-split to prove the result.
    We prove the case $\BD_{\tdyn}\inscr\BD_{\tdyn+1}$ as the other is analogous.
    In this case, the AR-split diagram has the following structure:
    \begin{center}\begin{tikzcd}
        H_d(N_2,N_1) 
            \arrow[d, leftarrow, dashed, swap, shift right, "k_\ast^d=\,0"]
            \arrow[r, leftarrow, "j_\ast^d"]&
            H_d(N_2,N_{0})\\
            H_d(N_1,N_{0})
            \arrow[ru, rightarrow, "i_\ast^d"] 
            &
    \end{tikzcd}\end{center}
    We use $a,b,c$ to denote the nodes in the quiver corresponding to vector spaces $H_d(N_{1},N_{0})$, $H_d(N_2,N_{0})$ and $H_d(N_2,N_{1})$, respectively.
    We recall that, due to Theorem~\ref{thm:CMbarcode-decomposes}, there exists an isomorphism $\perm{M}_d \cong \bigoplus_{u \in S^d} \perm{S}_u$, where $\perm{M}_d$ is the persistence module corresponding to the homology of degree $d$ and $S^d$ its Conley-Morse persistence module.
    When we restrict this isomorphism to the AR-split under consideration, it takes the following shape:
    \[\begin{tikzcd}
        \underset{u \in S^d_{c}}{\bigoplus} \perm{S}_u(c)
            \arrow[d, leftarrow, dashed, swap, shift right, "k_\ast^d=\,0"]
            \arrow[r, leftarrow, "j_\ast^d"]&
            \underset{u \in S^d_{b}}{\bigoplus} \perm{S}_u(b)\\
            \underset{u \in S^d_{a}}{\bigoplus} \perm{S}_u(a)
            \arrow[ru, rightarrow, "i_\ast^d"] 
            &
    \end{tikzcd}\]
    Where $S^d_{a}, S^d_{b}, S^d_{c}$ are the strings of $S^d$ passing through the nodes $a, b$ and $c$.
    Additionally, we define $K^d \coloneqq \{ u \in S^d \mid u \text{ ends at } a \}$ and $C^d \coloneqq \{ u \in S^d \mid u \text{ ends at } c \}$.
    Then,
    \[
        \ker i_\ast^d \cong \underset{u \in K^d}{\bigoplus} \perm{S}_u(a), \qquad \coker j_\ast^d \cong  \underset{u \in C^d}{\bigoplus} \perm{S}_u(c) .
    \]
    Moreover, the isomorphism $h_\ast$ described in Theorem~\ref{thm:index_triple}\ref{it:index_triple_isomorphism} tells us that $\ker i_\ast^{d-1}$ and $\coker j_\ast^d$ have the same rank.
    In particular, the number of strings of $S^{d-1}$ ending at $a$ must be the same as the number of strings of $S^{d}$ ending at $c$,
    and it is possible to define a pairing between the two sets of strings (see Remark~\ref{rem:general-coupling}).
    Since this is true for all $d$ and all AR-split diagrams at that stage of the filtration, it proves \ref{it:index_triple_con_pairing-2} for $\BD_{\tdyn}\inscr\BD_{\tdyn+1}$.

    We proceed now to prove \ref{it:index_triple_con_distribution-2} for $\BD_{\tdyn}\inscr\BD_{\tdyn+1}$.
    As in the previous case, the strings that continue to $a$ and $c$ are respectively $S^d_{a} \setminus K^d$ and $S^d_{c} \setminus C^d$.
    Hence, 
    \begin{align*}
        &\underset{u \in S^d_{a}}{\bigoplus} \perm{S}_u(a)  \big/ \ker i_\ast^d \cong \im i_\ast^d \cong \underset{u \in S^d_{a} \setminus K^d }{\bigoplus} \perm{S}_u(a), \\
        &\im j_\ast^d \cong \underset{u \in S^d_{c}}{\bigoplus} \perm{S}_{u}(c)  \big/ \coker j_\ast^d \cong \underset{u \in S^d_{c} \setminus C^d }{\bigoplus} \perm{S}_u(c) .
    \end{align*}
    In addition, by Theorem~\ref{thm:index_triple}\ref{it:index_triple_split}, $H_d(N_2, N_0) \cong \im i_\ast^d \oplus \frac{H_d(N_2, N_0)}{\ker j_\ast^d}$.
    Using that $\frac{H_d(N_2, N_0)}{\ker j_\ast^d} \cong \im j_\ast^d$, we get 
    \[
         \underset{u \in S^d_{b}}{\bigoplus} \perm{S}_u(b) 
         \cong 
         {H_d(N_2, N_0)}
         \cong
         \im i_\ast^d \oplus \im j_\ast^d
         \cong
         \left( \underset{u \in S^d_{a} \setminus K^d }{\bigoplus} \perm{S}_u(a) \right) 
         \oplus 
         \left( \underset{u \in S^d_{c} \setminus C^d }{\bigoplus} \perm{S}_u(c) \right),
    \]
    which implies $S^d_{b} = ( S^d_{a} \setminus K^d ) \cup ( S^d_{c} \setminus C^d )$, 
    and no string can have the left endpoint at~$b$, proving 
        \ref{it:index_triple_con_distribution-2} for $\BD_{\tdyn}\inscr\BD_{\tdyn+1}$. 
\end{proof}

\begin{remark}
    \label{rem:general-coupling}
    The coupling that we mentioned---denoted in the diagrams using dashed lines in Figure~\ref{fig:CM-barcode-sphere_example} or~\ref{fig:bifurcation1d}---is given by the isomorphism $h_\ast$ defined in  Theorem~\ref{thm:index_triple}\ref{it:index_triple_isomorphism}.
    In these examples the coupling is well defined,
        however, it is not clear if the same can be said 
        when the Conley index consists of multiple generators of the same degree.
    We leave these considerations for future investigation.
\end{remark}

\section{Algorithm}\label{sec:algorithm}
The algorithm is stated in terms of zigzag modules (see Section~\ref{sec:persistence}).
Zigzag modules appear using the homology functor on a sequence of simplicial complexes $\{K_t\}_{t=0\ldots m}$ where
complexes indexed consecutively are related by inclusions, that is,
either $K_t\hookrightarrow K_{t+1}$ or $K_t\hookleftarrow K_{t+1}$ for $t\in \{0,\ldots, m-1\}$.
These sequences are known as \emph{zigzag filtrations}.
We write $K_t\leftrightarrow K_{t+1}$ to denote that the arrows could be either left or right
inclusions. Thus, a zigzag filtration is written as:
\begin{eqnarray}
{\mathcal F}: \emptyset=K_0\leftrightarrow K_1 \leftrightarrow \cdots\leftrightarrow K_{m-1}\leftrightarrow K_m
\end{eqnarray}
which provides a zigzag module by considering the homology groups $H_*(K_t)$ for each complex
$K_t$ and linear maps $\psi_t^*: H_*(K_t)\leftrightarrow H_*(K_{t+1})$ between the homology groups of consecutive complexes induced by inclusions:
\begin{eqnarray}
H{\mathcal F}: 0=H_*(K_0)\stackrel{\psi^*_0}{\leftrightarrow} H_*(K_1) \stackrel{\psi^*_1}{\leftrightarrow} \cdots\stackrel{\psi_{m-2}^*}{\leftrightarrow} H_*(K_{m-1})\stackrel{\psi_{m-1}^*}\leftrightarrow H_*(K_m)
\end{eqnarray}
In our case, we have zigzag filtrations of \emph{index pairs} arising out of the
 transition diagram $\zzTD$, where we have an index pair $(P_t,E_t):=(P_{it},E_{it})$
in place of a simplicial complex $K_t$. 
We use $t$ instead of $\tdyn$ for 
indexing the columns of the final transition diagram, which may have more columns than steps in the input zigzag filtration of block decompositions $\zzBD$.
We say that a pair $(P_t,E_t)$ is included 
in a pair $(P_{t'},E_{t'})$ if $P_t\subseteq P_{t'}$ and $E_t\subseteq E_{t'}$.
In a zigzag filtration for index pairs, consecutive pairs are related by inclusions,
that is, either $(P_t,E_t)\hookrightarrow (P_{t+1},E_{t+1})$ or $(P_t,E_t)\hookleftarrow (P_{t+1},E_{t+1})$ for $t\in \{0,\ldots, m-1\}$.
Using the notation $G_t=(P_t,E_t)$ and the relative homology group $H_*(G_t)\coloneqq H_*(P_t,E_t)$
in any fixed degree, say $\fk$, we get a zigzag persistence module out of a zigzag filtration of
index pairs:
\begin{eqnarray}
H{\mathcal F}: 0=H_k(G_0)\stackrel{\psi^*_0}{\leftrightarrow} H_k(G_1) \stackrel{\psi^*_1}{\leftrightarrow} \cdots\stackrel{\psi_{m-2}^*}{\leftrightarrow} H_k(G_{m-1})\stackrel{\psi_{m-1}^*}\leftrightarrow H_k(G_m)
\end{eqnarray}
In what follows, we refer to indices of the columns in the transition diagram $\zzTD$, as time.
Our algorithm processes $\zzTD$,
with increasing time values.
We denote the poset underlying the $\zzTD$, as $\poset$.
Let $\poset_t \subseteq\poset$ be the set of all points with time $t$.
In what follows, we say points $a,b\in \perm{P}$ are \emph{immediate} to one another
if $a$ and $b$ are adjacent in the Hasse diagram of $\perm{P}$.
For each point $a_{it}\in \poset_t$, the algorithm
inductively assumes that it has already implicitly processed all zigzag filtrations
of index-pairs 
supported on paths ending
at the index $a_{it}$ and extends the filtrations to each of the 
immediate points of $a_{it}$ at time $t+1$.

Notice that $a_{it}$ can have at most two immediate points at time $t+1$ (recall Figure~\ref{fig:node_example}) in which case 
each zigzag filtration ending at $a_{it}$ is extended to two different zigzag filtrations
ending at two different points at time $t+1$. On the other hand, two points $a_{it}$ and $a_{jt}$ 
at time $t$ can have a common immediate 
point at time $t+1$ in which case the zigzag filtrations ending at $a_{it}$ and $a_{jt}$ get extended
to the common immediate point at time $t+1$.

\subsection{Incremental zigzag persistence algorithm}\label{sec:zigzagalgo}
Our aim is to apply an incremental zigzag persistence algorithm~\cite{DW22} to extend the processing of
zigzag filtrations of index pairs from one time to the next time in the poset. This requires converting these input filtrations of index pairs
into filtrations of simplicial complexes. We achieve it by converting each pair of simplicial
complexes $(P_t,E_t)$ in a zigzag filtration of index pairs into a simplicial
complex $K_t$ by `coning' every simplex in $E_t\subseteq P_t$ with a dummy vertex. Precisely,
every simplex $\sigma=\{v_1,v_2,\ldots, v_s\}\in E_t$ is replaced with a coned simplex
$\omega\cdot\sigma=\{v_1,v_2,\ldots, v_s, \omega\}$ where $\omega$ is a dummy vertex.
It is known that the resulting simplicial zigzag filtration has the same barcode (multiset of intervals
in the interval decomposition of the zigzag persistence module) as the original
one except that an additional infinite bar appears in degree $0$.
Then, for $\dg=0$, we process every zigzag filtration in the same way as for $\dg>0$, but
at the end of computing all bars, we delete the extra infinite bar starting at $t=0$ that appears for every zigzag filtration due to the dummy vertex.
In the following, we assume that all persistence modules are induced by the index pair diagram in non-zero homology
degrees $\dg>0$.

 With the above conversion, we can assume that when the algorithm arrives at the
 point $a_{it}$ at time $t$, it has already computed the
 interval modules into which zigzag modules induced by all simplicial zigzag
 filtrations ending at $a_{it}$ decompose. The supports of these interval modules
 constitute the bars in the barcode of the module $\perm{M}'$ that is the restriction of the Conley-Morse persistence module $\perm{M}$ to the poset
 $\perm{P}'\subseteq \perm{P}$ that has been processed so far.

The reason that we can use the incremental zigzag algorithm from~\cite{DW22} is that the support of the interval modules (bars) ending at point $a_{it}$ traces
 backward uniquely due to Theorem~\ref{thm:index-triple-consequences-persistence} meaning that they neither split nor merge.
 This allows us to choose a zigzag filtration
 ending at $a_{it}$ whose interval modules in the decomposition
 can be extended to the rest of the module.
 
 {\bf Matrices.} 
 For simplicity, in the rest of the section, we assume that the coefficient field is $\fk=\ZZ_2$.
 The algorithm maintains a set of
 representative $\dg$-cycles for each of the bars ending at a point $a_{it}$ at time $t$.
 In particular, it has access to two matrices $Z_{it}$ and $C_{it}$ of $\dg$-cycles and
 $(\dg+1)$-chains respectively. Each column of $Z_{it}$ represents a $\dg$-cycle 
 in the cycle space ${\sf Z}(K_{it})$ of the complex $K_{it}$ at the point
 $a_{it}$.
 If $z=\sum_j \alpha_j \sigma_j$, $\alpha_i\in \ZZ_2$, is such a cycle,
 then the column representing $z$ contains $\alpha_j$ at the $j$th row. 
 These cycles collectively constitute a basis of the cycle space ${\sf Z}(K_{it})$.
  
 The cycle space ${\sf Z}(K_{it})$ contains the subspace of
 boundaries ${\sf B}(K_{it}) \subseteq {\sf Z}(K_{it})$. A subset of the columns
 of $Z_{it}$ constitute a basis of this boundary space. We denote the corresponding
 submatrix as $B_{it}\subseteq Z_{it}$ and assume the column partition
 so that $Z_{it}=[A_{it}|B_{it}]$. It follows from the fact that the
 quotient space ${\sf Z}(K_{it})/{\sf B}(K_{it})$ represents the homology
 group ${\sf H}_\dg(K_{it})$, all $\dg$-cycles given by the columns in
 the submatrix $A_{it}$ collectively provide representative cycles whose classes
 constitute a basis of ${\sf H}_\dg(K_{it})$.
 These representative cycles also are representative cycles
 for all bars ending at $a_{it}$ which constitute a basis of the homology space ${\sf H}_\dg(K_{it})$. 
 
 The matrix $C_{it}$, on the other hand, represents
 a basis of the subspace of the chain space ${\sf C}(K_{it})$ in degree $\dg+1$ whose boundaries
 constitute the basis in $B_{it}$. 
 We maintain the invariant that if $c$ is a column in
 $C_{it}$, then the boundary $\partial c$ is a column
 in $B_{it}$. In other words, $\partial C_{it}=B_{it}$.

\subsubsection{Computing the bars}\label{sec:computebar}
The bars for zigzag filtrations are incrementally computed as we move from time $t$ to time $t+1$
using the matrices described above. 

{\it Choice of zigzag filtration}: When we arrive at point $a_{it}$, we need to choose a zigzag filtration among the many that end at $a_{it}$. We choose
this zigzag filtration using the following procedure whose justification will become clear when we discuss the correctness of the algorithm. 
We move backward from $a_{it}$. Assume that we have already arrived at the point $a:=a_{*t'}$ for $t'\leq t$ in this backward walk. If there is a single point  $a':=a_{*(t'-1)}$ so that $a'$ and $a$ are immediate points, we simply move
to $a':=a_{*(t'-1)}$. Otherwise, there are exactly two points, say $b:=b_{*(t'-1)}$
and $c:=c_{*(t'-1)}$ where $b\rightarrow a$ and $c\leftarrow a$ are two immediate pairs of points in $\perm{P}$ (we go against the arrow from $a$ to $b$ and along the arrow
from $a$ to $c$). See Figure~\ref{fig:persistence-paths}. 
In this case, we move to $c$. 
Continuing backward
this way, we obtain a unique zigzag filtration, say ${\mathcal{ZZ}_{it}}$, ending in $a_{it}$. Implicitly
we apply the algorithm in~\cite{DW22} on $\mathcal{ZZ}_{it}$ which
updates the matrices at time $t$ to get the matrices at time
$t+1$. We have three cases:

\begin{figure}
    \centering
    \includegraphics[width=1.\linewidth]{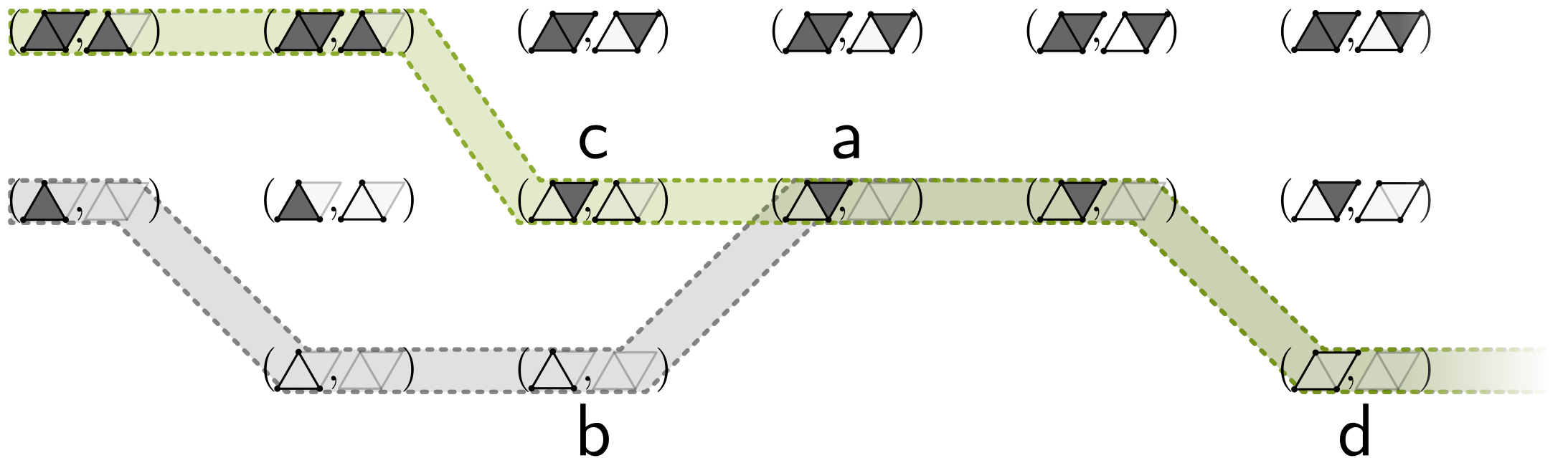}
    \caption{Choice of zigzag filtration (colored green): when the algorithm processes the point $d$, it chooses the zigzag
    filtration indicated green. While going backward, the path supporting the filtration faces a choice at $a$ between the points $b$ and $c$ where it chooses the point $c$.}
    \label{fig:persistence-paths}
\end{figure}

Case 1: Point $a_{it}$ has a single immediate point $a_{\ell(t+1)}$ at time $t+1$ and matrices at point $a_{\ell(t+1)}$ have not been computed yet. In this case, the zigzag algorithm as described in~\cite[Section 4.3]{DW22} is applied to extend the bars for
$\mathcal{ZZ}_{it}$ from $a_{it}$ to $a_{\ell(t+1)}$ with the
proper updates of the matrices.

Case 2: Point $a_{it}$ has a single immediate point $a_{\ell(t+1)}$ at time $t+1$ where the matrices have already been computed while proceeding from another point $a_{jt}$ that has also $a_{\ell(t+1)}$ as immediate point. In this case, these matrices
need to be updated further.
Observe that, proceeding from  $a_{it}$, the algorithm extends a set of bars (interval modules) for $\mathcal{ZZ}_{it}$ to
the point $a_{\ell(t+1)}$ which do not overlap with the set of bars that have already been extended by the algorithm while processing $\mathcal{ZZ}_{jt}$ at $a_{jt}$ according to Theorem~\ref{thm:index-triple-consequences-persistence}.
However, the new bars that are created while proceeding from $a_{it}$ may not be disjoint from the new bars that have already been computed while proceeding from $a_{jt}$, a case 
that needs to be reconciled via the updated matrices.

First, we observe that when we compute the extension from $a_{it}$, the bars that are continued
from $a_{jt}$ appear as new bars being born at $a_{\ell(t+1)}$. Let $\beta$ be a bar
that continues from $a_{it}$ to $a_{\ell(t+1)}$. The representative $\dg$-cycle computed
for $\beta$ at point $a_{\ell(t+1)}$ appears as a newly born $\dg$-cycle when seen from point $a_{jt}$.
Similarly, the representative cycle computed for a bar continued from point $a_{jt}$ appears
as a newly born $\dg$-cycle when seen from point $a_{it}$. Suppose that $Z_{i(t+1)}=[A_{i(t+1)} | B_{i(t+1)}]$ and
$Z_{j(t+1)}=[A_{j(t+1)} | B_{j(t+1)}]$ are the updated matrices of $Z_{it}$ and $Z_{jt}$ respectively at $a_{l(t+1)}$. 
Let $A_{i(t+1)}=[R_{i(t+1)}|S_{i(t+1)}]$ and $A_{j(t+1)}=[R_{j(t+1)}|S_{j(t+1)}]$ where $R_{i(t+1)}\subseteq A_{i(t+1)}$ and $R_{j(t+1)}\subseteq A_{j(t+1)}$ are the submatrices 
representing newly born $\dg$-cycles when proceeding from $a_{it}$ and $a_{jt}$ respectively.
By Theorem~\ref{thm:index-triple-consequences-persistence}(b)(coarsening case), no new bar is born at $a_{\ell(t+1)}$ for the module $\mathbb M$ induced by $\zzTD$. 
It follows that the space represented by the columns of $R_{i(t+1)}$ is equal to the
space represented by the columns of $S_{j(t+1)}$ and the space represented by the
columns of $R_{j(t+1)}$ is equal to the space represented by the columns of
$S_{i(t+1)}$. 
This observation allows us to construct a new matrix $Z_{\ell(t+1)}=[A_{\ell(t+1)}|B_{\ell(t+1)}]$ at point $a_{\ell(t+1)}$ where
\begin{equation*}
    A_{\ell(t+1)}=[S_{i(t+1)} | S_{j(t+1)}] \mbox{ and } B_{\ell(t+1)}= B_{i(t+1)}.
\end{equation*}
For the updated matrix $C_{\ell(t+1)}$, we simply take $C_{\ell(t+1)}\coloneqq C_{i(t+1)}$ because $B_{\ell(t+1)}=B_{i(t+1)}$ in the updated matrix $Z_{\ell(t+1)}$.

Case 3: Point $a_{it}$ has two immediate points $a_{j(t+1)}$ 
and $a_{\ell(t+1)}$ at time $t+1$. In this case, we proceed as in Case 1 or Case 2 as needed for each of $a_{j(t+1)}$ and $a_{\ell(t+1)}$ and we obtain the updated matrices
at these two points accordingly. 
It is worth noting that while processing the filtration $\mathcal{ZZ}_{it}$ for extension to $a_{j(t+1)}$, the bars that extend from $a_{it}$ to $a_{\ell(t+1)}$ appear as bars ending at $a_{it}$. Similarly, while extending the filtration $\mathcal{ZZ}_{it}$ to $a_{\ell(t+1)}$, the bars that extend from $a_{it}$ to $a_{j(t+1)}$ appear as bars ending at $a_{it}$. 
By Theorem~\ref{thm:index-triple-consequences-persistence}, the two sets of bars
do not overlap. Also, due to Theorem~\ref{thm:index-triple-consequences-persistence}\ref{it:index_triple_con_distribution-2} (refinement case), no bar actually
dies for the module $\mathbb M$ at $a_{it}$. 
Hence, we simply ignore the bars that appear to be ending at $a_{it}$ while
processing the filtration from $a_{it}$ to $a_{j(t+1)}$ and to $a_{\ell(t+1)}$.

\subsection{Correctness}
We argue that the bars computed by the zigzag algorithm described above compute
the barcode (Definition~\ref{def:CMbarcode}) of the Conley-Morse persistence module $\perm{M}$ given by the transition diagram $\zzTD$ with the underlying poset~$\mathbb{P}$. 
For this, we argue that the interval modules $\mathbb{I}_u$
supported on the paths $u\in S$ that the algorithm computes also decompose $\perm{M}$, that is, $\perm{M}\cong \bigoplus_{u\in S} \mathbb{I}_u$. Then, by the
uniqueness of indecomposables up to isomorphism (Theorem~\ref{thm:decomposition}), the set of intervals
$S=\{u\}$ forms the barcode of $\perm{M}$.

\begin{theorem}
The zigzag algorithm described in Section~\ref{sec:zigzagalgo} computes the barcode
of the input Conley-Morse persistence module $\perm{M}$ in $O(n^3m)$ time where~$n$ is the number of simplices in the simplicial complex that supports the input
dynamical system and $m$ is the number of points in the poset $\perm{P}$ indexing $\perm{M}$.
\end{theorem}
\begin{proof}
    According to the observation above, we need to show that the input module $\perm{M}$ decomposes into interval modules that the algorithm computes.
    The algorithm processes the points of $\perm{P}$ in 
    an order given by a linear extension of $\perm{P}$
    and let $a_0,\ldots, a_n$ denote this order where $\perm{P}_i\subseteq \perm{P}$ denote the subposet restricted to points $a_0,\ldots, a_i$. 
    Let $\perm{M}_i$ denote the restriction of $\perm{M}$ to the subposet $\mathbb{P}_i$.
    Assume inductively that the interval modules computed by the algorithm
    for $\perm{M}_i$ decompose it. We argue that after processing the filtration up to $a_{i+1}$, the module $\perm{M}_{i+1}$ still decomposes into the computed interval modules.

    Consider three cases as described in Section~\ref{sec:computebar}. In each case,
    one or more of the following occur for interval modules: 
        (i) an existing interval module at~$a_j$, $j\leq i$, extends to $a_{i+1}$, 
        (ii) an interval module ends at $a_j$ and thus does not extend to $a_{i+1}$, 
        (iii) an interval module is born at $a_{i+1}$. The algorithm based on the
    approach in~\cite{DW22} and the actions taken on matrices in all cases together make sure that the classes of representative cycles for the interval modules form a basis of the vector space $\perm{M}(a_{i+1})$. Also, the algorithm
    chooses representative cycles compatibly for the interval modules that are
    affected while moving to $a_{i+1}$. In particular, because of the special choice of the zigzag filtration $\mathcal{ZZ}_j$ while extending from $a_j$, $j\leq i$, to $a_{i+1}$, we can show the following: 
    \begin{proposition}
    Let $\perm{I}_u$
    be an interval module computed for $\perm{M}_{i+1}$; the homology classes $\{\perm{I}_u(a)\}_{a\in u}$ of the chosen representative cycles induce isomorphisms $\{\perm{I}_u(a)\rightarrow \perm{I}_u(b)\}$  for
    every $a\leq b$ in $u$.
    \label{prop:algo-correct}
    \end{proposition}
    It follows from the above proposition that the updated interval modules indeed decompose the module $\perm{M}_{i+1}$.

    To deduce the time complexity, observe that every matrix operation at each point $a_{it}$ takes time $O(n^3)$ if the matrices have dimensions $O(n)\times O(n)$. If the simplicial complex over which the input combinatorial dynamical system is defined has $n$ simplices, each matrix $Z_{it}$, $B_{it}$, and
    $C_{it}$ accessed and processed by the incremental zigzag algorithm indeed has dimensions $O(n)\times O(n)$. The other major step performed by the algorithm is the choice of the zigzag filtration by a backward walk for every point $a_{it}$. 
    This may take $O(m)$ time in the worst case. However, with a bookkeeping of the list of indices for the zigzag filtration chosen for a point and updating it in $O(1)$ time during processing that point, the total cost cannot exceed $O(m)$. This cost is dominated by the cost of matrix operations overall.
    The claimed time complexity follows because at each point $p_{it}\in \perm{P}$, the algorithm processes $O(1)$ matrices.
\end{proof}

{\it Proof of Proposition}~\ref{prop:algo-correct}: While moving from $a_i$ to
$a_{i+1}$, the algorithm extends certain zigzag filtrations $\mathcal{ZZ}_j$ for
some point $a_j\in \perm{P}_i$, $j\leq i$, to $a_{i+1}$. 
The interval modules $\perm{I}_u$
computed by the algorithm for $\perm{M}_{i+1}$ are of the following two types. In each case, we argue that the claim of the proposition holds.

(i) $\perm{I}_u$ is an interval module where the path $u$ is disjoint from the
supports of the zigzag filtrations $\mathcal{ZZ}_j$ that are extended.
In this case, $\perm{I}_u$ is not affected by the updates while moving from $a_i$ to $a_{i+1}$. By inductive hypothesis, $\perm{I}_u$ satisfies the claim.

(ii) $\perm{I}_u$ is an interval module where the path $u$ intersects the support of an extended filtration $\mathcal{ZZ}_j$.
There are two cases to be considered: (a) the path $u$ is completely contained in the support of the extended filtration $\mathcal{ZZ}_j$.
In this case, the matrix updates by the zigzag algorithm described in~\cite{DW22} implicitly update the representative cycles of $\perm{I}_u$ so that $\perm{I}_u$ satisfies the claim of the proposition. (b) the path $u$ intersects the
support of $\mathcal{ZZ}_j$ only partially. Let $a_k$ be the point where
the path $u$ deviates from the support of $\mathcal{ZZ}_j$ for the first
time while going backward from $a_j$. Let $a_{\ell}$ be the 
immediate point of $a_k$ going backward on the support of $\mathcal{ZZ}_j$. Then, by the choice of $\mathcal{ZZ}_j$, we have the backward inclusion $a_{\ell}\hookleftarrow a_k$. The interval module $\perm{I}_u$ restricted to the path starting at $a_k$ and going forward appears as a newly born interval module on the zigzag module induced by $\mathcal{ZZ}_j$. The algorithm in~\cite{DW22} does not change the representative cycles for such modules because the inclusion arrow
$a_{\ell}\hookleftarrow a_k$ is backward. This means that the representative cycles for $\perm{I}_u$ at points $a_k$ and backward (as computed for $\perm{M}_i$) remain intact. The rest of the representative cycles for $\perm{I}_u$ are computed
satisfying the compatibility condition during the extension of $\mathcal{ZZ}_j$ to $a_{i+1}$.

\section{Discussion}\label{sec:discussion}
The introduced Conley-Morse persistence barcode provides a new tool, rooted in persistent homology, for describing the evolution of a parameterized combinatorial multivector field.
It establishes a strong connection between dynamical systems---particularly continuation theory---and topological data analysis, opening possibilities for further exchanges of ideas that may enrich both fields.

For instance, the Conley-Morse persistence module is a naturally arising example of a persistence module over a poset that can be decomposed into string modules (bars), 
    making it a valuable case study for the rapidly growing field of multiparameter persistence.
Conversely, the interpretation of continuation from the viewpoint of persistence theory may enrich Conley index theory, as the Conley-Morse persistence barcode can be viewed as a \emph{parameterized Conley index}. 

This work also raises several open questions and unresolved hypotheses that are worth investigating as future directions:
\begin{itemize}
    \item\textbf{Bar coupling problem:}
        As pointed out in Remark~\ref{rem:general-coupling}, the clear coupling between Conley index generators in the AR-split diagram
        (see map $h_\ast^d$ in  Theorem~\ref{thm:index_triple}\ref{it:index_triple_isomorphism})
        does not easily generalize to a coupling of bars at the level of 
            the Conley-Morse persistence barcode,
            particularly when multiple Conley index generators of a single Morse set are born or die. 
        We hypothesize that a matching still exists, but we only provide a proof of quantitative matching 
            (Theorem~\ref{thm:index-triple-consequences-persistence}\ref{it:index_triple_con_pairing-2}).

    \begin{figure}
        \centering
        \includegraphics[width=0.45\linewidth]{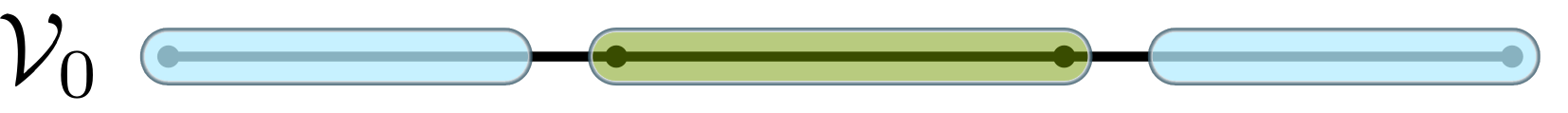}
        \vspace{0.5cm}
        
        \includegraphics[width=0.45\linewidth]{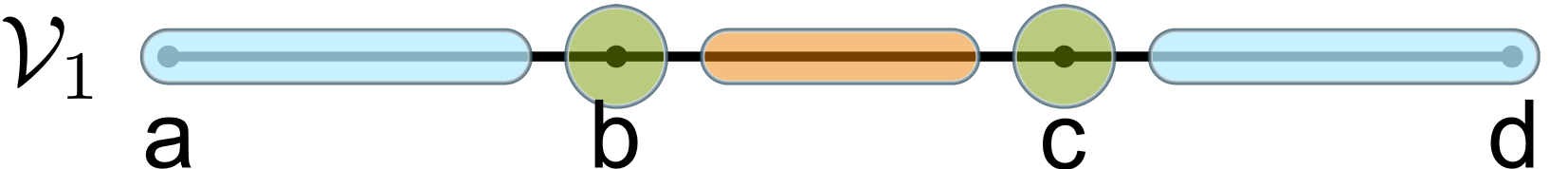}
        \caption{A minimalist combinatorial model for a pitchfork bifurcation.}
        \label{fig:pitchfork-mvf}
    \end{figure}
    \begin{figure}
        \centering
        \includegraphics[width=0.45\linewidth]{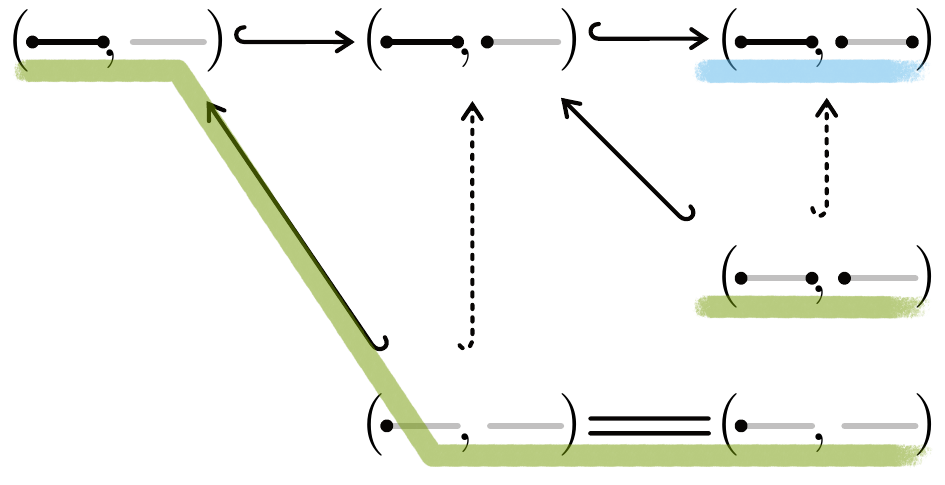}
        \hfill
        \includegraphics[width=0.45\linewidth]{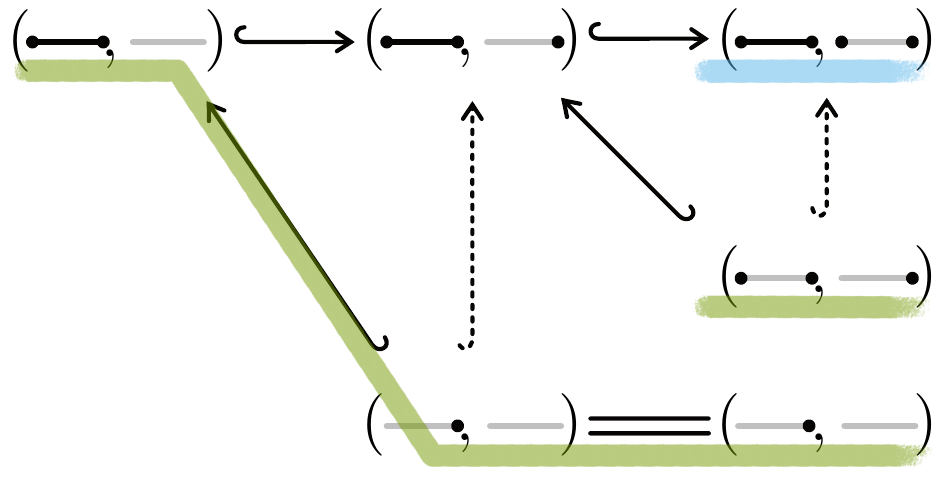}
        \caption{Two possible transition diagrams for the example in Figure~\ref{fig:pitchfork-mvf}.}
        \label{fig:pitchfork-CM-barcode}
    \end{figure}
    \begin{figure}
        \centering
        \includegraphics[width=0.4\linewidth]{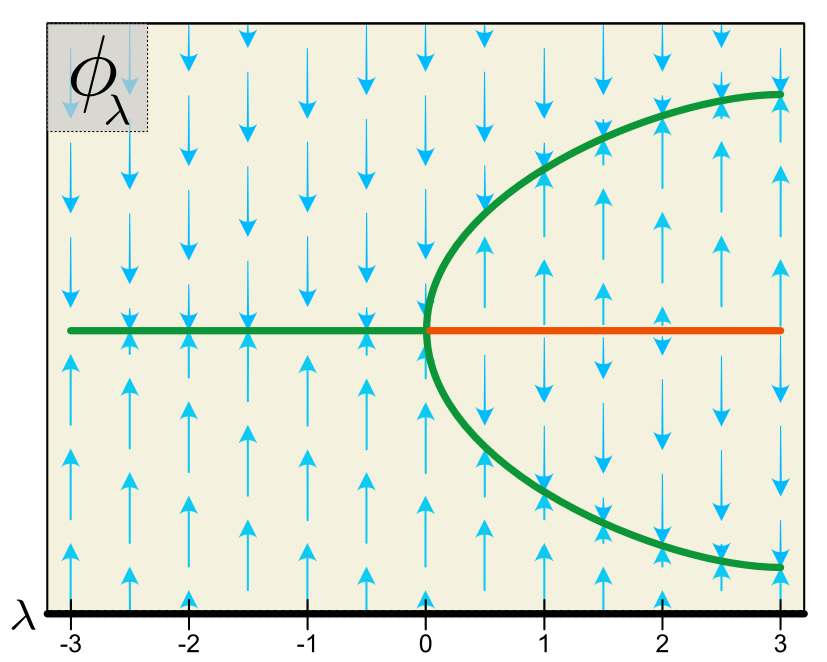}
        \includegraphics[width=0.4\linewidth]{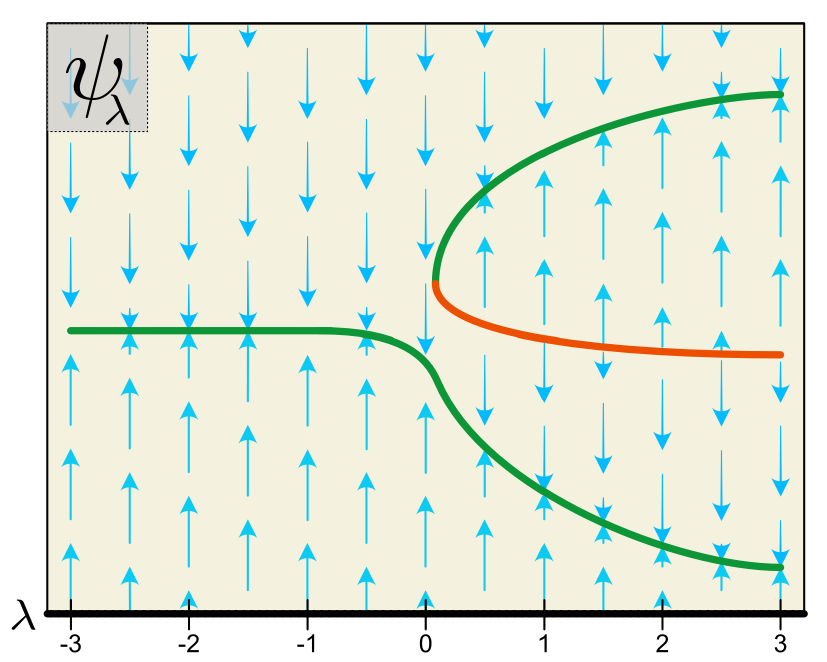}
        \caption{Pitchfork bifurcation (left) and a perturbation of the pitchfork bifurcation.}
        \label{fig:pitchfork-example}
    \end{figure}

    \item\textbf{Linear order sensitivity:} 
        The final form of the transition diagram, and thus, the Conley-Morse persistence module, 
            depends on the choice of AR-cascades (Section~\ref{subsec:TD-general-case}) and the filtration-consistent linear orders (Section~\ref{subsec:construction-of-the-trans-diag}).
        We conjecture that if all blocks are connected
            then Conley-Morse persistence modules are invariant, up to isomorphism, with respect
            to the filtration-consistent linear orders.
        However, it is not clear
            whether the same assertion holds with respect to the choice of
            the AR-cascades.
        It will be a subject for further study.

        Consider a minimalist model of a pitchfork bifurcation in Figure~\ref{fig:pitchfork-mvf} and a zigzag filtration of block decompositions: 
        $\zzBD\coloneqq\BD_0\ovscr\BD_1$, where 
        \begin{align*}
        \BD_0&\coloneqq\{\bl_{\bullet,0}\coloneqq\{b,c,bc\}\},\\
        \BD_1&\coloneqq\{\bl_{\beta,1}\coloneqq\{b\},
            \bl_{\gamma,1}\coloneqq\{c\},
            \bl_{\alpha,1}\coloneqq\{bc\}
            \}.
        \end{align*}
        We have two choices of AR-cascades to make the filtration basic:
        \begin{align*}
        \BD_{1'}&\coloneqq\{\bl_{\beta,1'}\coloneqq\{b, bc\},
            \bl_{\gamma,1'}\coloneqq\{c\}
            \} \quad\text{or}\\
        \BD_{1'}&\coloneqq\{\bl_{\beta,1'}\coloneqq\{b\},
            \bl_{\gamma,1'}\coloneqq\{c, bc\}
            \}.
        \end{align*}
        As shown in Figure~\ref{fig:pitchfork-CM-barcode}, the 0-degree generator corresponding to $\bl_{\bullet,0}$ continues to $\bl_{\beta, 1}$ in the first case, 
            while in the second, it continues to $\bl_{\gamma, 1}$.
        Nevertheless, the barcodes are isomorphic in terms of persistence.

        This ambiguity is not merely combinatorial.
        Note that in the continuous pitchfork bifurcation (Figure~\ref{fig:pitchfork-example} left) none of the newly created attractors inherits the generator either: 
            the attractor at $\lambda=0$ continues to the invariant interval consisting of all equilibria and trajectories connecting them.
        However, if we perturb the system (Figure~\ref{fig:pitchfork-example} right), we obtain uniqueness, 
            and such a perturbation precisely corresponds to the choice of an AR-cascade.
    
    \item\textbf{Extension to other settings:} 
        There is a natural question of extending the construction to other settings,
            for example, parameterized continuous flows, or discrete-time dynamical systems.
    \item\textbf{Improvement of the algorithm:} Currently we employ the incremental algorithm described in~\cite{DW22} to compute the Conley-Morse persistence barcode. 
        Can the fast zigzag algorithm presented in~\cite{DH22} be adapted to this setting to make the computation more practical?
\end{itemize}
            
\section*{Acknowledgment}
M.L. acknowledges support from the European Union’s Horizon 2020 research and innovation programme under the Marie Skło\-dow\-ska-Curie Grant Agreement No.~101034413. T.D. acknowledges the support of NSF funds CCF-2437030 and DMS-2301360.

The authors would like to thank the anonymous reviewers for their careful reading of the paper.
Their feedback significantly improved the quality of the article.
T.D. and M.L. would like to acknowledge many thought provoking discussions with Marian Mrozek on combinatorial dynamical systems and their continuations.
M.S.T. would like to thank Álvaro Sánchez for insightful discussions about representation theory.

\bibliography{bibliography}

@article{Alexandrov1937,
	author = {P.S. Alexandrov},
	journal = {Mathematiceskii Sbornik (N.S.)},
	pages = {501-518},
	title = {Diskrete {R}{\"a}ume},
	volume = {2},
	year = {1937}}

@inproceedings{DH22,
  author    = {Tamal K. Dey and
               Tao Hou},
  title     = {Fast Computation of Zigzag Persistence},
  booktitle = {30th Annual European Symposium on Algorithms, {ESA} 2022, September
               5-9, 2022, Berlin/Potsdam, Germany},
  series    = {LIPIcs},
  volume    = {244},
  pages     = {43:1--43:15},
  publisher = {Schloss Dagstuhl - Leibniz-Zentrum f{\"{u}}r Informatik},
  year      = {2022},
}

@article{Arai2009,
	author = {Zin Arai and William Kalies and Hiroshi Kokubu and Konstantin Mischaikow and Hiroe Oka and Pawel Pilarczyk},
	journal = {SIAM Journal on Applied Dynamical Systems},
	number = {3},
	pages = {757-789},
	title = {A Database Schema for the Analysis of Global Dynamics of Multiparameter Systems},
	volume = {8},
	year = {2009},
    doi = {10.1137/080734935}
}

@article{poset_decomposition,
	title = {Decomposition of persistence modules},
	volume = {148},
	doi = {10.1090/proc/14790},
	pages = {4581--4596},
	number = {11},
	journal = {Proceedings of the American Mathematical Society},
	author = {Botnan, Magnus and Crawley-Boevey, William},
	year = {2020}
}

@article{Brouillette:2024aa,
	author = {Brouillette, Guillaume and Allili, Madjid and Kaczynski, Tomasz},
	journal = {Journal of Applied and Computational Topology},
	number = {7},
	pages = {2155--2196},
	title = {Multiparameter discrete {M}orse theory},
	volume = {8},
	year = {2024}}

@inbook{Bubenik2024,
	address = {Cham},
	author = {Bubenik, Peter and Catanzaro, Michael J.},
	booktitle = {Toric Topology and Polyhedral Products},
	doi = {10.1007/978-3-031-57204-3_4},
	isbn = {978-3-031-57204-3},
	pages = {55--79},
	publisher = {Springer Nature Switzerland},
	title = {Multiparameter Persistent Homology via Generalized {M}orse Theory},
	url = {https://doi.org/10.1007/978-3-031-57204-3_4},
	year = {2024}}

@article{BushGameiroHarkerKokubuEtc:2012,
	author = {Bush, Justin and Gameiro, Marcio and Harker, Shaun and Kokubu, Hiroshi and Mischaikow, Konstantin and Obayashi, Ippei and Pilarczyk, Pawe{\l}},
	journal = {Chaos},
	month = {Dec},
	number = {4},
	pages = {047508},
	title = {Combinatorial-topological framework for the analysis of global dynamics.},
	volume = {22},
	year = {2012}}

@article{Cerf1970,
	author = {Cerf, Jean},
	journal = {Publications math{\'e}matiques de l'I.H.{\'E}.S.},
	pages = {5--173},
	title = {La stratification naturelle de functions diff{\'e}rentiables r{\'e}elles et le th{\'e}or{\`e}me de la pseudo-isotopie},
	volume = {39},
	year = {1970}}

@book{DW22,
    author  =   {Tamal K. Dey and Yusu Wang},
    title   =   {Computational Topology for Data Analysis},
    publisher= {Cambridge University Press},
    year    =   {2022}
}

@article{DowKalVan2023,
	author = {Alex K. Dowling and William D. Kalies and Robert C.A.M. Vandervorst},
	journal = {Journal of Differential Equations},
	pages = {124-198},
	title = {Continuation sheaves in dynamics: Sheaf cohomology and bifurcation},
	volume = {367},
	year = {2023}}

@article{DJKKLM2019,
	author = {Dey, Tamal K. and Juda, Mateusz and Kapela, Tomasz and Kubica, Jacek and Lipi\'{n}ski, Micha\l{} and Mrozek, Marian},
	journal = {SIAM Journal on Applied Dynamical Systems},
	number = {1},
	pages = {510-530},
	title = {Persistent Homology of {M}orse Decompositions in Combinatorial Dynamics},
	volume = {18},
	year = {2019}}

@article{DeLiMrSl2024,
	author = {Dey, Tamal K. and Lipi\'{n}ski, Micha\l{} and Mrozek, Marian and Slechta, Ryan},
	journal = {SIAM Journal on Applied Dynamical Systems},
	number = {1},
	pages = {81-97},
	title = {Computing Connection Matrices via Persistence-Like Reductions},
	volume = {23},
	year = {2024}}

@article{DeLiHa2026,
	author = {Dey, Tamal K. and Haas, Andrew and Lipi\'{n}ski, Micha\l{}},
	journal = {SIAM Journal on Applied Dynamical Systems},
	number = {1},
	pages = {108-130},
	title = {Computing a Connection Matrix and Persistence Efficiently from a {M}orse Decomposition},
	volume = {25},
	year = {2026}}

@inproceedings{DeLiMrSl2022,
	author = {Dey, Tamal K. and Lipi\'nski, Micha\l{} and Mrozek, Marian and Slechta, Ryan},
	booktitle = {38th Symposium on Computational Geometry},
	title = {{Tracking dynamical features via continuation and persistence}},
	year = {2022}}

@book{Conley1978,
	address = {Providence, R.I.},
	author = {Charles Conley},
	publisher = {American Mathematical Society},
	series = {CBMS Regional Conference Series in Mathematics},
	title = {Isolated Invariant Sets and the {M}orse Index},
	volume = {38},
	year = {1978}}

@article{DeyMrozekSlechta2022,
	author = {Dey,Tamal K. and Mrozek,Marian and Slechta,Ryan},
	journal = {SIAM Journal on Applied Dynamical Systems},
	number = {2},
	pages = {817-839},
	title = {Persistence of the {C}onley-{M}orse graph in combinatorial dynamical systems},
	volume = {21},
	year = {2022}}

@misc{DhaChaNat2025,
	author = {Amritendu Dhar and Apratim Chakraborty and Vijay Natarajan},
	eprint = {2507.00725},
	howpublished = {arXiv:2507.00725},
	primaryclass = {cs.GR},
	title = {Analyzing Time-Varying Scalar Fields using Piecewise-Linear {M}orse-{C}erf Theory},
	url = {https://arxiv.org/abs/2507.00725},
	year = {2025}
}

@book{Engelking1989,
	author = {Engelking, R.},
	publisher = {Heldermann Verlag, Berlin},
	title = {General Topology},
	year = {1989}
}

@article{Forman1998a,
	author = {Robin Forman},
	journal = {Advances in Mathematics},
	number = {1},
	pages = {90-145},
	title = {Morse Theory for Cell Complexes},
	volume = {134},
	year = {1998}}

@article{Forman1998b,
	author = {Forman, Robin},
	journal = {Mathematische Zeitschrift},
	number = {4},
	pages = {629--681},
	title = {Combinatorial vector fields and dynamical systems},
	volume = {228},
	year = {1998}}

@article{Franzosa1988,
	author = {Robert D. Franzosa},
	journal = {Transactions of the American Mathematical Society},
	number = {2},
	pages = {781--803},
	title = {The Continuation Theory for {M}orse Decompositions and Connection Matrices},
	volume = {310},
	year = {1988}}

@article{King2017,
	author = {Henry King and Kevin Knudson and Ne{\v z}a {Mramor Kosta}},
	journal = {Journal of Symbolic Computation},
	pages = {41-60},
	title = {Birth and death in discrete {M}orse theory},
	volume = {78},
	year = {2017}}

@article{Hotz2023,
	author = {Yan, Lin and Masood, Talha Bin and Rasheed, Farhan and Hotz, Ingrid and Wang, Bei},
	journal = {IEEE Transactions on Visualization \& Computer Graphics},
	month = aug,
	number = {08},
	pages = {3489-3506},
	title = {{ Geometry-Aware Merge Tree Comparisons for Time-Varying Data With Interleaving Distances }},
	volume = {29},
	year = {2023}}

@article{trophy2024,
	author = {Yan, Lin and Guo, Hanqi and Peterka, Thomas and Wang, Bei and Wang, Jiali},
	journal = {IEEE Transactions on Visualization and Computer Graphics},
	month = jan,
	number = {1},
	pages = {1249--1259},
	title = {TROPHY: A Topologically Robust Physics-Informed Tracking Framework for Tropical Cyclones},
	volume = {30},
	year = {2024}}

@article{LKMW2022,
	title = {Conley-{Morse}-{Forman} theory for generalized combinatorial multivector fields on finite topological spaces},
	volume = {7},
	issn = {2367-1726, 2367-1734},
	url = {https://link.springer.com/10.1007/s41468-022-00102-9},
	doi = {10.1007/s41468-022-00102-9},
	number = {2},
	urldate = {2024-01-21},
	journal = {Journal of Applied and Computational Topology},
	author = {Lipiński, Michał and Kubica, Jacek and Mrozek, Marian and Wanner, Thomas},
	year = {2023},
	pages = {139--184},
}

@article{LiMiMr2025,
	author = {Lipi{\'n}ski, Micha{\l} and Mischaikow, Konstantin and Mrozek, Marian},
	journal = {Qualitative Theory of Dynamical Systems},
	number = {1},
	pages = {5:1--5:33},
	title = {Morse Predecomposition of an Invariant Set},
	volume = {24},
	year = {2025}}

@incollection{MischMro_Conley_2002,
	title = {The {C}onley Index},
	volume = {2},
	series = {Handbook of Dynamical Systems},
	pages = {393--460},
	booktitle = {Handbook of Dynamical Systems},
	publisher = {Elsevier Science},
	author = {Mischaikow, Konstantin and Mrozek, Marian},
	editor = {Fiedler, Bernold},
	year = {2002},
	doi = {10.1016/S1874-575X(02)80030-3},
	note = {{ISSN}: 1874-575X},
}

@article{Mrozek2017,
	author = {Marian Mrozek},
	journal = {Foundations of Computational Mathematics},
	number = {6},
	pages = {1585--1633},
	title = {{C}onley--{M}orse--{F}orman Theory for Combinatorial Multivector Fields on {L}efschetz Complexes},
	volume = {17},
	year = {2017}}

@article{MrozekSrzednickiThorpeWanner2022,
	author = {Mrozek,Marian and Srzednicki,Roman and Thorpe,Justin and Wanner,Thomas},
	journal = {Communications in Nonlinear Science and Numerical Simulation},
	title = {Combinatorial vs. classical dynamics : recurrence},
	volume = {108},
        pages = {1--30},
	year = {2022}}

@article{MrozekWanner2021,
	author = {Marian Mrozek and Thomas Wanner},
	journal = {Journal of Differential Equations},
	pages = {375-434},
	title = {Creating semiflows on simplicial complexes from combinatorial vector fields},
	volume = {304},
	year = {2021}}

@book{MroWan2025,
	author = {Mrozek, Marian and Wanner, Thomas},
	edition = {1},
	month = {July},
	publisher = {Springer Cham},
	series = {SpringerBriefs in Mathematics},
	title = {Connection Matrices in Combinatorial Topological Dynamics},
	year = {2025},
    doi = {10.1007/978-3-031-87600-4}
}

@article{GuMuKh2022,
	author = {G{\"u}zel, {\.I}smail and Munch, Elizabeth and Khasawneh, Firas A.},
	journal = {Chaos: An Interdisciplinary Journal of Nonlinear Science},
	month = {09},
	number = {9},
	pages = {093111},
	title = {Detecting bifurcations in dynamical systems with CROCKER plots},
	volume = {32},
	year = {2022},
        doi = {10.1063/5.0102421}
}

@article{TyMuKh2020,
	author = {Tymochko, Sarah and Munch, Elizabeth and Khasawneh, Firas A.},
	journal = {Algorithms},
	number = {11},
	title = {Using Zigzag Persistent Homology to Detect Hopf Bifurcations in Dynamical Systems},
	volume = {13},
	year = {2020},
        pages = {1--16},
        doi = {10.3390/a13110278}}

@article{Woukeng2024,
	author = {Woukeng, Donald and Sadowski, Damian and Le{\'s}kiewicz, Jakub and Lipi{\'n}ski, Micha{\l} and Kapela, Tomasz},
	journal = {Journal of Applied and Computational Topology},
	number = {4},
	pages = {875--908},
	title = {Rigorous computation in dynamics based on topological methods for multivector fields},
	volume = {8},
	year = {2024}
}

@article{carlsson_daSilva_zigzag_2010,
	title = {Zigzag Persistence},
	volume = {10},
	issn = {1615-3375, 1615-3383},
	url = {http://link.springer.com/10.1007/s10208-010-9066-0},
	doi = {10.1007/s10208-010-9066-0},
	pages = {367--405},
	number = {4},
	journal = {Foundations of Computational Mathematics},
	shortjournal = {Found Comput Math},
	author = {Carlsson, Gunnar and De Silva, Vin},
	urldate = {2023-10-26},
	year = {2010},
	langid = {english},
}

@book{chazal_structure_2016,
	location = {Cham},
	title = {The Structure and Stability of Persistence Modules},
	isbn = {978-3-319-42543-6 978-3-319-42545-0},
	series = {{SpringerBriefs} in Mathematics},
	publisher = {Springer International Publishing},
	author = {Chazal, Frédéric and De Silva, Vin and Glisse, Marc and Oudot, Steve},
	urldate = {2024-08-29},
	year = {2016},
	doi = {10.1007/978-3-319-42545-0},
}

@article{Gabriel1972,
  title = {Unzerlegbare {D}arstellungen {I}},
  volume = {6},
  ISSN = {1432-1785},
  url = {http://dx.doi.org/10.1007/BF01298413},
  DOI = {10.1007/bf01298413},
  number = {1},
  journal = {Manuscripta Mathematica},
  publisher = {Springer Science and Business Media LLC},
  author = {Gabriel,  Peter},
  year = {1972},
  month = mar,
  pages = {71–103}
}

@book{Simson_Skowroński_book_2007, 
    place={Cambridge}, 
    series={London Mathematical Society Student Texts}, 
    title={Elements of the Representation Theory of Associative Algebras}, 
    publisher={Cambridge University Press}, 
    author={Simson, Daniel and Skowroński, Andrzej}, 
    year={2007}, 
    collection={London Mathematical Society Student Texts}
}

@article{Assem1981,
author = {Ibrahim Assem and Dieter Happel},
title = {Generalized tilted algebras of type $\mathbb{A}_n$},
journal = {Communications in Algebra},
volume = {9},
number = {20},
pages = {2101--2125},
year = {1981},
publisher = {Taylor \& Francis},
doi = {10.1080/00927878108822697},
}

@article{Assem1987,
  title = {Iterated tilted algebras of type $\tilde{A}_n$},
  volume = {195},
  ISSN = {1432-1823},
  url = {http://dx.doi.org/10.1007/BF01166463},
  DOI = {10.1007/bf01166463},
  number = {2},
  journal = {Mathematische Zeitschrift},
  publisher = {Springer Science and Business Media LLC},
  author = {Assem,  Ibrahim and Skowroński,  Andrzej},
  year = {1987},
  month = jun,
  pages = {269–290}
}

@article{gentel_triangulations,
author = {Ibrahim Assem and Thomas Br{\"u}stle and Gabrielle Charbonneau-Jodoin and Pierre-Guy Plamondon},
title = {{Gentle algebras arising from surface triangulations}},
volume = {4},
journal = {Algebra \& Number Theory},
number = {2},
publisher = {MSP},
pages = {201 -- 229},
year = {2010},
doi = {10.2140/ant.2010.4.201}
}

@article{Huerfano2006CATEGORIFICATIONOS,
  title={Categorification of some level two representations of quantum $\mathfrak{sl}_n$},
  author={Ruth Stella Huerfano and Mikhail Khovanov},
  journal={Journal of Knot Theory and Its Ramifications},
  year={2006},
  volume={15},
  pages={695-713},
  doi={10.1142/S0218216506004713}
}

@article{Butler1987AuslanderreitenSW,
  title={Auslander-{r}eiten sequences with few middle terms and applications to string algebras},
  author={M. C. R. Butler and Claus Michael Ringel},
  journal={Communications in Algebra},
  year={1987},
  volume={15},
  pages={145-179},
  doi= {10.1080/00927878708823416},
}
\bibliographystyle{amsplain}

\newpage
\appendix

\newpage
\section{Notation and Symbols}
\label{apdx:symbols}
\bgroup
\def\arraystretch{1.1}
\begin{longtable}[h]{c|c
    >{\raggedright\footnotesize\arraybackslash}p{6.5cm}
    >{\footnotesize}l}
\toprule
\textbf{Category} & \textbf{Notation} & \textbf{Description} & \textbf{Ref.} \\
\endfirsthead
\toprule
\midrule
\multirow{5}{*}{\makecell{Sets \&\\  Topology}} 
    & $\zintab{n}{m}$ & $\ZZ$ interval from $n$ to $m$ & Sec. \ref{subsec:prelim-sets} \\
    & $\inscr$ & inscribed relation & Sec. \ref{subsec:prelim-sets} \\
    & $\cl A$ & closure of set A& Sec. \ref{subsec:prelim-ftop} \\
    & $\opn A$ & opening of set A& Sec. \ref{subsec:prelim-ftop} \\
    & $\mo A$ & mouth of set A& Sec. \ref{subsec:prelim-ftop} \\
    & $\fk$ & a field & ... \\
    & $\poset$ & a poset & ... \\
\hline
\multirow{3}{*}{Graphs} 
    & $\pbeg\rho$ & the first element of a path $\rho$ & Sec. \ref{subsec:prelim-graphs} \\
    & $\pend\rho$ & the last element of a path $\rho$ & Sec. \ref{subsec:prelim-graphs} \\
    & $\rho\cdot\rho'$  & concatenation of paths $\rho$ and $\rho'$ & Sec. \ref{subsec:prelim-graphs} \\
\hline
\multirow{20}{*}{\makecell{Multivector\\ fields}} 
    & $\cV$             & a multivector field                       & Sec. \ref{subsec:mvf-elementary} \\
    & $[x]_\cV$         & multivector containing point $x$          & Sec. \ref{subsec:mvf-elementary} \\
    & $F_\cV:X\multimap X$ & multivalued map induced by $\cV$       & Sec. \ref{subsec:mvf-elementary} \\
    & $G_\cV$           & digraph induced by $\cV$                  & Sec. \ref{subsec:mvf-elementary} \\
    & $\inv_\cV$        & invariant part with respect to $\cV$      & Sec. \ref{subsec:mvf-elementary} \\
    & $\uimm\varphi$    & ultimate backward image of full solution $\varphi$ & Sec. \ref{subsec:morse-block-decomposition} \\
    & $\uimp\varphi$    & ultimate forward image of full solution $\varphi$ & Sec. \ref{subsec:morse-block-decomposition} \\
    & $(\cM, \PP)$      & Morse decomposition                       & Def. \ref{def:morse_decomposition} \\
    & $(\BD, \PP)$      & block decomposition                       & Def. \ref{def:block_decomposition} \\
    & $\indMDV{\BD}{\cV}$ & Morse decomposition induced by block decomposition $\BD$ in $\cV$ & Eq. \eqref{eq:induced-morse-decomposition} \\
    & $\con(S)$         & Conley index of $S$                       & Def. \ref{def:conley_index} \\
    & $\con_d(S)$       & degree $d$ component of the Conley index of $S$  & Def. \ref{def:conley_index} \\
    & $\pf_\cV(A)$      & push forward of a set $A$                 & Eq. \eqref{eq:push_forward} \\
    & $C_\cV(\cA, X)$   & connection set in $X$ for a family of sets $\cA$ & Eq. \eqref{eq:connection_set}\\
    & $\solV(A)$        & set of $\cV$-solutions in $A$             & Sec. \ref{subsec:mvf-elementary} \\
    & $\pathsV(A)$      & set of $\cV$-paths in $A$                 & Sec. \ref{subsec:mvf-elementary} \\
    & $\pathsV(x,y,A)$  & set of $\cV$-paths from $x$ to $y$ in $A$ & Sec. \ref{subsec:mvf-elementary} \\
    & $\isolV(A)$       & set of bi-infinite solutions in $A$       & Sec. \ref{subsec:mvf-elementary} \\
    & $\esolV(A)$       & set of essential solutions in $A$         & Sec. \ref{subsec:mvf-elementary} \\
    & $\smvf(X)$        & space of multivector fields on $X$        & Sec. \ref{subsec:combinatorial_continuation} \\
    & $\zzV$ & parameterized multivector field &  Sec. \ref{subsec:combinatorial_continuation}\\
\hline
\multirow{6}{*}{\makecell{Zigzag\\ filtration of\\ block\\ decompositions}} 
    & $\idxfwd{}$       & $\tdyn$-forward index map        & Sec. \ref{subsec:zigzag-filtration-BD} \\
    & $\idxbck{}$       & $\tdyn$-backward index map       & Sec. \ref{subsec:zigzag-filtration-BD} \\
    & $\zzBD$ & zigzag filtration of block decompositions & Sec. \ref{subsec:zigzag-filtration-BD} \\
    & $\zzTD$ & transition diagram of block decompositions & Sec. \ref{subsec:zigzag-filtration-BD} \\
    & $\ipairr{p,\tdyn}$ & a right-most index pair of a splitting diagram & Sec. \ref{subsec:basic-transition-diagram}\\
    & $\ipairl{p,\tdyn}$ & a left-most index pair of a splitting diagram & Sec. \ref{subsec:basic-transition-diagram}\\
\hline
\multirow{15}{*}{\makecell{Persistence \\ \& Gentle\\ algebras}} 
    & $\quiver$ & a quiver & Sec. \ref{subsec:quivers} \\
    & $\quiver_0$ & set of nodes in quiver $\quiver$ & Sec. \ref{subsec:quivers} \\
    & $\quiver_1$ & set of arrows in quiver $\quiver$ & Sec. \ref{subsec:quivers} \\
    & $s, e \colon \quiver_1 \to \quiver_0$ & source and target map & Sec. \ref{subsec:quivers} \\
    & $\perm{M}$ & persistence module & Sec. \ref{subsec:quivers} \\
    & $I$ & an ideal & Sec. \ref{subsec:quivers} \\
    & $R$ & the arrow ideal & Sec. \ref{subsec:quivers} \\
    & $(\quiver, I)$ & a bound quiver & Sec. \ref{subsec:quivers} \\
    & $\stringset(\quiver,I)$ & quotient set of strings modulo orientation & Sec.~\ref{subsec:gentle-algebras}\\
    & $\stringmodule{u}$ & string module over a string $u$ & Sec.~\ref{subsec:gentle-algebras}\\
    & $\mathcal{F}$ & zigzag filtration & Sec.~\ref{sec:persistence}\\
    & $\sf C$ & the chains space of complex $K$ & Sec.~\ref{sec:zigzagalgo}\\
    & $\sf Z$ & the cycles space of complex $K$ & Sec.~\ref{sec:zigzagalgo}\\
    & $\sf B$ & the boundaries space of complex $K$ & Sec.~\ref{sec:zigzagalgo}\\
\hline
\caption{Notation used across the paper.}
\label{tab:sample}
\end{longtable}
\egroup

\printindex

\end{document}